\documentclass[reqno,11pt]{amsart}
\usepackage{amsmath,amssymb,amsthm,graphicx,mathrsfs,bbm,url}
\usepackage{amsthm}
\usepackage{wrapfig}
\usepackage{enumitem}
\usepackage{mathtools}
\usepackage[utf8]{inputenc}
\usepackage[usenames,dvipsnames]{color}
\usepackage[colorlinks=true,linkcolor=Red,citecolor=Green]{hyperref}
\usepackage[super]{nth}
\usepackage[open, openlevel=2, depth=3, atend]{bookmark}
\hypersetup{pdfstartview=XYZ}
\usepackage[font=footnotesize]{caption}
\usepackage{a4wide}

\usepackage{epstopdf}
 
\usepackage{hyperref}

\newcommand{\C}{\mathbb{C}}
\newcommand{\R}{\mathbb{R}}

\newcommand{\Z}{\mathbb{Z}}

\newcommand{\M}{\mathcal{M}}
\newcommand{\N}{\mathbb{N}}

\newcommand{\V}{\mathbb{V}}
\newcommand{\X}{\mathbf{X}}

\newcommand{\HH}{\mathbb{H}}
\newcommand{\Ss}{\mathbb{S}}
\newcommand{\eps}{\varepsilon}

\newcommand{\mc}{\mathcal}

\newcommand{\End}{\mathrm{End}}
\newcommand{\VV}{\mathbf{\mathrm{v}}}

\DeclareMathOperator{\Tr}{\mathrm{Tr}}

\DeclareMathOperator{\dd}{d}
\DeclareMathOperator{\supp}{supp}
\DeclareMathOperator{\vol}{vol}

\DeclareMathOperator{\id}{Id}
\DeclareMathOperator{\E}{\mathcal{E}}
\DeclareMathOperator{\F}{\mathcal{F}}

\DeclareMathOperator{\ran}{ran}

\DeclareMathOperator{\divv}{div}

\DeclareMathOperator{\Lie}{\mathcal{L}}

\DeclareMathOperator{\e}{\mathbf{e}}
\DeclareMathOperator{\RR}{\mathbf{R}}

\DeclareMathOperator{\WW}{\mathbf{\mathrm{w}}}
\DeclareMathOperator{\EE}{\mathbf{\mathrm{e}}}
\DeclareMathOperator{\Deg}{\mathrm{deg}}

\DeclareMathOperator{\WF}{\mathrm{WF}}

\theoremstyle{plain}
\newtheorem{theorem}{Theorem}[section]

\newtheorem{lemma}[theorem]{Lemma}
\newtheorem{remark}[theorem]{Remark}
\newtheorem{proposition}[theorem]{Proposition}
\newtheorem{conj}[theorem]{Conjecture}

\newtheorem{corollary}[theorem]{Corollary}

\numberwithin{equation}{section}

\begin{document}

\begin{abstract}
In dimensions $\geq 3$, we prove that the X-ray transform of symmetric tensors of arbitrary degree is generically injective with respect to the metric on closed Anosov manifolds, and on manifolds with spherical strictly convex boundary, no conjugate points and a hyperbolic trapped set.  This has two immediate corollaries: local spectral rigidity, and local marked length spectrum rigidity (building on earlier work by Guillarmou, Knieper and the second author \cite{Guillarmou-Lefeuvre-18, Guillarmou-Knieper-Lefeuvre-19}), in a neighbourhood of a generic Anosov metric. In both cases, this is the first work going beyond the negatively curved assumption or dimension 2.

Our method, initiated in \cite{Cekic-Lefeuvre-20} and fully developed in the present paper, is based on a perturbative argument of the $0$-eigenvalue of elliptic operators via microlocal analysis which turn the analytic problem of injectivity into an algebraic problem of representation theory. When the manifold is equipped with a Hermitian vector bundle together with a unitary connection, we also show that the twisted X-ray transform of symmetric tensors (with values in that bundle) is generically injective with respect to the connection. This property turns out to be crucial when solving the \emph{holonomy inverse problem}, as studied in a subsequent article \cite{Cekic-Lefeuvre-21-1}.
\end{abstract}

\title{Generic injectivity of the X-ray transform}

\author[M. Ceki\'{c}]{Mihajlo Ceki\'{c}}
\date{\today}
\address{University of Zurich, Winterthurerstrasse 190, CH-8057 Zurich, Switzerland}
\email{mihajlo.cekic@math.uzh.ch}

\author[T. Lefeuvre]{Thibault Lefeuvre}

\address{Université de Paris and Sorbonne Université, CNRS, IMJ-PRG, F-75006 Paris, France.}

\email{tlefeuvre@imj-prg.fr}

\maketitle

\setcounter{tocdepth}{1}

\tableofcontents

\newpage

\section{Introduction}

Let $M$ be a smooth closed $n$-dimensional manifold, with $n \geq 2$. Let $\M$ be the cone of smooth metrics on $M$. Recall that a metric $g \in \M$ is said to be \emph{Anosov} if the geodesic flow $(\varphi_t)_{t \in \R}$ on its unit tangent bundle
\[
	SM := \left\{ (x,v) \in TM ~|~|v|_g = 1 \right\}
\]
is an Anosov flow (also called \emph{uniformly hyperbolic} in the literature), in the sense that is there exists a continuous flow-invariant splitting of the tangent bundle of $SM$ as:
\[
	T(SM) = \R X \oplus E_s \oplus E_u,
\]
where $X$ is the geodesic vector field, and such that:
\begin{equation}
\label{equation:anosov}
\begin{array}{l}
\forall t \geq 0, \forall w \in E_s, ~~ |\dd\varphi_t(w)| \leq Ce^{-t\lambda}|w|, \\
\forall t \leq 0, \forall w \in E_u, ~~ |\dd\varphi_t(w)| \leq Ce^{-|t|\lambda}|w|,
\end{array}
\end{equation}
the constants $C,\lambda > 0$ being uniform and the metric $|\bullet|$ arbitrary. We will denote by $\M_{\mathrm{Anosov}}$ the space of smooth Anosov metrics on $M$ and we will always assume in the following that it is not empty\footnote{Note that $\M_{\mathrm{Anosov}}(\Ss^2) = \emptyset$ (see \cite[Corollary 9.5]{Paternain-lecture-notes} for instance), that is to say not all manifolds can carry Anosov metrics. It is also not known if manifolds carrying Anosov metrics also carry negatively-curved metrics (the converse being obviously true).}.

Historical examples of Anosov metrics were provided by metrics of negative sectional curvature \cite{Anosov-67} but there are other examples as long as the metric carries ``enough" zones of negative curvature, see \cite{Eberlein-73,Donnay-Pugh-03}. As shown in \cite{Contreras-10}, generic metrics have a non-trivial hyperbolic basic set, i.e. a compact invariant set, not reduced to a single periodic orbit, where \eqref{equation:anosov} is satisfied (but this set may not be equal to the whole manifold though). Certain chaotic physical systems can also be described by Anosov Riemannian manifolds which are not globally negatively-curved: for instance, the \emph{Sinaï billiards} which arise as a model in physics for the Lorentz gas (a gas of electrons in a metal) can be naturally approximated by Anosov surfaces but these surfaces have a lot of flat areas (they consist of two copies of a flat tori connected by negatively-curved cylinders which play the role of the obstacles), see \cite[Chapter 6]{Kourganoff-15} for instance.

\subsection{Generic injectivity of the X-ray transform with respect to the metric: closed case} We let $\mc{C}$ be the set of free homotopy classes of loops on $M$. If $g \in \M_{\mathrm{Anosov}}$, it is known \cite{Klingenberg-74} that for all $c \in \mc{C}$, there exists a unique $g$-geodesic $\gamma_g(c) \in c$. We will denote by $L_g$ the \emph{marked length spectrum} of $g$, defined as the map:
\begin{equation}
\label{equation:mls}
	L_g \in \ell^\infty(\mc{C}), \quad L_g(c) := \ell_g(\gamma_g(c)),
\end{equation}
where $\ell_g(\gamma)$ denotes the Riemannian length of a curve $\gamma \subset M$ computed with respect to the metric $g$.

The closed curve $\gamma_g(c)$ on $M$ can be lifted to $SM$ to a periodic orbit $(\gamma_g(c), \dot{\gamma}_g(c))$ of $(\varphi_t)_{t \in \R}$, the geodesic flow of $g$. We then define the \emph{X-ray transform} as the operator:
\begin{equation}
\label{equation:xray}
	I^g : C^{\infty}(SM) \rightarrow \ell^\infty(\mc{C}), ~~~ I^gf(c) := \dfrac{1}{L_g(c)} \int_0^{L_g(c)} f(\varphi_t(x,v)) \dd t,
\end{equation}
where $(x,v)$ is an arbitrary point of the lift of $\gamma_g(c)$. Its kernel is given by \emph{coboundaries}, namely
\[
	\ker I^g|_{C^\infty(SM)} = \left\{ Xu ~|~ u \in C^\infty(SM), Xu \in C^\infty(SM)\right\}\footnote{Of course, the geodesic vector field $X$ depends on $g$. Note that when the context is clear, we try to avoid as much as possible the notation $X_g$ in order not to burden the discussion.}.
\]
The restriction of this operator to symmetric tensors appears in some rigidity questions in Riemannian geometry, as we shall see. We introduce $\pi_m^* : C^\infty(M,\otimes^m_S T^*M) \rightarrow C^\infty(SM)$, the natural pullback of symmetric $m$-tensors, defined by $\pi_m^*f(x,v) := f_x(\otimes^m v)$. We then set
\begin{equation}
\label{equation:xray-m}
	I^g_m := I^g \circ \pi_m^*.
\end{equation}
Any symmetric tensor $f \in C^\infty(M,\otimes^m_S T^*M)$ admits a canonical decomposition $f = Dp + h$, where $D$ is the \emph{symmetrized covariant derivative}, $p \in C^\infty(M,\otimes^{m-1}_S T^*M)$, $h \in C^\infty(M,\otimes^m_S T^*M)$ and $D^* h = 0$, see \S\ref{sssection:tensors-riemannian} for further details. The $Dp$ part is called \emph{potential} whereas $h$ is called \emph{solenoidal}. Using the fundamental relation $X \pi_m^* = \pi_{m+1}^* D$, we directly see that:
\[
	\left\{ Dp ~|~ p \in C^\infty(M,\otimes^{m-1}_S T^*M)\right\} \subset \ker I^g_m|_{C^\infty(M,\otimes^m_S T^*M)}.
\]
If in the place of inclusion we have equality, we say that the X-ray transform of symmetric $m$-tensors is \emph{s-injective} or \emph{solenoidally injective}, i.e. injective when restricted to solenoidal tensors. This is known to be true:
\begin{itemize}
\item for $m=0,1$ on all Anosov manifolds \cite{Dairbekov-Sharafutdinov-03},
\item for all $m \in \Z_{\geq 0}$ on Anosov manifolds with non-positive curvature \cite{Guillemin-Kazhdan-80, Croke-Sharafutdinov-98},
\item and for all $m \in \Z_{\geq 0}$ on all Anosov surfaces without any assumptions on the curvature by \cite{Guillarmou-17-1} (see also \cite{Paternain-Salo-Uhlmann-14-1} for the cases $m=0,1,2$).
\end{itemize}
Although the s-injectivity of $I^g_m$ is conjectured on Anosov manifolds of arbitrary dimension, it is still a widely open question. The main theorem of this article is a first step in this direction:

\begin{theorem}
\label{theorem:genericity-metric}
There exists an integer $k_0 \gg 1$ such that the following holds. Let $M$ be a smooth closed manifold of dimension $\geq 3$ carrying Anosov metrics. For all $m \in \Z_{\geq 0}$,\footnote{For $m=0,1$, the s-injectivity is already established \cite{Dairbekov-Sharafutdinov-03}.} there exists an open dense set $\mc{R}_m \subset \M_{\mathrm{Anosov}}$ (for the $C^{k_0}$-topology) such that for all metrics $g \in \mc{R}_m$, the X-ray transform $I^g_m$ is s-injective. In particular, the space of metrics $\mc{R} := \cap_{m \geq 0} \mc{R}_m$ whose $X$-ray transforms are $s$-injective for all $m \in \mathbb{Z}_{\geq 0}$ is \emph{residual} in $\M_{\mathrm{Anosov}}$.
\end{theorem}

The set $\mc{R}_m \subset \M_{\mathrm{Anosov}}$ is open and dense for the $C^{k_0}$-topology in the sense that:
\begin{itemize}
\item Openness: for all $g \in \mc{R}_m$, there exists $\eps > 0$ such that for all smooth metrics $g'$ with $\|g'-g\|_{C^{k_0}} < \eps$, $g' \in \mc{R}_m$,
\item Density: if $g \in \M_{\mathrm{Anosov}}$, then for all $\eps > 0$, there exists a smooth metric $g' \in \mc{R}_m$ such that $\|g-g'\|_{C^{k_0}} < \eps$.
\end{itemize}
Note that $\mc{R} \subset \M_{\mathrm{Anosov}}$ is a countable intersection of open and dense sets, and so in particular it is dense in the $C^\infty$ topology. Observe that the sets $\mc{R}_m$ and $\mc{R}$ are invariant by the action (by pullback of metrics) of the group of diffeomorphisms that are isotopic to the identity, which we denote by $\mathrm{Diff}_0(M)$.

As we shall see below, the generic s-injectivity of $I_m^g$ is equivalent to the s-injectivity of an elliptic pseudodifferential operator $\Pi_m^g$ introduced in \cite{Guillarmou-17-1}, called the \emph{generalized X-ray transform}, which enjoys very good analytic properties. This operator will also naturally appear below when discussing the \emph{twisted case}, i.e. when including a bundle $\E \rightarrow M$ in the discussion, see \S\ref{ssection:twisted-case}. In particular, this reduction to an elliptic $\Psi$DO will allow us to apply our technique of perturbation of the $0$-eigenvalue of elliptic operators, see \S\ref{ssection:strategy} for further details on the strategy of proof.

\subsection{Application to rigidity problems} We now detail the consequences of Theorem \ref{theorem:genericity-metric} on three problems of rigidity.

\subsubsection{The marked length spectrum rigidity conjecture}

In the following, an \emph{isometry class}, denoted by $\mathfrak{g}$, is defined as an orbit of metrics under the action of $\mathrm{Diff}_0(M)$, namely
\[
\mathfrak{g} := \left\{ \phi^*g ~|~ \phi \in \mathrm{Diff}_0(M)\right\}.
\]
If $M$ is closed, we let $\mathbb{M}_{\mathrm{Anosov}} := \M_{\mathrm{Anosov}}/\mathrm{Diff}_0(M)$ be the moduli space of smooth Anosov metrics modulo the action of $\mathrm{Diff}_0(M)$. The marked length spectrum introduced in \eqref{equation:mls} is invariant by the action of $\mathrm{Diff}_0(M)$ and thus descends as a map
\begin{equation}
\label{equation:mls-moduli}
L : \mathbb{M}_{\mathrm{Anosov}} \rightarrow \ell^\infty(\mc{C}), ~~ \mathfrak{g} \mapsto L_{\mathfrak{g}}.
\end{equation}
It is believed to parametrize entirely the moduli space of isometry classes.

\begin{conj}
\label{conjecture:bk}
Let $M$ be a smooth $n$-dimensional closed manifold such that $\M_{\mathrm{Anosov}}(M) \neq \emptyset$. Then the marked length spectrum map $L$ in \eqref{equation:mls-moduli} is injective.
\end{conj}

Originally, the conjecture was only phrased in the context of negatively-curved manifolds by Burns-Katok \cite{Burns-Katok-85} but it is believed to hold in the general Anosov case. Despite some partial results \cite{Guillemin-Kazhdan-80, Katok-88, Croke-Fathi-Feldman-92, Besson-Courtois-Gallot-95, Hamenstadt-99, Croke-Sharafutdinov-98, Paternain-Salo-Uhlmann-14-1} and the proof of the conjecture in the two-dimensional case for negatively-curved metrics \cite{Croke-90, Otal-90}, this question is still widely open. Recently, Guillarmou, Knieper and the second author proved that the s-injectivity of $I_2^{g_0}$ implies that the conjecture holds true locally around $g_0$ (see \cite{Guillarmou-Lefeuvre-18} and \cite[Theorem 1.2]{Guillarmou-Knieper-Lefeuvre-19}). In particular, by \cite{Croke-Sharafutdinov-98}, this solves locally the conjecture around an Anosov metric with nonpositive curvature in any dimension (and without any assumptions on the curvature in dimension two by \cite{Paternain-Salo-Uhlmann-14-1,Guillarmou-17-1}). A similar conjecture exists for the billiard flow of convex domains, see \cite{DeSimoi-Kaloshin-Wei-17} for the most recent developments. A straightforward consequence of Theorem \ref{theorem:genericity-metric}, combined with \cite[Theorem 1.3]{Guillarmou-Knieper-Lefeuvre-19} (and the remark following \cite[Theorem 1.2]{Guillarmou-Knieper-Lefeuvre-19}), is therefore the following:

\begin{corollary}[of Theorem \ref{theorem:genericity-metric} and \cite{Guillarmou-Lefeuvre-18,Guillarmou-Knieper-Lefeuvre-19}]
\label{corollary:mls}
There exists $k_0 \gg 1$ such that the following holds. Let $M$ be a smooth $n$-dimensional closed manifold carrying Anosov metrics. There is an open and dense set $\mathbbm{R}_2 \subset  \mathbb{M}_{\mathrm{Anosov}}$ (for the $C^{k_0}$-topology) such that: for all $\mathfrak{g}_0 \in \mathbbm{R}_2$, the marked length spectrum map $L$ in \eqref{equation:mls-moduli} is locally injective near $\mathfrak{g}_0$.
\end{corollary}

The set $\mathbbm{R}_2$ is equal to $\mc{R}_2/\mathrm{Diff}_0(M)$, where $\mc{R}_2$ is given by Theorem \ref{theorem:genericity-metric} (and this is well-defined since $\mc{R}_2$ is invariant by $\mathrm{Diff}_0(M)$). By locally injective, we mean the following: for any $g_0 \in \mathfrak{g}_0$, there exists $\eps_0 > 0$ such that the following holds: if $g_1, g_2 \in \mc{M}_{\mathrm{Anosov}}$ are such that there exist $\phi_1,\phi_2 \in \mathrm{Diff}_0(M)$ such that $\|\phi_1^*g_1-g_0\|_{C^{k_0}} + \|\phi^*_2g_2-g_0\|_{C^{k_0}}< \eps_0$ and $L_{g_1} = L_{g_2}$, then $g_1$ and $g_2$ are isometric. Except in dimension two, this is the first result allowing to relax the negative curvature assumption.

\subsubsection{Spectral rigidity} Since the celebrated paper of Kac \cite{Kac-66} ``\emph{Can one hear the shape of a drum?}", investigating the space of \emph{isospectral} manifolds (i.e. manifolds with same spectrum for the Laplacian $\Delta_g$ on functions, counted with multiplicities) has been an important question in spectral geometry, see \cite{Milnor-64, Guillemin-Kazhdan-80, Guillemin-Kazhdan-80-2, Vigneras-80, Sarnak-90, Gordon-Webb-Wolpert-92} for instance. It is known that there exist pairs of isospectral hyperbolic surfaces that are not isometric \cite{Vigneras-80}. On the other hand, by \cite{Guillemin-Kazhdan-80}, the s-injectivity of $I_2^{g_0}$ implies that $(M,g_0)$ is \emph{spectrally rigid} in the following sense: if $(g_s)_{s \in (-1,1)}$ is a smooth family of isospectral metrics, then they are isometric, i.e. there exists $(\phi_s)_{s \in (-1,1)}$ such that $g_0 = \phi_s^* g_s$. As a consequence, we obtain the following:

\begin{corollary}[of Theorem \ref{theorem:genericity-metric}]
Let $M$ be a $n$-dimensional closed manifold carrying Anosov metrics. Then, the open and dense set of isometry classes $\mathbbm{R}_2 \subset \mathbb{M}_{\mathrm{Anosov}}$ are spectrally rigid.
\end{corollary}

Once again, we conjecture that the previous corollary should actually hold for \emph{all} Anosov metrics in any dimension.

\subsection{Generic injectivity of the X-ray transform with respect to the connection}

\label{ssection:twisted-case}

We now consider a smooth closed Anosov Riemannian manifold $(M,g)$ and a smooth Hermitian vector bundle $\E \rightarrow M$. We let $\mc{A}_{\E}$ be the space of smooth unitary connections on the bundle $\E$. Contrary to the untwisted case (i.e. $\E = \C \times M$), \eqref{equation:xray} might not define a canonical notion of integration of sections along closed geodesics\footnote{Actually, \eqref{equation:xray} defines an interesting notion if the bundle $\pi^*\E \rightarrow SM$ is transparent, i.e. the holonomy with respect to the connection $\pi^*\nabla^{\E}$ along closed geodesics is trivial, see \cite[Section 7.2]{Cekic-Lefeuvre-21-1} for a discussion.}. It is therefore more convenient to define a similar notion via microlocal analysis.

If $\nabla^{\E} \in \mc{A}_{\E}$ and $\pi : SM \rightarrow M$ denotes the projection, we can consider the pullback bundle $\pi^*\E$ equipped with the pullback connection $\pi^*\nabla^{\E}$ and define the operator $\X := (\pi^*\nabla^{\E})_X$ acting on $C^\infty(SM,\pi^*\E)$. We then consider the meromorphic extension of the resolvent operators $\RR_{\pm}(z) := (\mp\X-z)^{-1} : C^\infty(SM,\pi^*\E) \rightarrow \mc{D}'(SM,\pi^*\E)$ to the whole complex plane $\C$ (here $\mc{D}'$ denotes the space of distributions), see \S\ref{ssection:ruelle} for further details on the Pollicott-Ruelle theory. It is known that there is an open and dense set of connections without resonances at $0$ (see \cite{Cekic-Lefeuvre-20}). When this is the case, we can define the \emph{twisted generalized X-ray transform} as:
\begin{equation}
\label{equation:generalized-xray}
\Pi_m^{\nabla^{\E}} := {\pi_m}_*(\RR_+(0) + \RR_-(0)) \pi_m^*,
\end{equation}
acting on sections of $\otimes^m_S T^*M \otimes \E \rightarrow M$, see \S\ref{sssection:t-g-x-ray} for further details. This operator turns out to be pseudodifferential of order $-1$ (see \cite[Section 7]{Cekic-Lefeuvre-21-1}) and has some very good analytic properties (such as ellipticity), as we shall see.

Symmetric tensors with values in the bundle $\E$ (also called twisted symmetric tensors in the following) also admit a canonical decomposition into a potential part and a solenoidal part, see \S\ref{sssection:tensors-twisted}. The twisted potential tensors are always contained in the kernel of $\Pi_m^{\nabla^{\E}}$ and we say that the operator is s-injective if this is an equality. We will prove the following:

\begin{theorem}
\label{theorem:genericity-connection}
Let $(M,g)$ be a smooth Anosov manifold of dimension $\geq 3$ and let $\pi_{\E} : \E \rightarrow M$ be a smooth Hermitian vector bundle. There exists $k_0 \gg 1$ such that the following holds. For all $m \in \Z_{\geq 0}$, there exists an open dense set $\mc{S}_m \subset \mc{A}_{\E}$ (for the $C^{k_0}$-topology) of unitary connections with s-injective twisted generalized X-ray transform $\Pi_m^{\nabla^{\E}}$. In particular, the space of connections $\mc{S} := \cap_{m \geq 0} \mc{S}_m$ whose twisted generalized X-ray transforms are all s-injective is \emph{residual} in $\mc{A}_{\E}^{k_0}$.
\end{theorem}

We also point out here that a similar result holds for the induced connection on the endomorphism bundle, see Theorem \ref{theorem:genericity-connection-endomorphism}. This plays a crucial role in the study of the \emph{holonomy inverse problem} which consists in reconstructing a connection (up to gauge) from the knowledge of the trace of its holonomy along closed geodesics, see \cite{Cekic-Lefeuvre-21-1} for further details. We believe that a similar result should hold in the boundary case and this is left for future investigation.

\subsection{Strategy of proof, organization of the paper}

\label{ssection:strategy}

The strategy of both Theorems \ref{theorem:genericity-metric} and \ref{theorem:genericity-connection} is the same, although the metric case (Theorem \ref{theorem:genericity-metric}) is more involved due to complicated computations. The idea is also reminiscent of our previous work \cite{Cekic-Lefeuvre-20}, where a notion of \emph{operators of uniform divergence type} was introduced. Let us discuss the metric case. If the X-ray transform $I_m^{g_0}$ is not s-injective for some $m \in \Z_{\geq 0}$ and $g_0 \in \mc{M}_{\mathrm{Anosov}}$ (or $\mc{M}^{k_0}_{\mathrm{Anosov}}$) then, equivalently, the generalized X-ray transform operator $\Pi_m^{g_0}$ is not s-injective. This operator is non-negative, pseudodifferential of order $-1$ and elliptic on $\ker D^*_{g_0}$ (see \S\ref{sssection:g-x-ray}): as a consequence, it has a well-defined spectrum when acting on the Hilbert space
\[
\mc{H} := L^2(M,\otimes^2_S T^*M) \cap \ker D^*_{g_0},
\]
which lies in $\R_{\geq 0}$ and accumulates to $z=0$. The fact that this operator is not s-injective is equivalent to the existence of an eigenvalue at $z=0$. The accumulation of the spectrum at $0$ (due to the compactness of the operator) is a slight difficulty and we first need to multiply $\Pi_m^{g_0}$ by a certain invertible Laplace-type operator $\Delta$ of order $k > 1/2$ to obtain $P_{g_0} = {\pi_{\ker D^*_{g_0}}} \Delta \Pi_m^{g_0} \Delta {\pi_{\ker D^*_{g_0}}}$ which is a pseudodifferential operator of positive order (hence the spectrum accumulates to $+\infty$) with the same kernel as $\Pi_m^{g_0}$. The idea is to show that we can produce arbitrarily small perturbations $g$ of the metric $g_0$ so that $P_{g}$ has no eigenvalue at $0$. 

If $\gamma$ denotes a small circle near $0$ in $\C$ (such that the interior of $\gamma$ only contains the $0$ eigenvalue of $P_{g_0}$) and $\lambda_g$ is the sum of the eigenvalues of $P_g$ inside $\gamma$, then by elementary spectral theory, we know that $C^{k_0} \ni g \mapsto \lambda_g$ is at least $C^3$ near $g_0$ when $k_0 \gg 1$ is large enough. Moreover, due to the non-negativity of the operators $P_g$, we have $\dd \lambda|_{g=g_0} = 0$. We then compute the second variation and show, using an abstract perturbative Lemma \ref{lemma:abstract}, that for all $f \in C^\infty(M,\otimes^2_S T^*M)$:
\begin{equation}
\begin{split}
\label{equation:reduction-intro}
\dd^2 \lambda|_{g=g_0}(f,f) & = \sum_{i=1}^d \Big(\big\langle \dd^2 P|_{g=g_0}(f,f)u_i, u_i\big\rangle_{L^2} \\
& \hspace{1cm}- 2\big\langle P_{g_0}^{-1}\dd P|_{g=g_0}(f) u_i, \dd P|_{g=g_0}(f) u_i \big\rangle_{L^2}\Big),
\end{split}
\end{equation}
where $d$ is the dimension of $\ker P_{g_0}$ and $(u_1,...,u_d)$ is an $L^2$-orthonormal basis of $\ker P_{g_0}$. Writing the perturbation of the metric as $g_t = g_0 + tf$, we have $\lambda_{g_t} = t^2 \dd^2 \lambda_{g=g_0}(f,f) + t^3 \mc{O}(\|f\|^3_{C^{k_0}})$ and it thus suffices to find $f \in C^\infty(M,\otimes^2_S T^*M)$ such that $\dd^2 \lambda_{g=g_0}(f,f) > 0$. This means that one of the $0$-eigenvalues was ejected for a small perturbation of $g_0$, and iterating this process, one obtains a metric $g$ close to $g_0$ such that $P_g$ is injective.

For that, we assume that the contrary holds, namely that the second variation is always zero. We then consider the maps in \eqref{equation:reduction-intro}, namely $f \mapsto \langle \dd^2 P|_{g=g_0}(f,f)u_i, u_i\rangle_{L^2} $ and $f \mapsto \langle  P_{g_0}^{-1}\dd P|_{g=g_0}(f) u_i, \dd P|_{g=g_0}(f) u_i \rangle_{L^2}$. We show that these quantities can all be put in the form $\langle A f, f \rangle_{L^2}$, for some pseudodifferential operator $A$. The important point here is to evaluate the order of $A$ and to compute exactly its principal symbol.

As a consequence, \eqref{equation:reduction-intro} can be put in the form $\langle B f, f \rangle_{L^2} = 0$ for some $\Psi$DO denoted by $B \in \Psi^*(M,\otimes^2_S T^*M \rightarrow \otimes^2_S T^*M)$\footnote{If $\E , \F \rightarrow M$ are two vector bundles over $M$, we denote by $\Psi^*(M,\E\rightarrow \F)$ the standard space of pseudodifferential operators (of all orders) obtained by quantizing symbols in the Kohn-Nirenberg class $\rho=1,\delta=0$, see \cite{Shubin-01} for further details.}. Taking (real-valued) Gaussian states for the perturbations $f$, we then obtain by an elementary lemma that for all $(x_0,\xi_0) \in T^*M \setminus \left\{0\right\}$ and for all $f \in \otimes^2_S T^*_{x_0}M$:
\[
\langle \sigma_{B}(x_0,\xi_0) f, f \rangle_{\otimes^2_S T^*_{x_0}M } = 0,
\]
where $\sigma_{B}(x_0,\xi_0) \in \End(\otimes^2_S T^*_{x_0}M)$ denotes the principal symbol of $B$. In order to conclude, it is therefore sufficient to contradict the previous equality. This problem turns out to be of \emph{purely algebraic nature} and relies on the representation theory of $\mathrm{SO}(n)$ via spherical harmonics, which is treated in the preliminary section \S\ref{section:spherical}. We also point out that the operator $B$ is \emph{a priori} not elliptic (see Remark \ref{remark:elliptic}), which prevents us from proving that, at least locally, there is only a finite-dimensional submanifold of isometry classes with non-injective X-ray transform. 

The main technical ingredients are recalled in \S\ref{section:technical} but we assume that the reader is familiar with the basics of microlocal analysis. The proof of the genericity in the connection case is developed in \S\ref{section:genericity-connection} (with applications to the tensor tomography question for connections in Corollary \ref{corollary:above}) and the metric case is handled in \S\ref{section:genericity-metric}. Applications of our theorems to generic injectivity of the $X$-ray transform on manifolds with boundary can be found in \S \ref{sec:boundary}.

To conclude, let us mention that the approach initiated in \cite{Cekic-Lefeuvre-20} and developed in the present paper to study generic properties of elliptic pseudodifferential operators seems new (here by generic properties we mean for instance the simplicity of the spectrum, non-degeneracy of nodal sets of eigenfunctions, and so on). It is at least very different from the historical approach of Uhlenbeck \cite{Uhlenbeck-76} and others.  \\

\noindent \textbf{Acknowledgement:} M.C. acknowledges the support of an Ambizione grant (project number 201806) from the Swiss National Science Foundation. During the course of writing the paper, he was also supported by the European Research Council (ERC) under the European Union’s Horizon 2020 research and innovation programme (grant agreement No. 725967). The authors are grateful to Colin Guillarmou and Gabriel Paternain for their encouragement. They also warmly thank the anonymous referees for their valuable feedback that improved this manuscript.

\section{Technical preliminaries}

\label{section:technical}

\subsection{Elementary Riemannian geometry}

\label{ssection:elementary}

We refer to \cite{Paternain-99} for further details on the content of this paragraph. Let $(M,g)$ be a smooth Riemannian manifold. We denote by
\[
SM := \left\{ (x,v) \in TM ~|~ g_x(v,v) = 1 \right\} \subset TM,
\]
its unit tangent bundle. Let $(\varphi_t)_{t \in \R}$ be the geodesic flow generated by the vector field $X$. If $\pi : SM \rightarrow M$ denotes the projection, we define $\V := \ker \dd \pi$ to be the \emph{vertical subspace}. Recall the definition of the \emph{connection map} $\mc{K} : T(SM) \rightarrow TM$: consider $(x,v) \in SM, w \in T_{(x,v)}(SM)$ and a curve $(-\eps,\eps) \ni t \mapsto z(t) \in SM$ such that $z(0)=(x,v), \dot{z}(0)=w$; write $z(t)=(x(t),v(t))$; then $\mc{K}_{(x,v)}(w) := \nabla_{\dot{x}(t)} v(t)|_{t=0}$, where $\nabla$ denotes the Levi-Civita connection of $(M, g)$. The Sasaki metric $g_{\mathrm{Sas}}$ on $SM$ is defined as follows:
\[
	g_{\mathrm{Sas}}(w,w') := g(d \pi(w), d\pi(w')) + g(\mc{K}(w),\mc{K}(w')).
\]
Write $\HH := \ker \mc{K} \cap (\R \cdot X)^\bot$ for the \emph{horizontal subspace} and $\HH_{\mathrm{tot}} := \R X \oplus \HH$ for the \emph{total horizontal space}. Then we have the following splitting:
\begin{equation}
\label{equation:splitting}
T(SM) = \R \cdot X \oplus^\bot \V \oplus^\bot \HH,
\end{equation}
where $\oplus^\perp$ denotes orthogonal sum with respect to $g_{\mathrm{Sas}}$. We will denote by $\pi_{\HH}, \pi_{\HH_{\mathrm{tot}}},\pi_{\V}$ the orthogonal projections onto the respective spaces $\HH, \HH_{\mathrm{tot}},\V$.
We denote by $\nabla^{\mathrm{Sas}}$ the gradient of the Sasaki metric. The splitting \eqref{equation:splitting} gives rise to a decomposition of the gradient
\[
\nabla^{\mathrm{Sas}}(f) = Xf \cdot X + \nabla_{\HH} f + \nabla_{\V}f,
\]
where $f \in C^\infty(SM)$, and $\nabla_{\HH} f \in C^\infty(SM,\HH), \nabla_{\V} f \in C^\infty(SM,\V)$.

The geodesic vector field $X$ is a contact vector field with contact $1$-form $\alpha$ such that $\alpha(X) = 1, \imath_X d \alpha = 0$ and $\alpha$ has the expression:
\begin{equation}\label{eq:alphaDef}
\alpha_{(x,v)}(\xi) = g_x(\dd_{(x,v)} \pi(\xi), v), \quad \xi \in T_{(x, v)}SM.
\end{equation}
We have $\ker \alpha = \HH \oplus \V$ and $d \alpha$ is non-degenerate on $\ker \alpha$ (it is a symplectic form). Moreover $d \alpha|_{\HH \times \HH} = d\alpha|_{\V \times \V} = 0$. The space $\ker \alpha$ is equipped with a canonical almost complex structure $J$ defined in the following way: if $Z \in C^\infty(SM,\ker \alpha)$, we write $Z = (Z_\HH, Z_\V)$ to denote its horizontal $Z_{\mathbb{H}} = d\pi(Z)$ and vertical $Z_{\mathbb{V}} = \mc{K}(Z)$ parts; then $JZ = (-Z_\V,Z_\HH)$, see \cite[Section 1.3.2]{Paternain-99}. For such $Z$, the following relation between the contact form and the Sasaki metric holds (see \cite[Proposition 1.24]{Paternain-99}):
\begin{equation}\label{eq:contact-complex-sasaki}
	\iota_Zd\alpha (\bullet) = -g_{\mathrm{Sas}}(JZ, \bullet).
\end{equation}
We will denote by $\divv^{\mathrm{Sas}} := (\nabla^{\mathrm{Sas}})^*$ the divergence operator with respect to the Sasaki metric. When clear from context, we will drop the Sasaki superscript. Then we have:
\begin{equation}\label{eq:divergence}
	\forall Z \in C^\infty(SM, T(SM)), \forall f \in C^\infty(SM), \quad \divv(fZ) = f\divv(Z) - Zf.
\end{equation}
Equivalently, the divergence operator is defined by
\[
\divv(Z) d\vol_{g_{\mathrm{Sas}}} = - \Lie_Z (d\vol_{g_{\mathrm{Sas}}}),
\]
where $\Lie_Z$ is the Lie derivative along $Z$ and $d\vol_{g_{\mathrm{Sas}}}$ is the volume form of the Sasaki metric. With our conventions, the following formal adjoint formula holds:
\begin{equation}\label{eq:vector_field_formal_adjoint}
	\forall Z \in C^\infty(SM, T(SM)),  \quad Z^* = -Z + \divv(Z). 
\end{equation}
The Sasaki volume form satisfies the property that (see \cite[Exercise 1.33]{Paternain-99})
\begin{equation}\label{eq:dvolSasaki}
	d\vol_{g_{\mathrm{Sas}}} = \frac{1}{(n - 1)!} \alpha \wedge (d\alpha)^{n - 1}.
\end{equation}
When the metric $g \in \mc{M}_{\mathrm{Anosov}}$ is Anosov, the following crucial property is known \cite{Klingenberg-74}:
\begin{equation}
\label{equation:transverse}
E_s \cap \V = \left\{0 \right\}, ~~ E_u \cap \V = \left\{ 0 \right\}.
\end{equation}
As we shall see, this property is essential in proving the pseudodifferential nature of certain operators, see \S\ref{section:pdo}. This also implies that the manifold has no conjugate points, namely:
\begin{equation}\label{equation:conjugate}
\forall t \neq 0, \quad \V \cap \dd \varphi_t(\mathbb{R}X \oplus\V) = \left\{ 0 \right\}.	
\end{equation}

\subsection{Symmetric tensors}

This material is standard but it might be hard to locate a complete reference in the literature. Further details can be found in \cite{Dairbekov-Sharafutdinov-10, Gouezel-Lefeuvre-19,Lefeuvre-thesis, Cekic-Lefeuvre-20}.

\subsubsection{Symmetric tensors in Euclidean space}

\label{sssection:tensors-euclidean}

Let $(E,g_E)$ be a $n$-dimensional Euclidean space and $(\e_1, ..., \e_n)$ be an orthonormal basis. Let $\e_i^* := g_E(\e_i, \bullet)$ be the covector given by the musical isomorphism. We denote by $\otimes^m E^*$, the space of $m$-tensors and $\otimes^m_S E^*$ the space of symmetric $m$-tensors, namely $f \in \otimes^m_S E^*$ if and only if
\[
 \forall v_1,...,v_m \in E,\,\, \forall \sigma \in \mathfrak{S}_m, \quad f(v_1, ..., v_m) = f(v_{\sigma(1)}, ..., v_{\sigma(m)}).
\]
Here $\mathfrak{S}_m$ denotes the permutation group of $\{1, \dotso, m\}$. Given $K = (k_1, ..., k_m)$, we write $\e_K^* := \e_{k_1}^* \otimes ... \otimes \e^*_{k_m}$. The metric $g_E$ induces a natural inner product on $\otimes^m E^*$ given by:
\[
\langle \e_K^*, \e_{K'}^* \rangle_{\otimes^m E^*} := \delta_{k_1 k'_1} ... \delta_{k_m k'_m}.
\]
The symmetrization operator $\mc{S} : \otimes^m E^* \rightarrow \otimes^m_S E^*$ defined by:
\[
\mc{S}(\eta_1 \otimes ... \otimes \eta_m) := \dfrac{1}{m!} \sum_{\sigma \in \mathfrak{S}_m} \eta_{\sigma(1)} \otimes ... \otimes \eta_{\sigma(m)},
\]
is the orthogonal projection onto $\otimes^m_S E^*$. We introduce the trace operator $\mc{T} : \otimes^m_S E^* \rightarrow \otimes^{m-2}_S E^*$:
\[
\mc{T} f := \sum_{i=1}^n f(\e_i,\e_i, \bullet, ..., \bullet),
\]
and this is formally taken to be equal to $0$ for $m=0,1$. We say that a symmetric tensor is \emph{trace-free} if its trace vanishes and denote by $\otimes^m_S E^*|_{0-\Tr}$ this subspace. We let $\mc{J} : \otimes^m_S E^* \rightarrow \otimes^{m+2}_S E^*$ be defined by $\mc{J}(f) := \mc{S}(g_E \otimes f)$ which is the adjoint of the trace map $\mc{T}$ (with respect to the standard inner product previously defined on symmetric tensors). The operator $\mc{T} \circ \mc{J}$ is a scalar multiple of the identity on $\otimes^m_S E^*|_{0-\Tr}$. Moreover, the total space of symmetric tensors of degree $m$ breaks up as the orthogonal sum:
\[
\otimes^m_S E^* = \oplus_{k \geq 0} \mc{J}^{2k}\left(\otimes^{m-2k}_S E^*|_{0-\Tr}\right).
\]
We define $\mathbf{P}_m(E)$ to be the set of homogeneous polynomials of degree $m$; $\mathbf{H}_m(E)$, the subset of harmonic homogeneous polynomials of degree $m$. There is a natural identification $\lambda_m : \otimes^m_S E^* \rightarrow \mathbf{P}_m(E)$ given by the evaluation map $\lambda_m f (v) := f(v,...,v)$ (where $v \in E$). Moreover, $\lambda_m : \otimes^m_S E^*|_{0-\Tr} \rightarrow \mathbf{H}_m(E)$ is an isomorphism and
\begin{equation}\label{eq:harmonicpolys}
\lambda_m : \otimes^m_S E^* = \oplus_{k \geq 0} \mc{J}^{2k}\left(\otimes^{m-2k}_S E^*|_{0-\Tr}\right) \rightarrow \oplus_{k \geq 0} |v|^{2k} \mathbf{H}_{m-2k}(E),
\end{equation}
is a graded isomorphism (it maps each summand to each summand isomorphically). We let $\Ss^{n-1} := \left\{ v \in E ~|~ g_E(v,v) = 1 \right\}$ be the unit sphere in $E$ and $r : C^\infty(E) \rightarrow C^\infty(\Ss^{n-1})$ be the operator of restriction. Define $\pi_m^* := r \circ \lambda_m$, ${\pi_m}_*$ its adjoint, and denote by $\Omega_m(E)$ the spherical harmonics of degree $m$, namely
\[
	\Omega_m(E) :=\ker(\Delta_{\Ss^{n-1}} + m(m+n-2))|_{{ C^\infty(\mathbb{S}^{n - 1})}},
\]
where $\Delta_{\Ss^{n-1}}$ denotes the induced Laplacian on the sphere, and $\mathbf{S}_m(E) := \oplus_{k \geq 0} \Omega_{m-2k}(E)$, where $\Omega_{k}(E) = \left\{ 0 \right\}$ for $k < 0$. It is well-known that
\[
\pi_m^*  : \otimes^m_S E^* = \oplus_{k \geq 0} \mc{J}^{2k}\left(\otimes^{m-2k}_S E^*|_{0-\Tr}\right) \rightarrow \oplus_{k \geq 0} \Omega_{m-2k}(E) = \mathbf{S}_m(E),
\]
is a graded isomorphism.

\subsubsection{Symmetric tensors on Riemannian manifolds}

\label{sssection:tensors-riemannian}

We now consider a Riemannian manifold $(M,g)$. Given $f \in C^\infty(M,\otimes^m_S T^*M)$, we define its symmetric derivative
\[
Df := \mc{S} \circ \nabla f \in C^\infty(M,\otimes^{m+1}_S T^*M),
\]
where $\nabla$ is the Levi-Civita connection. The operator $D$ is an elliptic differential operator of degree $1$ and is of gradient type, i.e. its principal symbol is injective (see \cite[Section 3]{Cekic-Lefeuvre-20} for instance). When the geodesic flow is ergodic, its kernel is given by $\ker(D) = \left\{ 0 \right\}$ for $m$ odd, and $\ker(D) = \R \cdot \mc{S}^{\frac{m}{2}} (g)$ for $m$ even. Its adjoint is denoted by $D^*f = - \Tr(\nabla f) \in C^\infty(M,\otimes^{m-1}_S T^*M)$ and is of divergence type.

Any symmetric tensor $f \in C^\infty(M,\otimes^m_S T^*M)$ can be uniquely decomposed as
\[
f = Dp + h,
\]
where $h \in C^\infty(M,\otimes^m_S T^*M) \cap \ker D^*$ is the solenoidal part and we have $p \in C^\infty(M,\otimes^{m-1}_S T^*M)$ (and $Dp$ is the potential part). We denote by $\pi_{\ran D}$ the $L^2$-orthogonal projection onto the first factor and by $\pi_{\ker D^*} = \mathbbm{1} - \pi_{\ran D}$ the $L^2$-orthogonal projection onto the second factor. Both are pseudodifferential operators of order $0$, namely in $\Psi^0(M,\otimes^m_S T^*M \rightarrow \otimes^m_S T^*M)$. The latter is given by the expression:
\begin{equation}
\label{equation:projection-pi}
\pi_{\ker D^*} = \mathbbm{1} - D(D^*D)^{-1}D^*.
\end{equation}
We use here the convention introduced in \S\ref{ssection:convention} for $(D^*D)^{-1}$ as $D^*D$ has some non-trivial kernel for $m$ even.
The principal symbol of $D$ is given by $\sigma_D(x,\xi) = i j_\xi$ where $j_\xi = i \mc{S}(\xi \otimes \bullet)$ and $i$ is the imaginary unit satisfying $i^2 = -1$, whereas that of $D^*$ is given by $\sigma_{D^*}(x,\xi) = - i \imath_{\xi^\sharp}$, where $\imath_w$ is the contraction by $w$. The space $\otimes^m_S T^*_xM$ breaks up as the orthogonal sum:
\[
\otimes^m_S T^*_xM = \ran(j_\xi) \oplus^\bot \ker \imath_{\xi^\sharp}.
\]
The principal symbol of $\pi_{\ker D^*}$ is then given by the orthogonal projection onto the second summand, namely 
\begin{equation}\label{eq:sigma-proj-ker}
	\sigma_{\pi_{\ker D^*}}(x,\xi) = \pi_{\ker \imath_{\xi^\sharp}}.
\end{equation}
We have the important relation, proved originally in \cite[Proposition 3.1]{Guillemin-Kazhdan-80-2}:
\begin{equation}
\label{equation:x}
	X \pi_m^* = \pi_{m+1}^* D.
\end{equation}
The spherical harmonics introduced previously in \S\ref{sssection:tensors-euclidean} allow to decompose smooth functions $f \in C^\infty(SM)$ as $f= \sum_{m \geq 0} f_m$, where $f_m \in C^\infty(M,\Omega_m)$ is the projection onto spherical harmonics of degree $m$ and
\[
	\Omega_m(x) := \ker\left(\Delta_{\V}(x) + m(m+n-2)\right)|_{{ C^\infty(S_xM)}},
\]
where $\Delta_{\V}(x) : C^\infty(S_xM) \rightarrow C^\infty(S_xM)$ denotes the vertical Laplacian acting on functions on $S_xM$ (i.e. the round Laplacian on the sphere). We call \emph{degree} of $f$ (denoted by $\mathrm{deg}(f)$) the highest non-zero spherical harmonic in this expansion (which can take value $+ \infty$) and say that $f$ has finite Fourier content if its degree is finite. We will say that a function is \emph{even} (resp. \emph{odd}) if it contains only even (resp. odd) spherical harmonics in its expansion, i.e. $f_{2k+1} = 0$ for all $k \in \Z_{\geq 0}$ (resp. $f_{2k}=0$ for all $k \in \Z_{\geq 0}$).
The operator $X$ acts on spherical harmonics as:
\[
X : C^\infty(M,\Omega_m) \rightarrow C^\infty(M,\Omega_{m-1}) \oplus C^\infty(M,\Omega_{m+1}),
\]
and therefore splits into $X = X_- + X_+$, where $X_\pm$ denotes the projection onto the $\Omega_{m\pm 1}$ factor. The operator $X_+$ is of gradient type and thus, for each $m \in \N$, $\ker X_+|_{C^\infty(M,\Omega_m)}$ is finite-dimensional, and we call elements in this kernel \emph{Conformal Killing Tensors} (CKTs).

Eventually, we will also use another lift of symmetric tensors to the unit tangent bundle via the following map, which we call the \emph{Sasaki lift}:
\begin{equation}
\label{equation:lift-sasaki}
\begin{split}
&\pi_{m,\mathrm{Sas}}^* : C^\infty(M,\otimes^m_S T^*M) \to C^\infty(SM, \otimes^m_S T^*(SM)),\\
&  \pi_{m,\mathrm{Sas}}^*f(\xi_1,...,\xi_m) := f(\dd \pi(\xi_1), ..., \dd \pi(\xi_m)).
\end{split}
\end{equation}
We note that the pullback $\pi_{m, \mathrm{Sas}}^*$ is different from $\pi_m^*$ introduced in \S \ref{sssection:tensors-euclidean}.

\subsubsection{Twisted symmetric tensors}

\label{sssection:tensors-twisted}

The previous discussion can be generalized in order to include a twist by a vector bundle $\E \rightarrow M$, see \cite[Section 2.3]{Cekic-Lefeuvre-20} for further details. We let $(e_1,...,e_r)$ be a local orthonormal frame of $\E$ (defined around a fixed point $x_0 \in M$). The smooth sections of the pullback bundle $\pi^*\E \rightarrow SM$ can also be decomposed into spherical harmonics, namely $f \in C^\infty(SM,\pi^*\E)$ can be written as $f = \sum_{m \geq 0} f_m$, where $f_m \in C^\infty(M,\Omega_m \otimes \E)$ and a similar notion of degree is defined, as well as the evenness/oddness of a section.

If $\nabla^{\E}$ is a unitary connection on $\E$ given in a local patch of coordinates $U \ni x_0$ by $\nabla^{\E} = d + \Gamma$, where $\Gamma$ is a connection $1$-form with values in skew-Hermitian endomorphisms, $(\Gamma(\partial_{x_i}))_{k \ell} =\Gamma_{i{\ell}}^k$, and we consider a twisted symmetric tensor $f \in C^\infty(M,\otimes^m_S T^*M \otimes \E)$, which we write locally as $f = \sum_{k=1}^r u_k \otimes e_k$ with $u_k \in C^\infty(U,\otimes^m_S T^*U)$, we can define its symmetric derivative
\begin{equation}
\label{equation:de}
D_{\E}\left(\sum_{k=1}^r u_k \otimes e_k\right) = \sum_{k=1}^r \left( D u_k + \sum_{\ell=1}^r \sum_{i=1}^n \Gamma_{i{\ell}}^k \mc{S}(u_{\ell} \otimes dx_i)\right) \otimes e_k.
\end{equation}
As in \S\ref{sssection:tensors-riemannian}, twisted symmetric tensors can be uniquely decomposed as $f = D_{\E}p + h$, where $p \in C^\infty(M,\otimes^{m-1}_S T^*M \otimes \E)$ and $h \in C^\infty(M,\otimes^m_S T^*M \otimes \E)$ is solenoidal, i.e. in $\ker (D_{\E})^*$.

The pullback operators extend as maps
\[
\pi_m^* : C^\infty(M,\otimes^m_S T^*M \otimes \E) \rightarrow \oplus_{k \geq 0} C^\infty(M,\Omega_{m-2k} \otimes \E).
\]
More precisely, in the local orthonormal frame $(e_i)_{i = 1}^r$, for a section $f$ written locally as above as $f = \sum_{k = 1}^r u_k \otimes e_k$, we have $\pi_m^*f(x, v) := \sum_{k = 1}^r \pi_m^*u_k(x, v) e_k(x) \in \E(x)$, where on the right hand side $\pi_m^*$ acts on symmetric tensors as introduced in \S \ref{sssection:tensors-euclidean}. For simplicity, we keep the same notation for both twisted and non-twisted pullbacks and note that the two agree in the case $\E = \M \times \mathbb{C}$. The bundle $\pi^*\E \rightarrow SM$ is naturally equipped with the pullback connection $\pi^*\nabla^{\E}$ and we set $\X := (\pi^*\nabla^{\E})_X$ which is a differential operator of order $1$ acting on $C^\infty(SM,\pi^*\E)$. We still have the relation:
\begin{equation}
\label{equation:xe}
\X \pi_m^* = \pi_{m+1}^* D_{\E},
\end{equation}
and $\X$ decomposes as:
\begin{equation}\label{eq:mapping-property-X}
\X : C^\infty(M,\Omega_m \otimes \E) \rightarrow C^\infty(M,\Omega_{m-1} \otimes \E) \oplus C^\infty(M,\Omega_{m+1} \otimes \E),
\end{equation}
that is $\X$ splits as $\X = \X_- + \X_+$, where $\X_+$ is of gradient type. Elements in $\ker \X_+$ are called \emph{twisted Conformal Killing Tensors}. Non-existence of twisted CKTs is a generic property of connections as proved in \cite{Cekic-Lefeuvre-20}.

\subsection{Pollicott-Ruelle theory}

\label{ssection:ruelle}

The theory of Pollicott-Ruelle resonances which is briefly recalled below has been widely studied in the literature, see \cite{Liverani-04, Gouezel-Liverani-06,Butterley-Liverani-07,Faure-Roy-Sjostrand-08,Faure-Sjostrand-11,Faure-Tsuji-13,Dyatlov-Zworski-16}. We also refer to \cite{Guillarmou-17-1,Cekic-Lefeuvre-20,Lefeuvre-thesis} for further details on these paragraphs. In what follows, we will make the running assumption that $(M, g)$ is a closed Anosov manifold.

\subsubsection{Meromorphic extension of the resolvents}

Let $\E \rightarrow M$ be a Hermitian bundle over $M$ equipped with a unitary connection $\nabla^{\E}$. We consider the pullback bundle $\pi^*\E \rightarrow SM$ equipped with the pullback connection $\pi^*\nabla^{\E}$ and set $\X := (\pi^*\nabla^{\E})_X$. Defining the domain
\[
\mc{D}(\X) := \left\{ u \in L^2(SM,\pi^*\E) ~|~ \X u \in L^2(SM,\pi^*\E) \right\}\footnote{The scalar product on $L^2$ is given by $\langle f, f' \rangle_{L^2} := \int_{SM} h\left( f(x,v), f'(x,v)\right)_{\E_x} \dd \mu(x,v)$, where $h$ denotes the Hermitian metric on $\E$ and $\mu$ is the Liouville measure on $SM$.},
\]
the differential operator $\X$ (of order $1$) is skew-adjoint as an unbounded operator with dense domain $\mc{D}(X)$ and has absolutely continuous spectrum on $i\R$ (with possibly embedded eigenvalues).

We introduce the positive (resp. negative) $\RR_+(z)$ (resp. $\RR_-(z)$) resolvents, defined for for $\Re(z) > 0$ by:
\[
\RR_{\pm}(z) := (\mp \X-z)^{-1} = - \int_0^{\pm \infty} e^{\mp tz} e^{-t \X} \dd t.
\]
Note that given $(x, v) \in SM$ and $f \in C^\infty(SM, \pi^*\E)$, we have that $e^{-t \X}f (x, v) \in \E_{x}$ is the parallel transport of $f(\varphi_{-t}(x, v))$ along the flowline $(\varphi_s(z))_{s \in [-t,0]}$ with respect to the connection $\pi^* \nabla^{\E}$.

These resolvents initially defined on $\left\{\Re(z) > 0 \right\}$ can be meromorphically extended to $\C$ by making $\X$ act on \emph{anisotropic Sobolev spaces}. More precisely, there exists a scale of Hilbert spaces $\mc{H}_\pm^s$ (where $s > 0$) and a constant $c > 0$ such that
\[
\left\{ \Re(z) > -cs \right\} \ni z \mapsto \RR_\pm(z) \in \mc{H}^s_\pm,
\]
are meromorphic families of operators with poles of finite rank. These spaces are defined so that $f \in \mc{H}^s_+$ (resp. $\mc{H}^s_-$) implies that $f$ is microlocally in $H^s$ near $E_s^*$ (resp. $H^{s}$ near $E_u^*$) and microlocally $H^{-s}$ near $E_u^*$ (resp. $H^{-s}$ near $E_s^*$). The poles are called the \emph{Pollicott-Ruelle resonances}: they are intrinsic to the operators $\X$ and do not depend on any choices made in the construction of the spaces. Moreover, these operators are holomorphic in $\left\{\Re(z) > 0\right\}$ and thus all the resonances are contained in $\left\{\Re(z) \leq 0\right\}$.

\subsubsection{Generalized X-ray transform}

\label{sssection:g-x-ray}

When $\E = \C \times M$, $\X = X$ is nothing but the vector field and we use the notations $R_\pm(z)$ for the resolvents. In this case, there is a single resonance on $i\R$ located at $z=0$. It is a pole of order $1$ and the resolvents have the expansion near $z=0$:
\[
R_{\pm}(z) = - \dfrac{\Pi_\pm}{z} - R_{\pm,0} - R_{\pm, 1}z + \mc{O}(z^2),
\]
for some operators $R_{\pm, 1}, R_{\pm,0} : \mc{H}^s_\pm \rightarrow \mc{H}^s_\pm$, bounded for any $s > 0$. Moreover, the spectral projection at $0$ (i.e. the residue at $z=0$) is 
\begin{equation}\label{eq:Pi_pm}
	\Pi_+ = \Pi_- = \langle \bullet, \mu \rangle \mathbf{1},
\end{equation} 
where $\mu$ is the normalized Liouville measure i.e. so that $\langle \mathbf{1}, \mu \rangle = 1$, see \cite{Guillarmou-17-1} for instance. We record a few useful relations involving $\Pi_\pm$ and $R_{\pm, 0}$:
\begin{align}\label{eq:resolvent-identities}
\begin{split}
	XR_{+, 0} &= R_{+, 0} X = \mathbbm{1} - \Pi_+,\quad XR_{-, 0} = R_{-, 0} X = -\mathbbm{1} + \Pi_-,\quad R_{+, 0}^* = R_{-, 0},\\
	XR_{+, 1} &= R_{+, 1}X = -R_{+, 0}, \quad\,\,\, XR_{-, 1} = R_{-, 1}X = R_{-, 0}.
\end{split}
\end{align}
We introduce the operator
\begin{equation}
\label{equation:capital-i}
\mc{I} := R_{+,0} + R_{-,0} +  \Pi_+
\end{equation}
and define the \emph{generalized X-ray transform} by:
\begin{equation}
\label{equation:g-x-ray}
	\Pi_m^g := {\pi_m}_* \mc{I} \pi_m^*,
\end{equation}
which is an operator acting on sections of the bundle $\otimes^m_S T^*M \rightarrow M$ of symmetric tensors. When clear from context, for the simplicity of notation we will drop the superscript $g$ in $\Pi_m^g$. We say that $\Pi_m$ is $\mathrm{s}$-injective if $\Pi_m$ is injective when restricted to $\ker D^*$. The following provides a relation with the X-ray transform $I_m$:
\begin{lemma}
\label{lemma:relation}
Assume that $(M, g)$ is a closed Anosov manifold. Let $u \in C^\infty(M,\otimes^m_S T^*M)$. Then $\Pi_m u = 0$ if and only if there exists $\VV \in C^\infty(SM)$ such that $\pi_m^* u = X \VV$. Moreover, $I_m$ is $\mathrm{s}$-injective if and only if $\Pi_m$ is $\mathrm{s}$-injective.
\end{lemma}
\begin{proof}
The first part follows from \cite[Theorem 1.1]{Guillarmou-17-1}, and the following observation. Note that our convention is different from \cite{Guillarmou-17-1} who writes $\Pi := R_{+, 0} + R_{-, 0}$ and uses this operator instead of $\mc{I}$ in definition \eqref{equation:g-x-ray}; however, both $\Pi$ and $\Pi_+$ are non-negative operators (see e.g. \cite[Lemma 5.1]{Cekic-Lefeuvre-20}) and so $\Pi_mu = 0$ is equivalent to $\Pi \pi_{m}^*u = 0$ and $\Pi_+ \pi_m^*u = 0$. By \eqref{eq:Pi_pm}, the latter condition gives that the average of $\pi_m^*u$ is zero, so \cite[Theorem 1.1]{Guillarmou-17-1} indeed applies.

The second part is then the consequence of the smooth Liv\v{s}ic theorem (see e.g. \cite[Lemma 2.5.4]{Lefeuvre-thesis}).
\end{proof}

By the preceding Lemma, the study of $I_m$ is reduced to the study of $\Pi_m$ and as we shall see in \S \ref{section:pdo}, $\Pi_m$ enjoys very good analytic properties.  In other words, in order to prove our main Theorem \ref{theorem:genericity-metric}, it suffices to show that $\Pi_m^g$ is generically $\mathrm{s}$-injective with respect to the metric.

\subsubsection{Twisted generalized X-ray transform}

\label{sssection:t-g-x-ray}

We now go back to the case of a Hermitian vector bundle $\E \rightarrow M$. For the sake of simplicity, we assume that $\ker \X = \left\{ 0 \right\}$ but the discussion could be generalized, see the footnote at the beginning of \S\ref{section:genericity-connection}. We introduce:
\[
\mc{I}_{\nabla^{\E}} := \RR_+ + \RR_-,
\]
where $\RR_+ := -\RR_+(z=0)$ and $\RR_- := -\RR_-(z=0)$, and define the \emph{twisted generalized X-ray transform} by:
\begin{equation}
\label{equation:t-g-x-ray}
\Pi^{\nabla^{\E}}_m := {\pi_m}_* \mc{I}_{\nabla^{\E}} \pi_m^*
\end{equation}
Similarly to the first part of Lemma \ref{lemma:relation}, the following was shown in \cite[Lemma 5.1]{Cekic-Lefeuvre-20}:
\begin{lemma}
\label{lemma:relation-twisted}
Let $u \in C^\infty(M,\otimes^m_S T^*M \otimes \E)$. Then $\Pi_m^{\nabla^{\E}} u = 0$ if and only if there exists $\VV \in C^\infty(SM, \pi^*\E)$ such that $\pi_m^* u = \X \VV$.
\end{lemma}

\subsection{Properties of the resolvent under perturbations}
\label{ssection:resolvent-perturbations-theory}

The generalized X-ray transform operators $\Pi^g_m$ (we now add the index $g$ to insist on the metric-dependence) and $\Pi_m^{\nabla^{\E}}$ introduced in the previous paragraphs depend on a choice of metric $g$ and/or connection $\nabla^{\E}$. In the following, we will consider perturbations of these operators with respect to these geometric data. For $\E, \F \to M$, two smooth Hermitian vector bundles over $M$, and $s \in \R$, the spaces of pseudodifferential operators $\Psi^s(M, \E \to \F)$ are Fréchet spaces (see \cite[Section 2.1]{Guillarmou-Knieper-Lefeuvre-19} for instance) where the seminorms are defined thanks to local coordinates by taking the seminorms of the full local symbol in the charts. Let us also mention that it is also possible to consider pseudodifferential operators obtained by quantizing symbols with limited regularity (see \cite{Taylor-91}): actually, all the standard arguments of microlocal analysis (such as boundedness on Sobolev spaces for instance) involve only a finite number of derivatives of the full symbol, and this number depends linearly on the dimension. As a consequence, for $k \gg n$ (where $n$ is the dimension of $M$), we can consider the space $\Psi^s_{(k)}(M, \E \to \F)$ of pseudodifferential operators obtained by quantizing $C^k$-symbols (satisfying the usual symbolic rules of derivation). 

\begin{lemma}
\label{lemma:differentiabilty}
The following maps are smooth:
\begin{align*}
\mc{M}_{\mathrm{Anosov}} \ni &g \mapsto \Pi_m^g \in \Psi^{-1}(M,\otimes^m_S T^*M \to \otimes^m_S T^*M), \\
C^\infty(M,T^*M \otimes \End_{\mathrm{sk}}(\E)) \ni &\nabla^{\E} \mapsto \Pi_m^{\nabla^{\E}} \in \Psi^{-1}(M,\otimes^m_S T^*M \otimes \E \to \otimes^m_S T^*M \otimes \E).
\end{align*}
More precisely, for every $k \in \mathbb{Z}_{\geq 0}$, $k'' \in \Z_{\geq 1}$, there is $k' \gg \max(n,k)$ such that the following maps are $C^k$:
\begin{align*}
\mc{M}^{k'}_{\mathrm{Anosov}} \ni &g \mapsto \Pi_m^g \in \Psi^{-1}_{(k'')}(M,\otimes^m_S T^*M \to \otimes^m_S T^*M), \\
C^{k'}(M,T^*M \otimes \End_{\mathrm{sk}}(\E)) \ni &\nabla^{\E} \mapsto \Pi_m^{\nabla^{\E}} \in \Psi^{-1}_{(k'')}(M,\otimes^m_S T^*M \otimes \E \to \otimes^m_S T^*M \otimes \E).
\end{align*}
\end{lemma}

Lemma \ref{lemma:differentiabilty} will be used in \S\ref{section:genericity-connection} and \S\ref{section:genericity-metric} in order to perturb the generalized X-ray transforms with respect to the connection/metric. For $k=0$, Lemma \ref{lemma:differentiabilty} is precisely the content of \cite[Proposition 4.1]{Guillarmou-Knieper-Lefeuvre-19}. Inspecting the proof, one can see that it also works for higher order derivatives. The heart of the proof is based on understanding the differentiability of the resolvent map
\[
	C^\infty(SM,T(SM)) \times \C \ni (Y,z) \mapsto (\mp Y-z)^{-1} \in \mc{L}(\mc{H}^s_\pm),
\]
where $\mc{H}^s_\pm$ is the scale of anisotropic Sobolev spaces (which can be made independent of the vector field by \cite{Bonthonneau-19}), and $Y$ is a vector field close to $X$. This perturbation theory is now standard and we refer to \cite{Bonthonneau-19, Dang-Guillarmou-Riviere-Shen-20,Cekic-Delarue-Dyatlov-Paternain-22} for further details.

\subsection{Notational convention} \label{ssection:convention} Throughout the paper (see for instance \eqref{eq:expansion-appendix} below), if $A$ is an operator on a Hilbert space $\mc{H}$ with meromorphic resolvent $(A-z)^{-1}$ on a half-space $\Re(z) \geq - \eps$ (for some $\eps > 0$), we use the convention that $A^{-1}$ denotes the holomorphic part of the resolvent at $0$. More precisely, close to $z = 0$ we can write
\[
(A-z)^{-1} = R_A + \sum_{j=1}^N \dfrac{A_j}{z^j} + \mc{O}(z),
\]
for some finite rank operators $(A_j)_{j = 1}^N$ and we set $A^{-1} := R_A$. In particular, if $z=0$ is not in the spectrum of $A$, $A^{-1}$ is the inverse for $A$.

\section{On spherical harmonics}

\label{section:spherical}

We record here some facts about spherical harmonics. We keep the notation $(E,g_E)$ for a Euclidean vector space of dimension $n$.

\subsection{The restriction operator}\label{ssection:restriction}

In the following, we will need to understand how the degree of a function is changed when restricting to a hypersphere. For $\xi \in E^*\setminus \{0\}$, define $\vec{n}(\xi) := \frac{\xi^\sharp}{|\xi|}$. If $\Ss^{n-1}$ denotes the unit sphere in $E$, we introduce $\Ss^{n-2}_\xi := \left\{ v \in \Ss^{n-1} ~|~ \langle \xi, v \rangle = 0 \right\}$. Any vector $v \in \Ss^{n-1} \setminus \{\pm \vec{n}(\xi)\}$ can be uniquely decomposed as $v = \cos(\varphi)\vec{n}(\xi) + \sin(\varphi) u$, where $\varphi \in (0,\pi), u \in \Ss^{n-2}_{\xi}$ (the diffeomorphism is singular at the extremal points $\varphi= 0$ and $\varphi=\pi$ but since they form a set of measure $0$, this is harmless in what follows). The round measure $\dd S(v)$ on $\Ss^{n-1}$ is then given in these new coordinates by 
\begin{equation}\label{eq:Jacobian}
	\dd S(v) = \sin^{n-2}(\varphi) \dd \varphi \dd S_{\xi}(u),
\end{equation}
	 where $\dd S_{\xi}(u)$ denotes the canonical round measure on the $(n-2)$-dimensional sphere $\Ss^{n-2}_{\xi}$.

\subsubsection{Standard restriction} We start with the following:

\begin{lemma}
\label{lemma:restriction-1}
Assume $n > 2$. Let $\WW \in C^\infty(\Ss^{n-1})$. Then $\WW$ has degree $\leq m$ if and only if the restriction $\WW|_{\Ss^{n-2}_{\xi}}$ to any hypersphere $\Ss^{n-2}_{\xi}$ has degree $\leq m$.
\end{lemma}

\begin{proof}
We start with the easy direction. If $\WW \in \Omega_m(\Ss^{n-1})$ is a spherical harmonic of degree $m$, then:
\[
\WW|_{\Ss^{n-2}_{\xi}} \in \oplus_{k \geq 0} \Omega_{m-2k}(\Ss^{n-2}_{\xi}).
\]
This fact follows from the following observation: if $\WW \in \Omega_m(\Ss^{n-1})$, then $\WW$ is the restriction of a harmonic homogeneous polynomial $P \in \mathbf{H}_m(E)$ defined on $E$. For any $\xi \in E^*\setminus \{0\}$, let $H = \ker(\xi) \subset E$ denote the codimension $1$ hyperplane determined by $\xi$. Then $P|_{H}$ is still a homogeneous polynomial of degree $m$ (it may not be harmonic, though) and thus its restriction to $\Ss^{n-2}_{\xi}$ is a sum of spherical harmonics of degree $\leq m$ (and with same parity as $m$).

We now show the converse. The case $m=0$ is obvious so we can always assume that $m \geq 1$. Note first that we may split $\WW$ into odd and even terms, $\WW = \WW_{\mathrm{odd}} + \WW_{\mathrm{even}}$. We have $(\WW|_{\mathbb{S}^{n-2}_{\xi}})_{\mathrm{odd}} = \WW_{\mathrm{odd}}|_{\mathbb{S}^{n-2}_{\xi}}$ and $(\WW|_{\mathbb{S}^{n-2}_{\xi}})_{\mathrm{even}} = \WW_{\mathrm{even}}|_{\mathbb{S}^{n-2}_{\xi}}$, and so for every $\xi$, both $\WW_{\mathrm{even}}|_{\mathbb{S}^{n-2}_{\xi}}$ and $\WW_{\mathrm{odd}}|_{\mathbb{S}^{n-2}_{\xi}}$ are of degree $\leq m$. Thus, we may assume $\WW$ is either pure odd or pure even, and that moreover this is the parity of $m$ (if $m$ and $\WW$ have distinct parities, then the hypothesis of the Lemma is true for $m-1$).

The conclusion is now implied by the following claim
: let $\WW \in C^\infty(E \setminus \left\{0\right\})$ be an $m$-homogeneous function (since $m > 0$ it is at least continuous at $x=0$) such that the restriction $\WW|_{H}$ to any hyperplane $H \subset E$ is a homogeneous polynomial of degree $m$. Then $\WW$ is a homogeneous polynomial of degree $m$. 

First of all, we start by proving that $\WW$ is smooth at $x=0$. Let $(\mathbf{e}_1, ..., \mathbf{e}_n)$ be an orthonormal basis of $E$ and write $x = \sum_{i=1}^n x_i \mathbf{e}_i$. We claim that $\partial^{m+1}_{x_i} \WW \equiv 0$ on $E$. Indeed, fix $i\in\left\{1,...,n\right\}$, fix $x \in E$ and consider a hyperplane $H$ containing both $x$ and $\mathbf{e}_i$ (which forces the condition $n > 2$). Then $f|_{H}$ is a polynomial of degree $m$. In particular, it is smooth and satisfies:
\[
\partial^{m+1}_{x_i} \WW (x) = \partial_t^{m+1} \WW(x+t\mathbf{e}_i)|_{t=0} = 0,
\]
since it is polynomial. Actually, $\partial^k_{x_i} \WW \equiv 0$ as long as $k \geq m+1$. In particular, $P \WW = 0$ on $E$, where $P = \sum_{i=1}^n \partial^{2m}_{x_i}$. As $P$ is elliptic and $\WW$ continuous, this gives that $\WW$ is smooth on $E$.

We now write by Taylor's theorem:
\[
\WW(x) = \sum_{|\alpha| \leq m} \dfrac{1}{\alpha !} x^\alpha \cdot \partial_x^\alpha \WW (0) + R(x),
\]
where $R(x) = \mc{O}(|x|^{m+1})$ as $|x| \to 0$ and define $S(x) := \sum_{|\alpha| \leq m} \dfrac{1}{\alpha !} x^\alpha \cdot \partial_x^\alpha \WW (0)$. Taking any hyperplane $H$, we obtain:
\[
\WW|_{H}(x) - S|_{H}(x) = \mc{O}(|x|^{m+1}).
\]
The left-hand side is a polynomial of degree $\leq m$ so it implies that it is equal to $0$. This gives that $R|_{H} \equiv 0$. Since this holds for every hyperplane $H$, this implies $R \equiv 0$ and $\WW$ is a polynomial of degree $\leq m$. Using $m$-homogeneity of $\WW$, it is homogeneous of degree $m$.
\end{proof}

\begin{remark}
We observe that the proof actually gives a stronger statement which is: assume $\WW \in C^\infty(E \setminus \left\{ 0 \right\})$ is $m$-homogeneous and a polynomial of degree $m$ when restricted to \emph{any plane} (and not hyperplane), then it is a polynomial of degree $m$ on $E$. 
\end{remark}

\subsubsection{Differentiated restriction}

The following lemma will be used for the generic s-injectivity with respect to the metric. We will denote by $\nabla$ the gradient with respect to the spherical metric on $\mathbb{S}^{n-1}$.

\begin{lemma}
\label{lemma:differentiated-restriction}
Assume $n \geq 3$ and let $m \in \Z_{\geq 0}$ and $\WW \in C^\infty(\Ss^{n-1})$ such that $\mathrm{deg}(\WW) \geq m+1$. Then, there exists $\xi \in E^* \setminus \left\{0\right\}$ such that $\langle \xi, \nabla \WW (\bullet)\rangle|_{\Ss^{n-2}_\xi}$ has degree $\geq m$ (seen as a function on $\Ss^{n-2}_\xi$).
\end{lemma}

This Lemma will be applied later in each fibre $E=T_xM$ and $\nabla$ will be the vertical gradient $\nabla_{\V}$; we will take $\xi \in T_x^*M$. If $m = 0$, degree $\geq 0$ also implies non-zero.

\begin{proof}
We assume the degree of $\langle \xi, \nabla\WW (\bullet)\rangle|_{\Ss^{n-2}_\xi}$ is always $< m$ (for all $\xi \neq 0$) and show that this forces $\WW$ to be of degree $\leq m$. In fact, we may assume without loss of generality that $\WW$ is either pure odd or pure even, and of the same parity as $m + 1$. Let us deal with the $m \geq 2$ case first. 

First of all, we extend the smooth function $\WW$ to an $(m - 1)$-homogeneous function on $E$ (which we still denote by $\WW$). In particular, this extension is smooth on $E \setminus \left\{0 \right\}$. We now claim that for every $\xi$, $\dd \WW (\xi^\sharp)$ is a homogeneous polynomial of degree $\leq m - 2$ on $\ker \xi$. Indeed, consider a point $v \in \ker \xi \setminus \left\{0\right\}$. The total gradient on $E$ is
\[
\nabla^{\mathrm{tot}} f = \sum_{i=1}^n \partial_{x_i} f (v). \partial_{x_i} = \nabla f + \dd f(\vec{n}). \vec{n},
\]
where $\vec{n} := \frac{v}{|v|}$, $\nabla$ denotes the gradient of $f$ restricted to the spheres and $(x_i)_{i = 1}^n$ are the coordinates induced by an orthonormal basis $(\e_i)_{i = 1}^n$ of $ E$. Hence, for $v \in \ker \xi$, we have:
\begin{equation}\label{eq:xiderivative}
\langle \xi , \nabla \WW (v) \rangle  = \sum_{i=1}^n \xi_i. \partial_{x_i} \WW (v) - \dd \WW(\vec{n}). \langle \xi, \vec{n}\rangle = \dd \WW(\xi^\sharp),
\end{equation}
where $\xi_i.$ denotes multiplication by $\xi_i$. Therefore, $\langle \xi , \nabla \WW (v) \rangle|_{\ker \xi}$ is a homogeneous function of degree $m-2$ whose restriction to the sphere $\Ss^{n-2}_\xi$ is of same parity as $m$, and thus has degree $\leq m - 2$. As a consequence, it is a homogeneous polynomial of degree $m-2$ on $\ker \xi$. 

We now fix an arbitrary $v_0 \in E \setminus \left\{ 0 \right\}$ and consider the Taylor-expansion of $\WW$ at this point:
\begin{equation}
\label{equation:taylor-malin}
\WW(v) = \underbrace{\sum_{|\alpha| \leq m - 1} (v-v_0)^\alpha (\alpha!)^{-1} \partial^\alpha_x f(v_0)}_{P(v):=} + R(v),
\end{equation}
where $R(v) = \mc{O}(|v-v_0|^{m})$. We consider $v_1 \in E \setminus \left\{0\right\}$, $w \in \mathrm{Span}(v_0,v_1)^\bot$. If we differentiate \eqref{equation:taylor-malin} in the $w$-direction and then restrict to the hyperplane $w^\bot$, then we know by the previous discussion that $\dd \WW(w) |_{w^\bot}$ is a polynomial of degree $\leq m-2$, and so is $\dd P(w)|_{w^\bot}$. Moreover, from Taylor's theorem $\dd R(w) = \mc{O}(|v-v_0|^{m - 1})$. As a consequence: $\dd (\WW - P)(w)|_{w^\perp} = \dd R(w)|_{w^\perp}$ is a polynomial of degree $\leq m-2$ which vanishes to order $m-1$ at $v_0$; it is therefore constant equal to $0$. Evaluating at $v_1$, this shows that $\dd R(w) = 0$ at $v_1$. 

We now introduce $G_{v_0} \subset \mathrm{SO}(n)$, the isotropy subgroup of $v_0$, i.e. the subgroup of rotations fixing the $v_0$ axis. By the previous discussion, $R$ satisfies the following (see Figure \ref{figure:sphere}): given a sphere $\Ss^{n-1}(r):=\left\{|v|=r\right\}$, $\gamma^* R = R$ for all $\gamma \in G_{v_0}$.
\begin{figure}[htbp!]
\centering
\includegraphics{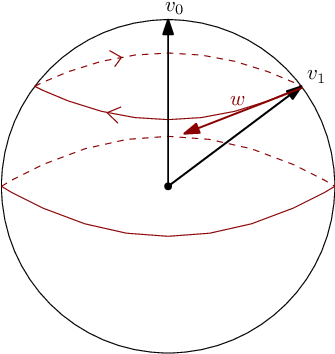}
\caption{The function $R$ is constant along the red orbits which correspond to the rotation around the $v_0$-axis.}
\label{figure:sphere}
\end{figure}

We restrict this equality to the unit sphere $\Ss^{n-1}$ and observe that \eqref{equation:taylor-malin} implies $\WW|_{\Ss^{n-1}} = q_{v_0} + S_{v_0}$, where $q_{v_0} \in \Omega_{\leq m - 1} = \oplus_{k \geq 0} \Omega_{m - 1 - k}$ is a sum of spherical harmonics of degree $\leq m-1$ and $S_{v_0}$ is invariant by the action of $G_{v_0}$. Note that $v_0$ is arbitrary and taking some other $v_1$, we see that $S_{v_1} - S_{v_0} \in \Omega_{\leq m - 1}$. As each $\Omega_{k}$ is a representation of $\mathrm{SO}(n)$ by pullback, in particular it is invariant by $G_{v_1}$. This gives that for all $\gamma \in G_{v_1}$, one has: $\gamma^*S_{v_1} - \gamma^*S_{v_0} = S_{v_1} - \gamma^* S_{v_0} \in \Omega_{\leq m - 1}$. Hence, for all $\gamma \in G_{v_1},  \gamma^* S_{v_0} - S_{v_0} \in \Omega_{\leq m - 1}$. Taking $\gamma' \in G_{v_2}$ for some other arbitrary $v_2$, we see that $(\gamma')^*\gamma^* S_{v_0} - (\gamma')^* S_{v_0} \in \Omega_{\leq m - 1}$ and since we also have $(\gamma')^* S_{v_0} - S_{v_0} \in \Omega_{\leq m - 1}$, this gives that $(\gamma')^*\gamma^* S_{v_0} - S_{v_0} \in \Omega_{\leq m - 1}$. By induction, for any $\gamma_1, \dotso, \gamma_{\ell}$ belonging to isotropy subgroups of $\mathrm{SO}(n)$, we have:
\[
\gamma_1^*\dotso \gamma_{\ell}^* S_{v_0} - S_{v_0} \in \Omega_{\leq m - 1}.
\]
As products of isotropy subgroups generate $\mathrm{SO}(n)$, we deduce that for all $\gamma \in \mathrm{SO}(n), \gamma^* S_{v_0} - S_{v_0}  \in \Omega_{\leq m - 1}$. Decomposing $S_{v_0} = \sum_{k \geq 0} (S_{v_0})_k$ into spherical harmonics, we then see that $\gamma^* (S_{v_0})_k = (S_{v_0})_k$ for all $k \geq m$ and $\gamma \in \mathrm{SO}(n)$. As $\Omega_k$ is irreducible \cite[Theorem 3.1]{Helgason-00}, this implies that $(S_{v_0})_k = 0$. Hence $S_{v_0}$ is of degree $\leq m - 1$ and $\WW|_{\mathbb{S}^{n-1}} = q_{v_0} + S_{v_0}$ is also of degree $\leq m - 1$. This completes the proof of the $m \geq 2$ case.

Finally, if $m = 0, 1$, then \eqref{eq:xiderivative} implies that $\dd \WW(\xi^\sharp)|_{\ker \xi} \equiv 0$. Then taking $R(v) := \WW(v)$ in \eqref{equation:taylor-malin}, it is straightforward that $\dd R(w) = 0$ at $v_1$ where $w, v_1$ are as before. The remainder of the proof works the same to show that $\WW|_{\mathbb{S}^{n-1}}$ is invariant under the isotropy subgroups, hence constant, contradicting that $\deg(\WW) > 0$.
\end{proof}

\subsection{The extension operator}

In this paragraph, we study an operator of extension from a hypersphere $\Ss^{n-2}_\xi$ to the whole sphere $\Ss^{n-1}$. First of all, for $m \in \Z_{\geq 1}$, we introduce the constant:
\begin{equation}
\label{equation:normalization}
C_m := \int_0^{\pi} \sin^{m-1}(\varphi) \dd \varphi = \sqrt{\pi} \dfrac{\Gamma(\frac{m}{2})}{\Gamma(\frac{m+1}{2})},
\end{equation}
where $\Gamma$ is the usual Gamma function. Given a smooth function $f \in C^\infty(\Ss^{n-2}_{\xi})$ or more generally, one can take a section $f \in C^\infty(\Ss^{n-2}_{\xi}, \pi^*\E)$, where $\pi : \Ss^{n-1} \rightarrow \left\{0\right\}$ is the projection, we define its \emph{extension of degree $k \in \N$ to the whole sphere} $\Ss^{n-1}$ by the formula:
\begin{equation}
\label{equation:extension-k}
E^{k}_{\xi}f \in L^2(\Ss^{n-1}, \pi^*\E),\quad E^k_{\xi}f\big(\cos(\varphi)\vec{n}(\xi) + \sin(\varphi) u\big) := \sin^k(\varphi) f(u).
\end{equation}
Note that $E^{k}_{\xi}$ extends to $L^2(\Ss^{n-2}_{\xi}, \pi^*\E)$ by continuity since by definition and \eqref{eq:Jacobian}
\begin{equation}\label{eq:E^k-squared}
	\|E^{k}_{\xi}f\|_{L^2(\mathbb{S}^{n-1}, \pi^*\E)}^2 = C_{2k + n -1} \|f\|_{L^2(\mathbb{S}^{n-2}_{\xi}, \pi^*\E)}^2.
\end{equation}
Moreover, we have (cf. \cite[Lemma B.1.1]{Lefeuvre-thesis}):
\begin{lemma}\label{lemma:useful-identity}
For any $f \in C^\infty(\mathbb{S}_\xi^{n-2}, \pi^*\E)$ and $f' \in \otimes^{m'}_S E^* \otimes \E$, and all $m \in \Z_{\geq 0}$, we have:
	\begin{equation}\label{eq:useful-identity}
		C_{m + m' + n - 1} \int_{\mathbb{S}^{n-2}_\xi} \langle{f(u), \pi_{m'}^*f'(u)}\rangle_{\E}. dS_\xi(u) = \langle{\pi_{\ker_{\imath_{\xi^\sharp}}}{\pi_{m'}}_*E^m_\xi(f), f'}\rangle_{\otimes_S^{m'}E^* \otimes \E}.
	\end{equation}
\end{lemma}
\begin{proof}
	The left hand side of \eqref{eq:useful-identity} equals, after using \eqref{equation:normalization}:
	\begin{align*}
		&\int_{\mathbb{S}^{n-2}_\xi} \int_0^\pi \left\langle{\sin^{m}(\varphi)f(u), \pi_{m'}^* \pi_{\ker \imath_{\xi^\sharp}} f'\left(\cos(\varphi).\vec{n}(\xi) + \sin(\varphi).u\right)}\right\rangle_{\E}. \\
		& \hspace{8cm} \sin^{n - 2}(\varphi). d\varphi. dS_\xi(u)\\
		&= \int_{\mathbb{S}^{n - 1}} \langle{E_\xi^mf, \pi_{m'}^* \pi_{\ker \imath_{\xi^\sharp}} f'}\rangle_{\E}. dS = \langle{\pi_{\ker_{\imath_{\xi^\sharp}}}{\pi_{m'}}_*E^m_\xi(f), f'}\rangle_{\otimes_S^{m'}E^* \otimes \E},
	\end{align*}
	where in the first line we used that $\langle{\xi, u}\rangle = 0$ on $\mathbb{S}^{n - 2}_{\xi}$ by definition and in the second equality we used the Jacobian formula \eqref{eq:Jacobian}. This concludes the proof.
\end{proof}

We have the following result on the degree:

\begin{lemma}
\label{lemma:extension}
For all $\xi \in E^*\setminus \left\{ 0 \right\}$, the following holds. Let $f \in C^\infty(\Ss^{n-2}_\xi)$ such that $\Deg(f) \geq m+1$. Then, $\Deg(E^m_{\xi}(f)) \geq m+1$.
\end{lemma}

\begin{proof}
We argue by contradiction. We assume that $E^m_{\xi}(\WW)$ has degree $\leq m$. In particular, it is smooth. Moreover, observe that its $(m-1)$-th jet vanishes at the North pole $N := \left\{ \varphi=0\right\} = \vec{n}(\xi)$. We can therefore compute its differential of degree $m$. Let $0 \neq Z \in T_N \Ss^{n-1}$, $Z = r Z'$, where $|Z'|=1$. Note that in the $(\varphi,u)$-coordinates, $Z'$ corresponds to $\partial_{\varphi}$ at $(\varphi=0,u)$ where $u = Z' \in \Ss^{n-2}_{\xi}$. Then:
\begin{equation}\label{eq:mjet}
\begin{split}
\dd^m E^m_{\xi}(\WW)_N (Z, ..., Z) & = r^m \partial_\varphi^m (E^m_{\xi}(\WW)) (0,u) \\
&  = r^m \partial_\varphi^m\big|_{\varphi = 0}(\sin^m(\varphi) \WW|_{\Ss^{n-2}_{\xi}}(u)) = m! \times r^m \WW|_{\Ss_{\xi}^{n-2}}(u).
\end{split}
\end{equation}
There is a natural identification between $T_N \mathbb{S}^{n-1}$ and $\ker \xi = \vec{n}(\xi)^\perp$, and with this identification $\dd^m E^m_{\xi}(\WW)_N$ defines a symmetric $m$-tensor on $\ker \xi$. Then \eqref{eq:mjet} says $\WW|_{\Ss_{\xi}^{n-2}} = \frac{1}{m!} \pi_m^*(\dd^m E^m_{\xi}(\WW)_N)$ has degree $\leq m$, which is a contradiction.
\end{proof}

\subsection{Multiplication of spherical harmonics}

We end this section with standard results on multiplication of spherical harmonics:

\begin{lemma}
\label{lemma:multiplication}
Let $m, k \in \Z_{\geq 0}$ and assume without loss of generality that $m \geq k$. If $f \in \Omega_m(E), f' \in \Omega_{k}(E)$, then:
\[
f \times f' \in \oplus^k_{\ell = 0} \Omega_{m+k-2 \ell}(E).
\]
\end{lemma}

\begin{proof}
First of all, extending $f$ and $f'$ by $m$- and $k$-homogeneity to $E$, respectively, we directly see that $f \times f'$ is a homogeneous polynomial of degree $m+k$ and so by \eqref{eq:harmonicpolys}:
\[
f \times f' \in \oplus_{\ell \geq 0} |v|^{2 \ell} \mathbf{H}_{m+k-2 \ell}(E).
\]
The only non-trivial part is to show that the projection onto
\[
\oplus_{\ell \geq k+1} |v|^{2\ell} \mathbf{H}_{m+k-2 \ell}(E)
\]
is zero. For that it suffices to show that $\Delta^{\ell}(f \times f') = 0$ as long as $\ell \geq k+1$. Observe that:
\[
\Delta(f \times f') = (\Delta f) \times f' + 2 \nabla f \cdot \nabla f' + f \times (\Delta f') = 2 \sum_{i=1}^n \partial_{v_i}f \times \partial_{v_i}f',
\]
and thus by iteration:
\[
\Delta^{\ell}(f \times f') = 2^{\ell} \sum_{|\alpha| = \ell} \partial^\alpha_v f \times \partial^\alpha_v f',
\]
which clearly vanishes for $\ell \geq k+1$ as $f'$ is a polynomial of degree $k$.
\end{proof}

In the particular case where $k=1$, the previous lemma shows that $f \in \Omega_1(E)$ gives rise to two operators $f_\pm$ defined in the following way: if $\WW \in \Omega_m(E)$, then $f \times \WW = f_-\WW + f_+ \WW$ with $f_\pm \WW \in \Omega_{m \pm 1}(E)$. Moreover, by extending $f$ and $\WW$ as $1$- and $m$-homogeneous harmonic polynomials denoted by the same letter, we get ($\nabla$ denotes the total gradient of $E$)
\begin{equation}\label{eq:f_-}
	f_- \WW = \frac{1}{n + 2(m - 1)} \big(\nabla f \cdot \nabla \WW\big)|_{\mathbb{S}^{n-1}}.
\end{equation}
In fact, for non-zero $f$ the map $f_-: \Omega_m(E) \to \Omega_{m - 1}(E)$ is surjective, implying also that $f_+: \Omega_m(E) \to \Omega_{m + 1}(E)$ is injective (see \cite[Lemma 2.3]{Cekic-Lefeuvre-20}).

\begin{lemma}
\label{lemma:multiplication-surjective}
Assume $n \geq 2$ and let $m, k \in \Z_{\geq 0}$. Consider $\WW \in C^\infty(\Ss^{n-1})$ such that $\mathrm{deg}(\WW) \geq m$. Then, there exists $f \in \mathbf{S}_k(E)$ such that $\mathrm{deg}(f. \WW) \geq m+k$.
\end{lemma}

Equivalently, there exists $f \in \otimes^k_S E^*$ such that $\mathrm{deg}(\pi_k^*f. \WW) \geq m+k$. 

\begin{proof}
We first prove the case $k=1$; the cases $m = 0$ or $k = 0$ are trivial so we assume $m \geq 1$ from now on. We write $\WW = \sum_{j = 0}^\infty \WW_j$, where $\WW_j \in \Omega_j$, and denote by the same letter the harmonic extension of $\WW_j$ (as a $j$-homogeneous polynomial) to $\mathbb{R}^n$. Take $\ell \geq m$ such that $\WW_{\ell} \neq 0$, and assume for any $i = 1, \dotso, n$ that $(v_i)_+ \WW_{\ell} + (v_i)_- \WW_{\ell + 2} = 0$, which by \eqref{eq:f_-} is equivalent to:
	 \[(n + 2(\ell + 1))^{-1} \partial_{v_i} \WW_{\ell + 2} + v_i \WW_{\ell} - |v|^2 (n + 2(\ell - 1))^{-1} \partial_{v_i} \WW_{\ell} = 0.\]
	 Multiplying by $v_i$ and summing over $i$, we obtain using Euler's formula (i.e. homogeneity)
	 \[-\frac{\ell + 2}{n + 2(\ell + 1)} \WW_{\ell + 2} = \frac{n + \ell -2}{n + 2(\ell -1)}  |v|^2 \WW_{\ell}.\]
	Applying $\Delta$, this contradicts the fact that $\WW_{\ell} \neq 0$.	 
	 
For general $k \in \Z_{\geq 1}$, by iteratively applying the case $k = 1$ above, there exist	 $f_1, ...,f_k \in \Omega_1(E)$ such that $\mathrm{deg}(f_k \cdots f_1 \WW) \geq m+k$. Since $f := f_k \cdots f_1 \in \mathbf{S}_k(E)$, this completes the proof.
\end{proof}

Note that there is a straightforward extension to the bundle case (just by applying the previous lemma coordinate-wise), that is, when considering sections of a trivial bundle $\pi^*\E \to \Ss^{n-1}$, where $\pi : \Ss^{n-1} \to \left\{0\right\}$ is the constant map. We record it here and leave the proof as an exercise for the reader:

\begin{lemma}
\label{lemma:multiplication-surjective-bundle}
Let $m, k \in \Z_{\geq 0}$. Consider $\WW \in C^\infty(\Ss^{n-1},\pi^*\E)$ such that $\mathrm{deg}(\WW) \geq m$. Then, there exists $f \in \mathbf{S}_k(E) \otimes \End_{\mathrm{sk}}(\E)$ such that $\mathrm{deg}(f. \WW) \geq m+k$.
\end{lemma}

\section{Pseudodifferential nature of perturbed generalized X-ray transforms}

\label{section:pdo}

Under a weaker form, the results of this section can be found in \cite{Guillarmou-17-1,Gouezel-Lefeuvre-19} and in \cite[Chapter 2]{Lefeuvre-thesis} where the principal symbol of the generalized X-ray transform is computed in details. We here need a more general result where we ``sandwich" (pseudo)differential operators (we recall that the constants $C_\bullet$ were defined in \eqref{equation:normalization}):

\begin{proposition}
\label{proposition:sandwich}
Let $P_R,P_L : C^\infty(SM,\pi^* \E) \rightarrow C^\infty(SM,\pi^*\E)$ be differential operators of degree $m_R,m_L \geq 0$ and fix $m_1,m_2 \in \N_0$. Then the operator
\[
A_{P_R,P_L} := {\pi_{m_1}}_* P_L \mc{I}_{\nabla^{\E}} P_R \pi_{m_2}^*,
\]
is a classical\footnote{See below \eqref{equation:pdo-expansion} for a definition.} pseudodifferential operator of order $m := m_R + m_L -1$ in
\[
A_{P_R,P_L} \in \Psi^m(M,\otimes^{m_2}_S T^*M \otimes \E \rightarrow \otimes^{m_1}_S T^*M \otimes \E).
\]
Moreover, its principal symbol satisfies, for any $f \in \otimes^{m_2}_S T_x^*M \otimes \E_x$ and $f' \in \otimes^{m_1}_S T_x^*M \otimes \E_x$:
\begin{equation}\label{eq:symbol-1}
\begin{split}
& \langle \sigma_{A_{P_R,P_L}}(x,\xi) f, f' \rangle_{\otimes^{m_1}_S T^*M_x \otimes \E_x} \\
& = \dfrac{2 \pi}{|\xi|} \int_{\Ss^{n-2}_\xi} \left\langle \sigma_{P_R}\big((x,u),\xi_{\HH}(x,u)\big) (\pi_{m_2}^* f(u)), \sigma_{P^*_L}\big((x,u),\xi_{\HH}(x,u)\big)(\pi_{m_1}^* f'(u)) \right\rangle_{\E_x} \\
& \hspace{10cm} \dd S_\xi(u),
\end{split}
\end{equation}
where $\xi_{\HH}(x,u) := \xi\left(\dd_{x,u} \pi(\bullet)\right)$. More explicitly, the principal symbol of $A_{P_R, P_L}$ is given by the formula, for any $m \in \mathbb{N}_0$:
\begin{equation}
\begin{split}
\label{eq:symbol-2}
	& \sigma_{A_{P_R, P_L}}(x, \xi) f \\
	& = \frac{2\pi}{|\xi|} C_{m_1 + m + n - 1}^{-1} \pi_{\ker \imath_{\xi^\sharp}} {\pi_{m_1}}_* E_\xi^m \Big[\sigma_{P_L P_R} \big(x, u, \xi_{\HH}(x, u)\big)
	 (\pi_{m_2}^* \pi_{\ker \iota_{\xi^\sharp}} f)\Big].
	 \end{split}
\end{equation}
\end{proposition}

\begin{remark}\rm This was originally proved with $P_L = P_R = \mathbbm{1}$ in \cite{Guillarmou-17-1}, see also \cite{Lefeuvre-thesis}. Note that one could actually take $P_R,P_L$ in Proposition \ref{proposition:sandwich} to be pseudodifferential of arbitrary order (that only makes the proof slightly longer but the idea is the same).
\end{remark}

In the following, we will refer to this result as the sandwich Proposition \ref{proposition:sandwich}. In the case where $\E = \C \times M$, the formula reads (using that $\sigma_{P_L^*} = \overline{\sigma_{P_L}}$):
\[
\begin{split}
&\langle \sigma_{A_{P_R,P_L}}(x,\xi) f, f' \rangle_{\otimes^{m_1}_S T^*M_x \otimes \E_x}\\
& = \dfrac{2 \pi}{|\xi|} \int_{\Ss^{n-2}_\xi}  \sigma_{P_R}\big((x,u),\xi_{\HH}(x,u)\big) \sigma_{P_L}\big((x,u),\xi_{\HH}(x,u)\big). \pi_{m_2}^* f(u). \overline{\pi_{m_1}^* f'(u)} \\
& \hspace{11cm} \dd S_\xi(u).
\end{split}
\]
We will only prove Proposition \ref{proposition:sandwich} in the case of the trivial line bundle with trivial connection in order to simplify the discussion; the generalization to the twisted case is straightforward modulo some tedious notation. We also make the following important remark:

\begin{remark}
\label{remark:lift}
Proposition \ref{proposition:sandwich} can also be generalized by considering differential operators
\[
\begin{split}
& P_R : C^\infty(SM, \otimes^{m_2}_S T^*(SM)) \to C^\infty(SM), \\
& P_L : C^\infty(SM) \rightarrow C^\infty(SM, \otimes^{m_1}_S T^*(SM)),
\end{split}
\]
(of degree $m_R,m_L \geq 0$) and looking at the operator:
\[
A_{P_R,P_L} := {\pi_{m_1,\mathrm{Sas}}}_* P_L \mc{I} P_R \pi_{m_2,\mathrm{Sas}}^*,
\]
where $\pi_{m_2,\mathrm{Sas}}^*$ denotes the Sasaki lift introduced in \eqref{equation:lift-sasaki}. The same proof shows that this operator is pseudodifferential of order $m_R + m_L -1$ with principal symbol satisfying:
\begin{equation}
\begin{split}
& \langle \sigma_{A_{P_R,P_L}}(x,\xi) f, f' \rangle_{\otimes^{m_1}_S T^*M_x} \\
& = \dfrac{2 \pi}{|\xi|} \int_{\Ss^{n-2}_\xi}  \sigma_{P_R}\big((x,u),\xi_{\HH}(x,u)\big) (\pi_{m_2,\mathrm{Sas}}^* f(u)). \overline{\sigma_{P^*_L}\big((x,u),\xi_{\HH}(x,u)\big)(\pi_{m_1,\mathrm{Sas}}^* f'(u))} \\
& \hspace{11cm} \dd S_\xi(u),
\end{split}
\end{equation}
where $\xi_{\HH}(x,u) := \xi\left(\dd_{x,u} \pi(\bullet)\right)$. We leave this claim as an exercise for the reader.
\end{remark}

First of all, let us fix $\varepsilon > 0$ and a cut off function $\chi \in C_0^\infty(\mathbb{R}; [0, 1])$, symmetric around zero, such that 
\[
	\chi(t) = 1,\,\, t \in \left[-\frac{\varepsilon}{2}, \frac{\varepsilon}{2}\right], \quad \chi(t) = 0,\,\, |t| \geq \varepsilon.
\]
We set (here $e^{tX} = \varphi_t^*$):
\[
	\mc{I}_\eps := \int_{-\eps}^{+\eps} \chi(t) e^{tX} \dd t.
\]

\begin{lemma}
\label{lemma:simplification}
The following property holds:
\[
{\pi_{m_1}}_* P_L \mc{I} P_R \pi_{m_2}^* - {\pi_{m_1}}_* P_L \mc{I}_\eps P_R \pi_{m_2}^* \in \Psi^{-\infty}(M, \otimes^{m_2}_S T^*M \rightarrow \otimes^{m_1}_S T^*M).
\]
\end{lemma}

\begin{proof}
Recall that $\mc{I} = R_{+,0} + R_{-,0} + \Pi_{+}$ and the term $\Pi_{+}$ will only contribute to a smoothing operator. We first derive an auxiliary identity; we start by the following
\[
	\int_0^\infty \chi'(t) e^{-t(z + X)} \dd t = -\id + \left(\int_0^\infty \chi(t) e^{-t(z + X)} \dd t\right) \circ (X + z),\quad z \in \mathbb{C},
\]
where we integrated by parts in the equality. Composing on the right by $(X + z)^{-1}$ with $z \neq 0$ close to zero, using the meromorphic extension and taking the bounded terms at $z = 0$ we get:
\begin{equation}\label{eq:resolvent-integrate-by-parts}
	R_{+, 0} = \int_0^\infty \chi(t) e^{-tX} \dd t - \left(\underbrace{\int_0^\infty \chi'(t) e^{-tX} \dd t}_{\mc{R}:=}\right) \circ R_{+, 0} + \left(\int_0^\infty t\chi'(t) \dd t\right) \cdot \Pi_+,
\end{equation}
where for the last term we used that $\varphi_{-t}^* \Pi_+ = \Pi_+$ since $\Pi_+$ is the orthogonal projection onto constant functions; integrating by parts the multiplier in the last term simplifies to $\int_0^\infty t\chi'(t) \dd t = - \int_0^\infty \chi(t) \dd t$. Using the analogous formula for $R_{-, 0}$, and that $\Pi_+$ is smoothing, we see it suffices to prove that the middle term of \eqref{eq:resolvent-integrate-by-parts} contributes to a smoothing operator, that is, 
\[ K := {\pi_{m_1}}_* P_L \mc{R} R_{+, 0} P_R \pi_{m_2}^* \in \Psi^{-\infty}(M, \otimes^{m_2}_S T^*M \rightarrow \otimes^{m_1}_S T^*M).\]

It is sufficient to prove that if $f \in \mc{D}'(M, \otimes^{m_2}_S T^*M)$, one has $Kf \in C^\infty(M,\otimes^{m_1}_S T^*M)$. For that, we will use the wavefront set calculus of Hörmander \cite[Chapter 8]{Hormander-90}.

Using the notation of \S \ref{ssection:elementary}, define the subbundles $\HH^*,\V^* \subset T^*(SM)$ such that $\HH^*(\HH \oplus \R \cdot X) = 0, \V^*(\V) = 0$. Observe that since $\pi_{m_2}^*$ is a pullback operator, we have $\WF(\pi_{m_2}^* f) \subset \V^*$ (see also \cite[Lemma 2.1]{Lefeuvre-19-1} for a detailed proof). Since $P_R$ is a differential operator, we have $\WF(P_R \pi_{m_2}^* f) \subset \V^*$. We then use the characterization of the wavefront set of the resolvent $R_{+,0}$ in \cite[Proposition 3.3]{Dyatlov-Zworski-16}, namely\footnote{We use the standard conventions, namely if $B : C^\infty(M) \rightarrow \mc{D}'(M)$ is a linear operator with kernel $K_B \in \mc{D}'(M \times M)$, we define $\WF'(B) := \left\{(x,\xi,y,\eta) \in T^*M \times T^*M ~|~ (x,\xi,y,-\eta) \in \WF(K_B)\right\}$.}:
\begin{equation}
\label{equation:dz}
\WF'(R_{+,0}) \subset \Delta(T^*(SM)) \cup \Omega_{+} \cup E_u^* \times E_s^*,
\end{equation}
where $\Delta(T^*(SM))$ is the diagonal in $T^*(SM) \times T^*(SM)$, and
\[
\Omega_{+} := \left\{ (\Phi_{t}(z,\xi),(z,\xi)) ~|~ t \geq 0,~ \langle \xi, X(z) \rangle = 0\right\}
\]
is the positive flow-out and $\Phi_t : T^*(SM) \rightarrow T^*(SM)$ is the symplectic lift of the geodesic flow $(\varphi_t)_{t \in \R}$, given by $\Phi_t(z,\xi) := (\varphi_t(z), \dd \varphi_t^{- \top}(z)(\xi))$\footnote{We use ${}^{-\top}$ to denote the inverse transpose.}. From \eqref{equation:dz} we obtain using \cite[Theorem 8.2.13]{Hormander-90}:
\begin{equation}\label{eq:wf-dz}
	\WF(R_{+,0}  P_R \pi_{m_2}^*f) \subset \mathbb{V}^* \cup E_u^* \cup_{t \geq 0} \Phi_t(\V^* \cap \ker \imath_X).
\end{equation}

Next, we show that since $\mc{R}$ is given by integration along the flow, it is microlocally smoothing outside $\ker \imath_X$ (i.e. it is smoothing in the elliptic set of $X$). For that, fix an arbitrary $k \in \mathbb{Z}_{\geq 1}$, and let $B \in \Psi^0(\M)$ be arbitrary microlocally equal to $\id$ near $\ker \imath_X$ (where $\M := SM$), that is, $\WF(\id - B)$ does not intersect a conical neighbourhood of $\ker \imath_X$. We will show that $\mc{R}(\id - B)$ is smoothing for any such $B$. By ellipticity, there are $E \in \Psi^{-k}(\M)$ and $F \in \Psi^{-\infty}(\M)$ such that
\[X^k E = \id - B + F.\]
Therefore we can compute, for an arbitrary $u \in \mc{D}'(\M)$ that
\[\mc{R}(\id - B) u + \mc{R}F u = \int_0^\infty \chi'(t) \varphi_{-t}^* X^k E u \dd t = \int_0^\infty (\partial_t^{k + 1}\chi)(t) \varphi_{-t}^* Eu \dd t,\]
where in the last equality we used that $\varphi_{-t}^* X^k = (-\partial_t)^k \varphi_{-t}^*$ and integrated by parts $k$ times. Since $\mc{R}F u \in C^\infty(\M)$, and $E \in \Psi^{-k}(\M)$ where $k$ could be chosen arbitrary, we conclude that $\mc{R}(\id - B) \in \Psi^{-\infty}(\M)$, proving the claim. Therefore, by \eqref{eq:wf-dz} and by the behaviour of the wavefront set under pullbacks (see \cite[Theorem 8.2.4]{Hormander-90}):
\begin{align*}
	&\WF(P_L \mc{R} R_{+,0} P_R \pi_{m_2}^*f) \subset \WF(\mc{R} R_{+,0}  P_R \pi_{m_2}^*f)\\
	&\hspace{150pt} \subset  \cup_{t \in [\varepsilon/2, \varepsilon]} \Phi_{t}(\mathbb{V}^* \cap \ker \imath_X) \cup E_u^* \cup_{t \in [\eps/2, \eps]} \Phi_t(\V^* \cap \ker \imath_X).
\end{align*}

Next, as ${\pi_{m_1}}_*$ is a pushforward, we obtain (see \cite[Proposition 4.12]{Melrose-03}):
\[
\begin{split}
& \WF(Kf) \\
& \subset \left\{ (x,\xi) \in T^*M ~|~  \exists v \in S_xM,\, \xi(\dd_{(x,v)} \pi(\bullet)) \in E_u^* \cup_{t \in [\eps/2, \eps]} \Phi_t(\V^* \cap \ker \imath_X)\right\},
\end{split}
\]
where we note that $\xi(\dd_{(x,v)} \pi(\bullet)) \in \V^*(x,v)$. As a consequence, the lemma is proved if we can show that $\V^* \cap \left(E_u^* \cup_{t \geq [\eps/2, \eps]} \Phi_t(\V^* \cap \ker \imath_X)\right) = \left\{ 0 \right\}$. But this follows from \eqref{equation:transverse} and \eqref{equation:conjugate}, completing the proof.
\end{proof}

We now turn to the sandwich Proposition \ref{proposition:sandwich}. For that, it is convenient to use the historical characterization of pseudodifferential operators \cite[Definition 2.1]{Hormander-65} which we now recall: $P$ is a pseudodifferential operator of order $m \in \R$ if $P:C^\infty(M) \to C^\infty(M)$ is continuous, and there exists a sequence $s_0=0 < s_1 < ...$ of real numbers converging to $+\infty$ such that for all $f \in C^\infty(M),\, S \in C^\infty(M)$ such that $\dd S \neq 0$ on $\mathrm{supp}(f)$, there is an asymptotic expansion:
\begin{equation}
\label{equation:pdo-expansion}
e^{-i \frac{S}{h}} P(f e^{i \frac{S}{h}}) \sim h^{-m} \sum_{j=0}^{+\infty} P_j(f,S) h^{s_j}.
\end{equation}
By this, we mean that for every integer $N > 0$, for every compact set \footnote{A set $A \subset C^\infty(M)$ is bounded if there exists a sequence $(A_k)_{k \in \Z_{\geq 0}}$ such that for all $f \in A, \|f\|_{C^k(M)} \leq A_k$. It is known that $A$ is compact if and only if it is closed and bounded.}
$K$ of real-valued functions $S \in C^\infty(M)$ with $\dd S \neq 0$ on $\mathrm{supp}(f)$, for every $0 < h < 1$, the following holds: the error term
\begin{equation}
\label{equation:negligible}
h^{-s_N+m}\left(e^{-i \frac{S}{h}} P\left(f e^{i \frac{S}{h}}\right) - h^{-m}\sum_{j=0}^{N-1} P_j(f,S)h^{s_j}\right)
\end{equation}
belongs to a bounded set in $C^\infty(M)$ with bound independent of $h$ and $S \in K$. In particular, $P$ is \emph{classical} if and only if in the sum \eqref{equation:pdo-expansion} the $s_j$'s take integer values.

\begin{proof}[Proof of Proposition \ref{proposition:sandwich}] We first note that the formula \eqref{eq:symbol-2} is an immediate consequence of \eqref{eq:symbol-1} and Lemma \ref{lemma:useful-identity}; henceforth we focus on \eqref{eq:symbol-1}. We divide the proof in two steps. \\

\emph{1. Principal symbol computation.} For the moment, let us assume that the operator is pseudodifferential and compute its principal symbol. By Lemma \ref{lemma:simplification}, we can replace $\mc{I}$ by $\mc{I}_\eps$ in the definition of the operator, that is, it suffices to compute the principal symbol of ${\pi_{m_1}}_* P_L \mc{I}_\eps P_R \pi_{m_2}^*$.

Take a Lagrangian state $f_h := e^{i \frac{S}{h}} \widetilde{f}$, where $S \in C^\infty(M)$ is a real-valued, smooth phase such that $S(x_0)=0,\, \dd S(x_0) = \xi_0$, $\widetilde{f} \in C^\infty(M,\otimes^{m_2}_S T^*M)$ and $\widetilde{f}(x_0) := f \in\otimes^{m_2}_S T^*_{x_0}M$, and further assume that $\dd S$ does not vanish on the support of $\widetilde{f}$. As $P_R$ is a differential operator, we have:
\[
\begin{split}
P_R \pi_{m_2}^* f_h & = P_R\left(e^{i \frac{\pi_0^*S}{h}} \pi_{m_2}^* \widetilde{f}\right) \\
& = h^{-m_R} e^{i \frac{\pi_0^*S}{h}}\left(\sigma_{P_R}\big(\bullet, S_\HH(\bullet)\big). \pi_{m_2}^*\widetilde{f}(\bullet) + \mc{O}_{C^\infty}(h)\right),
\end{split}
\]
where $S_{\HH}(x, v) := \dd_x S \circ \dd_{(x,v)} \pi$. Hence:
\[
\begin{split}
&\mc{I}_\eps P_R \pi_{m_2}^* f_h (x,v)\\
& =  \int_{-\eps}^{+\eps} h^{-m_R} e^{\frac{i}{h} S(\pi(\varphi_t(x,v)))} \chi(t)\\
&\times  \left(\sigma_{P_R}\big(\varphi_t(x,v), S_\HH(\varphi_t(x,v))\big) \pi_{m_2}^* \widetilde{f}(\varphi_t(x,v)) + \mc{O}_{C^\infty}(h)\right) \dd t,
\end{split}
\]
and thus:
\[
\begin{split}
& P_L \mc{I}_\eps P_R \pi_{m_2}^* f_h (x,v) \\
& = \int_{-\eps}^{+\eps} h^{-(m_R+m_L)} e^{\frac{i}{h} S(\pi(\varphi_t(x,v)))} \chi(t)\Big( \sigma_{P_R} \big( \varphi_t(x,v), S_\HH(\varphi_t(x,v)) \big)  \\
& \hspace{2cm} \times \sigma_{P_L}\big((x,v), S^{(t)}_\HH(x,v)\big) \pi_{m_2}^* \widetilde{f}(\varphi_t(x,v)) + \mc{O}_{C^\infty}(h)\Big) \dd t,
\end{split}
\]
where $S^{(t)}_\HH(x,v) := \dd_x S \circ \dd_{(x, v)} \pi \circ \dd_{(x,v)} \varphi_t $ (and $S^{(0)}_{\HH} = S_{\HH}$). This gives for $m=m_R+m_L-1$ and any $f' \in \otimes_S^{m_1}T_{x_0}^*M$:
\begin{equation}\label{eq:aux00}
\begin{split}
& \langle \sigma_{A_{P_R,P_L}}(x_0,\xi_0) f, f' \rangle_{\otimes^{m_1}_S T_{x_0}^*M} =  \lim_{h \rightarrow 0} h^{m} \langle (A_{P_R,P_L} f_h)(x_0), f' \rangle_{\otimes^{m_1}_S T_{x_0}^*M} \\
& =  \lim_{h \rightarrow 0}   h^{m} \int_{S_{x_0}M} (P_L \mc{I}_\eps P_R \pi_{m_2}^* f_h)(x_0,v). \overline{\pi_{m_1}^* f'}(x_0,v) \dd S_{x_0}(v) \\
& = \lim_{h \rightarrow 0}   h^{-1} \int_{S_{x_0}M} \int_{-\eps}^{+\eps} e^{\frac{i}{h} S(\pi(\varphi_t(x_0,v)))}\chi(t). \overline{\pi_{m_1}^* f'}(v) \\
& \times \left[ \sigma_{P_R} \big( \varphi_t(x_0,v), S_\HH(\varphi_t(x_0,v)) \big).  \sigma_{P_L}\big((x_0,v), S^{(t)}_\HH(x_0,v)\big). \pi_{m_2}^* \widetilde{f}(\varphi_t(x_0,v)) \right. \\
& \hspace{8cm} \left. + \mc{O}_{C^\infty}(h)\right]  \dd t\dd S_{x_0}(v),
\end{split}
\end{equation}
where $\dd S_{x_0}$ stands for the round measure on the sphere $S_{x_0}M$. As we shall see, the term $h^{-1}$ comes from the fact that we will perform a stationary phase over a two-dimensional space.

We define the (real) phase $\Phi : (-\eps,\eps) \times S_{x_0}M  \rightarrow \R$ by
\begin{equation}\label{eq:phase}
	\Phi(t,v) := S(\pi(\varphi_t(x_0,v))).
\end{equation}
We recall (see \S \ref{ssection:restriction}) the diffeomorphism, singular at the poles, $\Ss_{\xi_0}^{n-2} \times (0,\pi) \ni (u,\varphi) \mapsto v(u,\varphi) := \cos(\varphi) \vec{n}(\xi_0) + \sin(\varphi) u \in S_{x_0}M$. Observe that for fixed $u \in \Ss_{\xi_0}^{n-2}$, the phase $\Phi_u : (-\eps,\eps) \times (0,\pi) \rightarrow \R$ defined by $\Phi_u(t,\varphi) := \Phi(t, v(u, \varphi))$ has a critical point at $t=0,\varphi=\frac{\pi}{2}$ and the determinant of the Hessian at this point is equal to $-|\dd S(x_0)|^2 = -|\xi_0|^2$ (see the proof of \cite[Theorem 4.4]{Gouezel-Lefeuvre-19} or \cite[Theorem 2.5.1]{Lefeuvre-thesis} for further details). Hence by the stationary phase lemma \cite[Theorem 3.16]{Zworski-12}, for any $u \in \mathbb{S}^{n-2}_{\xi_0}$, writing $v = v(u, \varphi)$:
\begin{equation}\label{eq:stationary-phase}
\begin{split}
&\lim_{h \rightarrow 0}  h^{-1}   \int_0^\pi \int_{-\eps}^{+\eps} e^{\frac{i}{h} \Phi_u(t,\varphi)}\chi(t).\overline{\pi_{m_1}^* f'}\big(v\big). \Big(\mc{O}_{C^\infty}(h) +\\
&+ \sigma_{P_R} \big( \varphi_t(x_0,v), S_\HH(\varphi_t(x_0,v))\big).  \sigma_{P_L}\big((x_0,v), S^{(t)}_\HH(x_0,v)\big). \pi_{m_2}^* \widetilde{f}(\varphi_t(x_0,v))\Big) \\
& \hspace{8cm} \sin^{n-2}(\varphi) \dd t  \dd \varphi\\
& = \dfrac{2 \pi}{|\xi_0|} \sigma_{P_R} \big( (x_0,u), S_\HH(x_0,u)\big).  \sigma_{P_L}\big((x_0,u), S_\HH(x_0,u) \big). \pi_{m_2}^*f(u).  \overline{\pi_{m_1}^* f'}(u).
\end{split}
\end{equation}
Using that this limit is uniform in $u$ and integrating over $\mathbb{S}^{n-2}_{\xi_0}$, inserting into \eqref{eq:aux00}, as well as recalling the Jacobian formula \eqref{eq:Jacobian}, completes the proof.\\

\emph{2. Pseudodifferential nature.} By the characterization \eqref{equation:pdo-expansion} of $\Psi$DOs via the asymptotic expansion, the proof is very similar to the first point except that one needs to go to arbitrary order in the expansions. For the sake of simplicity, we assume that $m_1 = m_2 = 0$ (this does not change the nature of the proof). By Lemma \ref{lemma:simplification}, it suffices to show that ${\pi_{0}}_* P_L \mc{I}_\eps P_R \pi_{0}^*$ is a pseudodifferential operator of order $m$. Consider an arbitrary $f \in C^\infty(M)$ and a compact set $K \subset C^\infty(M)$ of (real) phases $S$ such that $dS \neq 0$ on $\supp(f)$. 

Since $P_R$ is differential, we can write
\[
P_R \pi_0^*\left(f e^{i\frac{S}{h}}\right) = h^{-m_R} e^{i\frac{\pi_0^*S}{h}} \underbrace{\sum_{j=0}^{m_R} P_R^{(j)}(\pi_0^*f,\pi_0^*S) h^{j}}_{=:f_h} = h^{-m_R} e^{i \frac{\pi_0^* S}{h}} f_h,
\]
where $P_R^{(j)}(\pi_0^*f,\pi_0^*S)(x,v)$ depends on the jet of order $\leq j$ of $f$ at $x$ (and on the $(m_R-j)$-th jet of the phase $S$). Then:
\[
\mc{I}_{\eps} P_R \pi_0^*\left(f e^{i\frac{S}{h}}\right)(x,v) = h^{-m_R} \int_{-\eps}^{+\eps} e^{\frac{i}{h}\pi_0^*S(\varphi_t(x,v))} \chi(t) f_h(\varphi_t(x,v)) \dd t,
\]
which gives:
\[
\begin{split}
& P_L \mc{I}_{\eps}  P_R \pi_0^*\left(f e^{i\frac{S}{h}}\right)(x,v) \\
& = h^{-(m_R+m_L)} \int_{-\eps}^{+\eps} e^{\frac{i}{h}\pi_0^*S(\varphi_t(x,v))} \chi(t) \sum_{k=0}^{m_L} h^k P^{(k)}_L\left(e^{tX} f_h, e^{tX} \pi_0^*S\right)(x,v) \dd t,
\end{split}
\]
and thus:
\begin{equation}\label{eq:F-expression}
\begin{split}
& {\pi_0}_* P_L \mc{I}_{\eps}  P_R \pi_0^*\left(f e^{i\frac{S}{h}}\right)(x) \\
&= h^{-(m_R+m_L)} \int_{S_{x}M} \int_{-\eps}^{+\eps} e^{\frac{i}{h} \pi_0^*S(\varphi_t(x,v))} \chi(t) \sum_{k=0}^{m_L} h^k P^{(k)}_L\left(e^{tX} f_h, e^{tX} \pi_0^*S\right)(x,v) \dd t \dd v.
\end{split}
\end{equation}

We now split to cases according to the location of $x$ and the value of $dS(x)$ as follows. Since $K \subset C^\infty(M)$ is compact, and $dS \neq 0$ on $\supp(f)$, there is an open neighbourhood $N$ of $\supp(f)$ and $\delta = \delta(K) > 0$ such that 
\[|dS(x)| \geq \delta, \quad \forall x \in N.\] 
We introduce $\xi := \dd S(x)$ and first consider the case $x \in N$. We will use the coordinates $(u, \varphi) \in \Ss^{n-2}_\xi \times (0,\pi)$ on $S_{x}M$ as in the previous step, and write $\Phi_u(t, \varphi)$ for the phase introduced in \eqref{eq:phase}. It is possible to compute the derivatives of $\Phi_u$ at $t = 0$ as follows:
\begin{align*}
	&\partial_{t}\Phi_u(0, \varphi) = \cos(\varphi) |dS(x)|, \quad \partial_\varphi \Phi_u(0, \varphi) = 0,\\
	&\partial_\varphi \partial_{t} \Phi_u(0, \varphi) = -\sin(\varphi) |dS(x)|, \quad \partial_\varphi^2 \Phi_u(0, \varphi) = 0, \quad \partial_t^2\Phi_u(0, \tfrac{\pi}{2}) = \mathrm{Hess}(S)_{(0, \tfrac{\pi}{2})}(u, u),
\end{align*}
where $\mathrm{Hess}(S)$ denotes the Hessian of $S$ in the $(t, \varphi)$ coordinates. Therefore, the only critical point of $\Phi_u$ on $\{t = 0\}$ is at $(t, \varphi) = (0, \frac{\pi}{2})$, and here the Hessian of $\Phi_u$ is non-degenerate. It follows that $(0, \frac{\pi}{2})$ is an isolated critical point, and moreover using Taylor's theorem that there is a $C > 0$ depending only on $K$ such that the derivative $d\Phi_u(t, \varphi)$ vanishes only at $(0, \frac{\pi}{2})$ for 
\[
	(t, \varphi) \in \mc{S} := \left\{(t, \varphi) \in (-\eps, \eps) \times (0, \pi) \mid |t| + |\tfrac{\pi}{2} - \varphi| < C|dS(x)|^2\right\}.
\] 
Let $\psi \in C_0^\infty(\mathbb{R})$ be a cut off function such that $\psi(t) = 1$ for $|t| < \beta := \min(\tfrac{\eps}{4}, \tfrac{\pi}{4}, \tfrac{C}{4} |dS(x)|^2)$ and $\psi(t) = 0$ for $|t| > 2\beta$, such that it is bounded uniformly in $C^\infty(\mathbb{R})$ depending on $S \in K$; here it is important to note that by assumption $|dS(x)| \geq \delta(K) > 0$. We also note that this construction of $\psi$ can be made to depend smoothly on $x \in U$ for an open set $U \subset N$ of points close to $x$; we note that $\varphi$ in this case encodes the distance to the equator $\ker(dS(x)) \cap S_xM$. For simplicity, we drop $x$ from the notation of $\psi$.

Using the formula \eqref{eq:Jacobian} and writing $v = v(u, \varphi)$ we obtain:
\[
\begin{split}
& {\pi_0}_* P_L  \mc{I}_{\eps}  P_R \pi_0^*\left(f e^{i\frac{S}{h}}\right)(x) =h^{-(m_R+m_L)}  \\
&\times \int_{\Ss^{n-2}_{\xi}} \int_0^{\pi} \int_{-\eps}^{+\eps} e^{\frac{i}{h}\pi_0^*S(\varphi_t(x,v))} \chi(t) \big(\psi(t) \psi(\varphi - \tfrac{\pi}{2}) + (1 - \psi(t)\psi(\varphi - \tfrac{\pi}{2}))\big)\\ 
&\times\underbrace{\sum_{k=0}^{m_L}  h^k P^{(k)}_L\left(e^{tX} f_h, e^{tX} \pi_0^*S\right)(x,v)}_{F :=} \sin^{n-2}(\varphi) \dd t  \dd \varphi \dd S_\xi(u)~ \\
& = h^{-(m_R+m_L)} \int_{\Ss^{n-2}_{\xi}} \left(\int_0^{\pi} \int_{-\eps}^{+\eps} e^{\frac{i}{h}\pi_0^*S(\varphi_t(x,v(u,\varphi)))}  \sum_{i = 1}^2 F_i(x,u,\varphi,t) \dd t \dd \varphi\right) \dd S_\xi(u),
\end{split}
\]
where the terms $F_i(x, u, \varphi, t)$ for $i = 1, 2$ represent the terms appearing in the second  line (in the order of appearance). We study each term separately. For $F_1$ (which may be also seen as a smooth function on $(-\eps, \eps) \times SU$), as in \eqref{eq:stationary-phase}, for $x \in U$ and $\xi = dS(x)$, and $u \in \Ss^{n-2}_\xi$, we apply the stationary phase lemma \cite[Theorem 3.16]{Zworski-12} at $t=0$, $\varphi=\frac{\pi}{2}$, which gives:
\begin{equation}
\label{equation:truth}
\begin{split}
&h^{-(m_R+m_L-1)} \int_{\Ss^{n-2}_{\xi}} e^{i\frac{S(x)}{h}} \left(\sum_{\ell=0}^{N-1} h^\ell A_{2\ell}(x,u,D_{\varphi},D_t) F_1\left(x,u,\frac{\pi}{2}, 0\right) + h^N R'_N(x,u) \right) \dd S_\xi(u) \\
& =  h^{-(m_R+m_L-1)} e^{i\frac{S(x)}{h}} \\
& \times \left( \sum_{\ell=0}^{N-1}h^\ell \int_{\Ss^{n-2}_\xi} A_{2\ell}(x,u,D_{\varphi},D_t)  F_1\left(x,u,\frac{\pi}{2},0\right) \dd S_\xi(u) + h^NR_N(x)\right).
\end{split}
\end{equation}
Here, for any $\ell \in \N$, $A_{2\ell}(x,u,D_\varphi,D_t)$ is a differential operator of degree $\leq 2\ell$ depending smoothly on $x \in U$ and $u$ and $R'_N$ satisfies the bound
\[
	\|R_N'\|_{C^0(SU \cap \ker dS)} \leq C_N' \|F_{1}\|_{C^{2N+3}( (-\eps, \eps) \times S U)},
\]
where $SU$ denotes the unit tangent bundle of $U$ (where $F_{1}$ is defined). The order $2N+3 = 2N + 2 + 1$ comes from the remainder term in \cite[Theorem 3.16]{Zworski-12}. After integration in the variable $u$, i.e. setting $R_N(x) = \int_{\mathbb{S}_{\xi}^{n-2}} R'_N(x, u)\dd S_{\xi}(u)$, this gives:
\begin{equation}
\label{equation:rn}
\|R_N\|_{C^0(U)} \leq C_N \|F_{1}\|_{C^{2N+3}(SU)},
\end{equation}
and one can control higher order derivatives of $R_N$ in the same fashion (up to increasing the order of the norm on the right-hand side of \eqref{equation:rn}). 

Next, observe that by definition:
\[
\begin{split}
& F_{1}(x,u,\varphi,t) \\
& =  \sin^{n-2}(\varphi) \chi(t) \psi(t) \psi(\tfrac{\pi}{2} - \varphi) \sum_{k=0}^{m_L} P^{(k)}_L\left(e^{tX}  \sum_{j=0}^{m_R} P_R^{(j)}(\pi_0^*f,\pi_0^*S) h^{j}, e^{tX} \pi_0^*S\right)(x,v(u,\varphi)) h^k,
\end{split}
\]
which implies that the $C^k$-norms of $F_{1}$ are controlled by the $C^{k'}$-norms of $f$ and $S$ (for some $k' \gg k$), that is, the remainder $R_N$ is indeed negligible in the sense of \eqref{equation:negligible}. Also note that $A_{2\ell}(x,u,D_{\varphi},D_t)F_{1}(x,u,\frac{\pi}{2},0)$ depends only on a finite number $K(\ell)$ of derivatives of the function $f$ and the phase $S$ at $x$. Hence \eqref{equation:truth} shows that the term corresponding to $F_1$ has the correct asymptotic expansion.

For the term $F_2$ in \eqref{equation:truth}, it is more convenient to go one step back and to write it as
\begin{equation}\label{eq:F_2}
	h^{-(m_R+m_L)} \int_{S_{x}M} \int_{-\eps}^{+\eps} e^{\frac{i}{h} \pi_0^*S(\varphi_t(x,v))} \chi(t) \left(1 - \psi(t) \psi(\tfrac{\pi}{2} - \varphi)\right) F(t, x, v) \dd t \dd v,
\end{equation}
where the coordinate $\varphi$ is encodes the distance to $\ker (dS(x)) \cap S_xM$ (as explained above). Now, assume that the phase $\Phi$ (introduced in \eqref{eq:phase}) has a critical point $(t, v) \in \mathbb{R} \times S_xM$, that is,
\[
	dS(\gamma(t)) \big(d\pi(\varphi_t(x, v)) X(\varphi_t(x, v))\big) = 0, \quad dS (\gamma(t)) \big(d\pi(\varphi_t(x, v)) d\varphi_t(x, v) V\big) = 0,
\]
where $\gamma(t) = \pi(\varphi_t(x, v))$, for any vertical vector field $V \in \mathbb{V}(x, v)$. Then either $t = 0$, in which case we have $v \in \mathbb{S}_{\xi}^{n - 2}$ and the cut off in \eqref{eq:F_2} is zero, or $t \neq 0$, so by the absence of conjugate points (see \eqref{equation:conjugate}), we conclude that $dS(\gamma(t)) = 0$ and therefore $F(t', x, v) = 0$ for $|t' - t|$ uniformly small independent of $S \in K$ (using that $P_{L/R}$ are differential operators and $|dS| \geq \delta$ on $N$). It follows, upon applying the (non-)stationary phase lemma similarly to the argument in the previous paragraph, that the expression in \eqref{eq:F_2} equals $\mc{O}_{C^\infty}(h^\infty)$ for $x \in U$, with the seminorms uniformly bounded as $S \in K$. This completes the discussion for the case $x \in N$ and show the required asymptotic expansion.

Finally, it remains to deal with the case $x \not \in N$; for that we go back to \eqref{eq:F-expression}. Then in particular $f = 0$ near $x$ (with the neighbourhood independent of $S \in K$) and by the analysis of the phase function $\Phi$ from the previous step (note that in this case the integrand vanishes for $|t|$ small enough uniformly in $S \in K$) we conclude similarly using the (non-)stationary phase lemma that near $x$ the expression \eqref{eq:F-expression} contributes to $\mc{O}_{C^\infty}(h^\infty)$ with seminorms uniformly bounded with respect to $S \in K$. This completes the proof.
\end{proof}

\begin{remark}\rm
	Note that in the above proof the fact that $P_{L/R}$ are differential operators gets used in the last two paragraphs through their \emph{locality}. In the more general case of arbitrary pseudodifferential operators $P_{L/R}$, this has to be replaced by \emph{pseudolocality} and a similar proof applies.
\end{remark}

\section{Generic injectivity with respect to the connection}

\label{section:genericity-connection}

We now prove Theorem \ref{theorem:genericity-connection} in this section. This case is much less technical than the metric case but still provides a good insight on the argument. In what follows, differentiation will be mostly carried out without recalling that the objects depend smoothly on the parameter and we refer the reader to \S\ref{ssection:resolvent-perturbations-theory}, Lemma \ref{lemma:differentiabilty} for further details. \\

\subsection{Preliminary remarks}

Consider an Anosov Riemannian manifold, denoted by $(M,g)$, with a Hermitian vector bundle $\E \rightarrow M$, equipped with a unitary connection $\nabla^{\E}$. Consider a linear perturbation $\nabla^{\E} + \tau \Gamma$ for some skew-Hermitian $\Gamma \in C^\infty(M,T^*M \otimes \End_{\mathrm{sk}}(\E))$ and $\tau \in \R$, and the operator $\X_\tau := \pi^*(\nabla^{\E} + \tau  \Gamma)_X$, where we recall that $\pi: SM \to M$ is the footpoint projection. We set $\X := \X_0$. For the sake of simplicity, let us assume that $\ker \X|_{C^\infty(SM, \E)} = \{0\}$\footnote{\label{footnote:1}The following arguments can be generalized to the case where $\ker \X$ consists of \emph{stable} elements of degree $0$ (equivalently, we will say that $\ker \X$ is stably non-empty): by stable, we mean that any perturbation of the operator will still have the same resonant space at $0$ and that this space only contains elements of degree $0$. This is the case for the operator $X$ acting on functions as it always has $\C \cdot \mathbf{1}$ (the constant sections) as resonant space at $z=0$; this is also the case for $(\pi^* \nabla^{\End(\E)})_X$ as it always contains $\C \cdot \mathbbm{1}_{\E}$ and is generically equal to $\C \cdot \mathbbm{1}_{\E}$ by \cite{Cekic-Lefeuvre-20} (where $\nabla^{\End(\E)}$ is the induced connection on $\End(\E)$). Instead of taking the resolvent at $0$, one needs to work with the holomorphic part of the resolvent. This is done in the metric case, see \S\ref{section:genericity-metric}. For the sake of simplicity, we assume in this section that $\ker \X$ is trivial.}, which is generically true by \cite{Cekic-Lefeuvre-20}. We will consider the operator
\[
P_\tau := \pi_{\ker D_0^*} \Delta_0 \Pi_m^\tau \Delta_0 \pi_{\ker D^*_0},
\]
where $\Pi_m^\tau = \pi_{m*}(\RR_+^{\tau} + \RR_-^{\tau})\pi_m^*$, $\Delta_0$ is an elliptic, formally self-adjoint, positive, pseudodifferential operator with diagonal principal symbol of order $k>1/2$ that induces an isomorphism on Sobolev spaces $H^s$ for any $s$ and $D_\tau$ denotes the twisted (with respect to $\nabla^{\E} + \tau \Gamma$) symmetric derivative on tensors, as in \S\ref{sssection:tensors-twisted}. Here $\RR_\pm^\tau := - \RR_\pm^\tau(z=0)$ is the (opposite of) resolvent at zero of $\X_\tau$ as defined in \S\ref{sssection:t-g-x-ray}. It is important to note that 
\[\pi_{\ker D^*_\tau} \Pi_m^\tau = \Pi_m^\tau \pi_{\ker D^*_\tau} = \Pi_m^\tau,\]
and that by continuity we have:
\begin{lemma}
\label{lemma:isomorphisms}
	For any given $s \in \mathbb{R}$, for $|\tau|$ small enough, the map $\tau \mapsto \|\pi_{\ker D_\tau^*}\|_{H^s \to H^s}$ is continuous. Moreover, for all $s \in \R$, there exists $\varepsilon = \varepsilon(s) > 0$ and $C = C(s) > 0$ such that for $|\tau| < \varepsilon$:
	\begin{equation}\label{eq:firstpoint}
	\forall f \in H^{s+k}(M,\otimes^m_S T^*M \otimes \E), ~~~	\|\pi_{\ker D_\tau^*} \Delta_0 \pi_{\ker D_\tau^*}f\|_{H^s} \geq C\|\pi_{\ker D_\tau^*}f\|_{H^{s + k}}.
	\end{equation}
	Hence for any $s \in \mathbb{R}$, for all $|\tau| < \varepsilon$, where $\varepsilon = \varepsilon(s) > 0$ is small enough, the following maps are isomorphisms:
	\[
	\begin{split}
	&\pi_{\ker D^*_\tau} \Delta_0 \pi_{\ker D^*_0}: \ker D_0^* \cap H^{s + k} \to \ker D_\tau^* \cap H^{s}, \\ &\pi_{\ker D^*_0} \Delta_0 \pi_{\ker D^*_\tau}: \ker D_\tau^* \cap H^{s + k} \to \ker D_0^* \cap H^{s}.
	\end{split}
	\]
\end{lemma}
\begin{proof}
The first claim follows from the formula (see \eqref{equation:projection-pi}):
	\begin{equation}\label{eq:projection-op-auxiliary}
		\pi_{\ker D_\tau^*} = \mathbbm{1} - D_\tau (D_\tau^* D_\tau)^{-1} D_\tau^*,
	\end{equation}
	once we show that $\tau \mapsto \|(D_\tau^*D_\tau)^{-1}\|_{H^s \to H^{s + 2}}$ is well-defined and continuous for $|\tau|$ small enough. Firstly, note that if $D_0 f = 0$, then $\X \pi_{m}^*f = 0$ by \eqref{equation:xe} and this implies $f = 0$ by our assumptions, so the map $D_0^*D_0$ is invertible on Sobolev spaces.\footnote{If $\ker \X$ is not empty but stably non-empty, this argument also works.} Using the identity
	\[
		D_\tau^* D_\tau = D_0^*D_0 \left(\mathbbm{1} - (D_0^* D_0)^{-1} (\underbrace{D_0^*D_0 - D_\tau^* D_\tau}_{\tau S_\tau :=})\right),
	\]
	and the fact that $\|S_\tau\|_{H^{s + 2} \to H^{s + 1}} = \mc{O}(1)$ as $\tau \to 0$ (which follows upon applying the formula \eqref{equation:de} and its adjoint), we conclude by inverting this identity that $(D_\tau^* D_\tau)^{-1}$ is an isomorphism $H^s \to H^{s + 2}$ for $|\tau|$ small enough, and moreover by using Neumann series that 
	\[
		(D_\tau^*D_\tau)^{-1} - (D_0^* D_0)^{-1} = \tau (D_0^* D_0)^{-1}S_\tau \left(\mathbbm{1} - \tau (D_0^*D_0)^{-1} S_\tau\right)^{-1} (D_0^*D_0)^{-1} = \mc{O}_{H^{s} \to H^{s + 3}}(\tau).
	\]
	From here we deduce using \eqref{eq:projection-op-auxiliary}
	\begin{equation}\label{eq:auxiliary-uniform-bound}
	\begin{split}
		\pi_{\ker D_\tau^*} - \pi_{\ker D_0^*} &= (D_0 - D_\tau) (D_0^* D_0)^{-1} D_0 + D_\tau \left((D_0^*D_0)^{-1} - (D_\tau^*D_\tau)^{-1}\right) D_0^*\\ 
		&+ D_\tau (D_\tau^*D_\tau)^{-1} (D_0^* - D_\tau^*) = \mc{O}_{H^s \to H^{s + 1}}(\tau),
	\end{split}
	\end{equation}
	as $\tau \to 0$, where the constant depends on $s$; the claim follows.

Next, since $\Delta_0$ is an isomorphism on Sobolev spaces $H^s$, we have that $\|\Delta_0 f\|_{H^s} \geq C_s \|f\|_{H^{s + k}}$ for some $C_s > 0$. Using the identity
	\[\pi_{\ker D_\tau^*} \Delta_0 \pi_{\ker D_\tau^*} = [\pi_{\ker D_\tau^*} - \pi_{\ker D_0^*}, \Delta_0] \pi_{\ker D_\tau^*} + [\pi_{\ker D_0^*}, \Delta_0]\pi_{\ker D_\tau^*} + \Delta_0 \pi_{\ker D_\tau^*},\]
	as well as that $[\pi_{\ker D_0^*} - \pi_{\ker D_\tau^*}, \Delta_0] = \mc{O}_{H^{s + k - 1} \to H^s}(1)$ as $\tau \to 0$ (which follows from \eqref{eq:auxiliary-uniform-bound} with possibly a different constant, by replacing $s$ there with $s - 1$ and $s + k - 1$), we obtain the estimate
	\begin{equation}
	\begin{split}
	\label{eq:remainder}
		\|\pi_{\ker D_\tau^*}f\|_{H^{s + k}} & \leq C_s \|\Delta_0 \pi_{\ker D_\tau^*}f\|_{H^s} \\
		&  \leq C' (\|\pi_{\ker D_\tau^*} \Delta_0 \pi_{\ker D_\tau^*}f\|_{H^s} + \|\pi_{\ker D_\tau^*}f\|_{H^{s + k - 1}})
		\end{split}
	\end{equation}
	for some constant $C' = C'(s) > 0$ independent of $\tau$. To show \eqref{eq:firstpoint}, we argue by contradiction and assume there is a sequence $f_n \in \ker D_{\tau_n}^*$ with $\|f_n\|_{H^{s + k}} = 1$ and $\|\pi_{\ker D_{\tau_n}}^* \Delta_0f_n\|_{H^s} \to 0$. We assume $\tau_n \to 0$ (but the same argument works if $\tau_n \to \tau$ for some $\tau \neq 0$). By compactness, we may assume $f_n$ converges in $H^{s + k -1}$. In fact, by \eqref{eq:remainder} we have:
	\begin{align*}
		\|f_n - f_m\|_{H^{s + k}} &\leq \|\pi_{\ker D_{\tau_n}^*}(f_n - f_m)\|_{H^{s + k}} + \|(\pi_{\ker D_{\tau_n}^*} - \pi_{\ker D_{\tau_m}^*}) f_m\|_{H^{s + k}}\\
		&\leq C'\|\pi_{\ker D_{\tau_n}^*}\Delta_0 \pi_{\ker D_{\tau_n}^*}f_m\|_{H^s} + C'\|\pi_{\ker D_{\tau_n}^*} (f_n - f_m)\|_{H^{s + k -1}} \\
		& \hspace{7cm} + o(1)\\
		&\leq C'\|\pi_{\ker D_{\tau_m}^*} \Delta_0 f_m\|_{H^s} + C'\|(\pi_{\ker D_{\tau_n}^*} - \pi_{\ker D_{\tau_m}^*}) \Delta_0 f_m\|_{H^s}\\
		 &+ C'\|\pi_{\ker D_{\tau_n}^*} \Delta_0 (\pi_{\ker D_{\tau_n}^*} - \pi_{\ker D_{\tau_m}^*}) f_m\|_{H^s} + o(1) = o(1).
	\end{align*}
	as $m, n \to \infty$. In the second line, we used that $\pi_{\ker D_{\tau_n}^*} - \pi_{\ker D_{\tau_m}^*} = o_{H^{s + k} \to H^{s + k}}(1)$ (which follows from \eqref{eq:auxiliary-uniform-bound} with possibly a different constant, by replacing $s$ there with $s + k$) and $\|f_m\|_{H^{s+k}} = 1$, \eqref{eq:remainder}, and the fact that $\|\pi_{\ker D_{\tau_n}}^* \Delta_0f_n\|_{H^s} = o(1)$. In the last line, we also used the assumption that $f_n$ converges in $H^{s + k - 1}$. Therefore, $(f_n)_{n \in \mathbb{N}}$ is a Cauchy sequence in $H^{s + k}$ and it converges to some $f \in H^{s + k}$ with $\|f\|_{H^{s + k}}=1$, $D_0^* f = 0$ and $\pi_{\ker D_0^*} \Delta_0 f = 0$. Using that $[\pi_{\ker D_0^*}, \Delta_0]f = -\Delta_0f$ and the fact that $[\pi_{\ker D_0^*}, \Delta_0] \in \Psi^{k - 1}$ implies by elliptic regularity that $f \in H^{s + k + 1}$. Bootstrapping we get $f \in C^\infty$ and so
	\[
		0 = \langle{\pi_{\ker D_0^*}\Delta_0f, f}\rangle_{L^2} = \langle{\Delta_0f, f}\rangle_{L^2},		\]
	which means that $f = 0$ as $\Delta_0$ was chosen to be positive. This contradicts that $\|f\|_{H^{s + k}} = 1$ and proves \eqref{eq:firstpoint}.

	Finally, by the first point we have $\|\pi_{\ker D_0^*} - \pi_{\ker D_\tau^*}\|_{H^s \to H^s} = o_s(1)$ as $\tau \to 0$, so by \eqref{eq:firstpoint} for small $|\tau|$ depending on $s$ we get
		\begin{align}\label{eq:closedrange1}
			\begin{split}
				\|\pi_{\ker D_\tau^*} \Delta_0 \pi_{\ker D_0^*}f\|_{H^s} &\geq \|\pi_{\ker D_0^*} \Delta_0 \pi_{\ker D_0^*} f\|_{H^s} \\
				& \hspace{1cm} - \|(\pi_{\ker D_0^*} - \pi_{\ker D_\tau^*} ) \Delta_0 \pi_{\ker D_0^*}f\|_{H^{s}}\\
			&\geq \frac{C(s)}{2} \|\pi_{\ker D_0^*}f\|_{H^{s+k}}.
			\end{split}
		\end{align}
	Similarly, using \eqref{eq:firstpoint} for $|\tau|$ small enough we obtain:
\begin{equation}\label{eq:closedrange2}
	\|\pi_{\ker D_\tau^*} \Delta_0 \pi_{\ker D_0^*}f\|_{H^s} \geq \frac{C(s)}{2} \|\pi_{\ker D_0^*}f\|_{H^{s+k}}.
\end{equation}
		 Estimates \eqref{eq:closedrange1} and \eqref{eq:closedrange2} show that the operators $\pi_{\ker D_0^*} \Delta_0 \pi_{\ker D_\tau^*}$, $\pi_{\ker D_\tau^*} \Delta_0 \pi_{\ker D_0^*}$ are injective and have a closed range for $|\tau|$ small enough, and then the surjectivity follows since their $L^2$-adjoints are injective. This completes the proof. 
\end{proof}

Next, using Proposition \ref{proposition:sandwich}, the fact that $\Delta_0$ acts diagonally to principal order and the equation \eqref{eq:sigma-proj-ker}, we obtain that for $\xi \in T^*_xM \setminus \{0\}$:
\[
	\sigma(P_\tau)(x, \xi) = \frac{2\pi}{|\xi|} C_{2m + n - 1}^{-1}\sigma(\Delta_0)^2(x, \xi) \pi_{\ker  \imath_{\xi^\sharp}} {\pi_m}_*\pi_m^*  \pi_{\ker  \imath_{\xi^\sharp}} \otimes \mathbbm{1}_{\E_x}.
\]
Note that here we simply use Proposition \ref{proposition:sandwich} to compute the symbol of  $\Pi_m$, and then the pseudodifferential nature and the symbol of $P_\tau$ follow from the usual pseudodifferential calculus. Therefore, the symbol of $P_\tau$ at $(x, \xi)$ is invertible on $\ker_{\imath_{\xi^\sharp}}$ and by standard microlocal analysis for each $\tau$ there exist pseudodifferential operators $Q_\tau$ and $R_\tau$ of respective orders $1 - 2k$ and $-\infty$ (cf. \cite[Lemma 2.5.3]{Lefeuvre-thesis}) , such that 
\[Q_\tau P_\tau = \pi_{\ker D_0^*} + R_\tau.\]
Using that $\Pi_m^\tau \geq 0$ we get $P_\tau \geq 0$ and it follows that $(P_\tau + 1)^{-1}: L^2 \cap \ker D_0^* \to L^2 \cap \ker D_0^*$ is compact and thus the spectrum of $P_\tau$ is well-defined. It is discrete, non-negative and accumulates at infinity, and the eigenfunctions of $P_\tau$ are smooth. Moreover, by Lemma \ref{lemma:isomorphisms} and again using that $\Pi_m^\tau \geq 0$, for small $|\tau|$ we have that
\[\pi_{\ker D_\tau^*} \Delta_0 \pi_{\ker D_0^*}: \ker P_\tau|_{C^\infty \cap \ker D_0^*} \xrightarrow{\sim} \ker \Pi_m^\tau|_{C^\infty \cap \ker D_0^*}\]
is an isomorphism. Therefore we see that for each $\tau$, $0$ is an eigenvalue of $\Pi_m^{\tau}$ if and only if $0$ is an eigenvalue of $P_\tau$.

\subsection{Variations of the ground state}\label{ssection:variations_ground_state}

We assume that $\ker P_0|_{C^\infty \cap \ker D_0^*}$ is $d$-dimensional, for some $d \geq 1$, and spanned by $u_1, ..., u_d \in C^\infty(M, \otimes_S^m T^*M \otimes \E) \cap \ker D_0^*$ with $\langle u_i, u_j \rangle_{L^2} = \delta_{ij}$. Let $\Pi_\tau$ be the $L^2$-orthogonal spectral projector
\begin{equation}
\label{equation:pi-tau}
\Pi_\tau := \dfrac{1}{2 \pi i} \oint_{\gamma'} (z-P_\tau)^{-1} \dd z,
\end{equation}
where $\gamma'$ is a small circle centred around $0$ and not containing any other eigenvalue of $P_0$ in its interior. In particular, we have $\Pi_{\tau=0} = \sum_{i=1}^d \langle \bullet , u_i \rangle_{L^2} u_i$ and by ellipticity of $P_0$ on $\ker D_0^*$ we have the meromorphic expansion close to zero, valid on $S := L^2 \cap \ker D_0^*$
\begin{equation}\label{eq:expansion-appendix}
	(z - P_0)^{-1} = \frac{\Pi_0}{z} - P_0^{-1} + zH_1 + \mc{O}(z^2),
\end{equation}
for some maps $P_0^{-1}, H_1: S \to S$, where we recall our notational conventions were explained in \S \S \ref{ssection:convention}. These maps satisfy the relations (cf. \eqref{eq:resolvent-identities}):
\begin{equation}\label{eq:resolvent-identities-appendix}
	P_0 P_0^{-1} = P_0^{-1} P_0 = \id - \Pi_0, \quad P_0H_1 = -P_0^{-1}, \quad \Pi_0 P_0^{-1} = P_0^{-1} \Pi_0 = 0.
\end{equation}
We introduce $\lambda_\tau$ as the sum of the eigenvalues of $P_\tau$ inside $\gamma'$:
\begin{equation}
\label{equation:lambda-tau}
\lambda_{\tau} := \Tr(P_\tau \Pi_\tau).
\end{equation}
Note that both $\tau \mapsto \Pi_\tau \in \mc{L}(L^2)$ and $\tau \mapsto \lambda_\tau \in \C$ are smooth by standard elliptic theory (see \cite[Section 4]{Cekic-Lefeuvre-20}). Observe that $\lambda_{\tau=0} = 0$ and as $P_\tau \geq 0$, we have $\lambda_{\tau} \geq 0$. Our goal is to produce a small perturbation $\nabla^{\E} + \tau \Gamma$ (where $\Gamma$ is skew-Hermitian) such that $\lambda_{\tau} > 0$ for $\tau \neq 0$. This will say that at least one of the eigenvalues was ejected from $0$ and that $\ker P_{\tau}$ is at most $(d-1)$-dimensional (for $\tau \neq 0$). Iterating the process, we will then obtain a perturbation of $\nabla^{\E}$ with injective (twisted) X-ray transform.

We make the easy observation that the first variation is zero, as $\lambda_{\tau = 0} = 0$ is a local minimum of the smooth function $\tau \mapsto \lambda_\tau$:
\begin{equation}\label{eq:firstvariation-lambda}
	\dot{\lambda}_{\tau = 0} = 0.
\end{equation}
Next, we note that since $P_\tau u_i = 0$ and $\Pi_m^\tau \geq 0$, we have $\Pi_m^\tau \Delta_0u_i = 0$. Therefore, by Lemma \ref{lemma:relation-twisted}, there exists $\VV_i \in C^\infty(SM, \E)$ such that
\begin{equation}\label{eq:u_iVV_i}
	\pi_m^*\Delta_0 u_i = \mathbf{X} \VV_i, \quad \Pi_+\VV_i = 0.
\end{equation}
By the mapping properties of $\X$ (see \eqref{eq:mapping-property-X}), we have that if $m$ is even then $\VV_i$ may be chosen odd and vice versa, if $m$ is odd then $\VV_i$ may be chosen even.

\subsubsection{Second order variations.}

We now compute $\ddot{\lambda}_{\tau = 0}$. For simplicity, when clear from the context we will drop the $\tau = 0$ subscript and simply write $\dot{\lambda}$ and $\ddot{\lambda}$. We start with an abstract lemma, valid in a more general setting (this will also get used in the metric case, see Lemma \ref{lemma:second-variation-resolvent} below):

\begin{lemma}
\label{lemma:abstract}
	The following variational formula holds:
	\[
		\ddot{\lambda}_{\tau = 0} = \Tr(\ddot{P}_0 \Pi_0) - 2\Tr(\Pi_0 \dot{P}_0 P_0^{-1} \dot{P}_0 \Pi_0).
	\]
\end{lemma}

\begin{proof}
	We compute:
	\begin{align*}
		\dot{\Pi}_0 &= \frac{1}{2\pi i} \oint_{\gamma'} (z - P_0)^{-1} \dot{P}_0 (z - P_0)^{-1} dz = -\left(\Pi_0 \dot{P}_0 P_0^{-1} + P_0^{-1} \dot{P}_0 \Pi_0\right),\\
		\ddot{\Pi}_0 &= 2 \times \frac{1}{2\pi i} \oint_{\gamma'} (z - P_0)^{-1} \dot{P}_0 (z - P_0)^{-1} \dot{P}_0 (z - P_0)^{-1} dz \\
		& + \frac{1}{2\pi i} \oint_{\gamma'} (z - P_0)^{-1} \ddot{P}_0 (z - P_0)^{-1} dz.
	\end{align*}
	We expand the second formula using \eqref{eq:expansion-appendix} to get:
	\begin{multline}\label{eq:Pi-ddot-abstract}
		\ddot{\Pi}_0 = 2\left[\Pi_0 \dot{P}_0 \Pi_0 \dot{P}_0 H_1 + \Pi_0 \dot{P}_0 P_0^{-1} \dot{P}_0 P_0^{-1} + \Pi_0 \dot{P}_0 H_1 \dot{P}_0 \Pi_0 + P_0^{-1} \dot{P}_0 \Pi_0 \dot{P}_0 P_0^{-1}\right.\\
		 \left.+ P_0^{-1} \dot{P}_0 P_0^{-1} \dot{P}_0 \Pi_0 + H_1 \dot{P}_0 \Pi_0 \dot{P}_0 \Pi_0\right] - \left(\Pi_0 \ddot{P}_0 P_0^{-1} + P_0^{-1} \ddot{P}_0 \Pi_0\right).
	\end{multline}
	Therefore, we compute using \eqref{eq:resolvent-identities-appendix}:
		\begin{equation}\label{eq:P_0Pi-ddot-appendix}
		P_0\ddot{\Pi}_0 = 2(\id - \Pi_0) \dot{P}_0 \left(\Pi_0 \dot{P}_0 P_0^{-1} + P_0^{-1} \dot{P}_0 \Pi_0\right) - 2P_0^{-1} \dot{P}_0 \Pi_0 \dot{P}_0 \Pi_0 - (\id - \Pi_0) \ddot{P}_0 \Pi_0,
		\end{equation}
		which implies that, using the cyclicity of the trace (here and below, we use the fact that $\mathrm{Tr}(AB) = \mathrm{Tr}(BA)$ for two bounded operators $A, B$ on a Hilbert space $\mc{H}$, as soon as one of them has finite rank, see \cite[Appendix B.4]{Dyatlov-Zworski-19}) and \eqref{eq:resolvent-identities-appendix}:
		\[\Tr\left(P_0\ddot{\Pi}_0\right) = 2\Tr\left(\dot{P}_0 \Pi_0 \dot{P}_0 P_0^{-1}\right) = 2\Tr\left(\Pi_0 \dot{P}_0 P_0^{-1}\dot{P}_0 \Pi_0\right).\]
		Finally, we obtain using once more the cyclicity of the trace:
		\begin{align*}
			\ddot{\lambda} &= \Tr\left(\ddot{P}_0 \Pi_0 + 2\dot{P}_0 \dot{\Pi}_0 + P_0 \ddot{\Pi}_0 \right) = \Tr\left(\ddot{P}_0 \Pi_0\right) - 2\Tr\left(\Pi_0 \dot{P}_0 P_0^{-1} \dot{P}_0 \Pi_0\right).
		\end{align*}
\end{proof}

Next, we compute $\dot{P}_0, \ddot{P}_0$ and apply Lemma \ref{lemma:abstract}. Before doing that, note that by Lemma \ref{lemma:isomorphisms}, on $\ker D_0^*$ we have 
\begin{equation}\label{eq:P_0^-1}
	P_0^{-1} = (\pi_{\ker D_0^*} \Delta_0 \pi_{\ker D_0^*})^{-1} \Pi_m^{-1} (\pi_{\ker D_0^*} \Delta_0 \pi_{\ker D_0^*})^{-1},
\end{equation}
where $\Pi_m^{-1}$ stands for the holomorphic part of the resolvent at zero as defined in \S\ref{ssection:convention} that is, if $(\Pi_m- z)^{-1} = \frac{\Pi_{m, 0}}{z} + R_m(z)$, where $\Pi_{m, 0}$ is the orthogonal projection onto $\ker \Pi_m|_{\ker D_0^*}$ zero and $R_m(z)$ holomorphic close to $z = 0$ then $\Pi_m^{-1} := R_m(0)$.

\begin{lemma}
\label{lemma:second-variation}
We have:
\begin{equation}
\begin{split}
\label{eq:inter}
\ddot{\lambda}  & = 2 \sum_{i=1}^d \big\langle \mc{I}_{\nabla^{\E}} (\pi_1^*\Gamma. \VV_i), \pi_1^*\Gamma. \VV_i \big\rangle_{L^2} \\
& \hspace{2cm} - \big\langle{\Pi_m^{-1} \pi_{m*} \mc{I}_{\nabla^{\E}} (\pi_1^*\Gamma. \VV_i), {\pi_{m}}_*\mc{I}_{\nabla^{\E}} (\pi_1^*\Gamma. \VV_i) }\big\rangle_{L^2}.
\end{split}
\end{equation}
\end{lemma}
\begin{proof}
We start with the first term in \eqref{eq:inter}. For the variation of the resolvent, recalling that (similarly to the metric case in \eqref{eq:resolvent-identities})
\begin{equation}
\label{equation:ajout}
\X\RR_- = \RR_-\X= -\mathbbm{1}, \qquad \X\RR_+ = \RR_+\X = \mathbbm{1}, \qquad \RR_+^* = \RR_-,
\end{equation}
as $z=0$ is not a resonance by assumption, and the notation of \S \ref{sssection:t-g-x-ray} ($\gamma$ is a small loop around zero), we have:
\begin{equation}\label{eq:resolvent-second-variaiton-connection}
	\begin{split}
	\ddot{\RR}_+ &= \partial_\tau^2|_{\tau = 0} \frac{1}{2\pi i} \oint_\gamma \frac{1}{z} (z + \X_\tau)^{-1} dz = \partial_\tau^2|_{\tau = 0} \X_\tau^{-1} = 2\RR_+ \pi_1^*\Gamma \RR_+ \pi_1^*\Gamma \RR_+,\\
	\ddot{\RR}_- &= (\ddot{\RR}_+)^*  = \partial_\tau^2|_{\tau = 0} \frac{1}{2\pi i} \oint_\gamma \frac{1}{z} (z - \X_\tau)^{-1} dz \\
	& = \partial_\tau^2|_{\tau = 0} (-\X_\tau)^{-1} = 2\RR_- \pi_1^*\Gamma \RR_-\pi_1^*\Gamma \RR_-.
	\end{split}
\end{equation}
We remark that here we strongly use the facts that $\X_\tau$ is linear in $\tau$ (so $\partial_\tau^2 \X_\tau = 0$) and that $\X_{\tau}$ is invertible with inverse denoted by $\X_\tau^{-1}$ (as we shall see below in Lemma \ref{lemma:second-variation-resolvent} this signficantly complicates in the case of metrics and more terms appear). Therefore
\[\ddot{P}_0 = 2 \pi_{\ker D_0^*} \Delta_0 \pi_{m*} \left(\RR_+ \pi_1^*\Gamma \RR_+ \pi_1^*\Gamma \RR_+ + \RR_- \pi_1^*\Gamma \RR_- \pi_1^*\Gamma \RR_-\right) \pi_m^* \Delta_0 \pi_{\ker D_0^*}\]
and we obtain, using that $\X \VV_i = \pi_m^*\Delta_0 u_i$ (see \eqref{eq:u_iVV_i}):
\begin{equation}
\label{eq:term1}
\begin{split}
\Tr(\ddot{P}_0 \Pi_0) & = 2\sum_{i=1}^d \big\langle \big(\RR_+ \pi_1^*\Gamma \RR_+ \pi_1^*\Gamma \RR_+ + \RR_- \pi_1^*\Gamma \RR_- \pi_1^*\Gamma \RR_-\big) \X \VV_i, \X \VV_i \big\rangle_{L^2} \\
& = 2 \sum_{i=1}^d \big\langle \mc{I}_{\nabla^{\E}} (\pi_1^* \Gamma. \VV_i), \pi_1^* \Gamma. \VV_i \big\rangle_{L^2},
\end{split}
\end{equation}
using \eqref{equation:ajout} in the second line, as well as that $\Gamma$ is skew-Hermitian.

For the second term of \eqref{eq:inter}, we first observe that similarly to \eqref{eq:resolvent-second-variaiton-connection}
\begin{equation*}
	\dot{\RR}_+ = \partial_\tau|_{\tau = 0} \X_\tau^{-1} = -\RR_+ \pi_1^*\Gamma \RR_+, \quad
	\dot{\RR}_- = \partial_\tau|_{\tau = 0} (-\X_\tau)^{-1} = \RR_- \pi_1^*\Gamma \RR_-.
\end{equation*}
Therefore it holds that
\begin{equation}\label{eq:dot-P_0}
	\dot{P}_0 = \pi_{\ker D_0^*} \Delta_0 {\pi_m}_* (-\RR_+ \pi_1^*\Gamma \RR_+ + \RR_- \pi_1^*\Gamma \RR_-) \pi_m^*\Delta_0 \pi_{\ker D_0^*},
\end{equation}
and we finally obtain that, using $(\dot{P}_0)^* = \dot{P}_0$:
\begin{equation}
\label{equation:p-moins-un}
\begin{split}
	\Tr&\left(\Pi_0 \dot{P}_0 P_0^{-1} \dot{P}_0 \Pi_0\right) \\
	&=  \sum_{i=1}^d \left\langle P^{-1}_0 \dot{P}_0 u_i, \dot{P}_0 u_i \right\rangle_{L^2}\\
	&= \sum_{i = 1}^d \left\langle{P_0^{-1} \pi_{\ker D_0^*}\Delta_0 \mc{I}_{\nabla^{\E}} \left(\pi_1^*\Gamma. \VV_i\right), \pi_{\ker D_0^*} \Delta_0 \mc{I}_{\nabla^{\E}} \left(\pi_1^*\Gamma. \VV_i\right)}\right\rangle_{L^2}\\
	&= \sum_{i = 1}^d\left\langle{\Pi_m^{-1} \pi_{m*} \mc{I}_{\nabla^{\E}} (\pi_1^*\Gamma. \VV_i), \pi_{m*} \mc{I}_{\nabla^{\E}} (\pi_1^*\Gamma. \VV_i)}\right\rangle_{L^2}.
	\end{split}
\end{equation}
Here in the second line we used \eqref{eq:dot-P_0} and \eqref{eq:u_iVV_i}; in the final line we used  \eqref{eq:P_0^-1}. This proves the announced result.
\end{proof}

\subsubsection{Properties of the operators involved in the second variation}

First of all we note that by Proposition \ref{proposition:sandwich}, $\Pi_m$ and hence $\Pi_m^{-1} \pi_{\ker D_0^*}$, are pseudodifferential operators of orders $-1$ and $1$, respectively. More precisely, recall $\Pi_m^{-1}$ was defined just below \eqref{eq:P_0^-1} as the holomorphic part of $(\Pi_m-z)^{-1}$ at zero and is thus a pseudodifferential operator; $\pi_{\ker D_0^*}$ is a pseudodifferential operator using the formula analogous to \eqref{equation:projection-pi}. It follows that, for $(x, \xi) \in T^*M\setminus 0$ (cf. \cite[Lemma 2.5.3]{Lefeuvre-thesis}):
\begin{equation}\label{eq:symbols-Pi_m-Pi_m^-1}
\begin{split}
	\sigma_{\Pi_m}(x,\xi) &= C^{-1}_{n-1+2m} \dfrac{2\pi}{|\xi|} \pi_{\ker \imath_{\xi^\sharp}} {\pi_m}_* \pi_m^* \pi_{\ker \imath_{\xi^\sharp}} \otimes \mathbbm{1}_{\E_x},\\
	\sigma_{\Pi_m^{-1} \pi_{\ker D_0^*}}(x,\xi) &= C_{n-1+2m} \dfrac{|\xi|}{2\pi} (\pi_{\ker \imath_{\xi^\sharp}} {\pi_m}_* \pi_m^* \pi_{\ker \imath_{\xi^\sharp}})^{-1} \pi_{\ker \imath_{\xi^\sharp}} \otimes \mathbbm{1}_{\E_x}.
\end{split}
\end{equation}
We now fix $\VV \in C^\infty(SM,\pi^*\E)$ and introduce the multiplication map
\[
M_{\VV} : C^\infty(SM,\pi^* \End(\E)) \rightarrow C^\infty(SM,\pi^*\E), \quad M_{\VV} A := A.\VV.
\]	
Its adjoint is given by $M_{\VV}^* \WW = \langle{\bullet, \VV}\rangle_{\E} \otimes \WW$, for any $\WW \in C^\infty(SM, \pi^*\E)$. 
Next, we show that the terms appearing in the formula for $\ddot{\lambda}$ in Lemma \ref{lemma:second-variation} have pseudodifferential nature:

\begin{lemma}
\label{lemma:symbol-q}
For all $\VV \in C^\infty(SM,\pi^*\E)$, the operator
\begin{equation}\label{eq:Q_v-def}
	Q_{\VV} := {\pi_m}_* \mc{I}_{\nabla^{\E}} M_{\VV} \pi_1^* \in \Psi^{-1}\left(M, T^*M \otimes \End(\E) \rightarrow \otimes^m_S T^*M \otimes \E\right)
\end{equation}
is pseudodifferential of order $-1$ with principal symbol for $(x,\xi) \in T^*M \setminus \left\{ 0 \right\}$
\[
\sigma_{Q_{\VV}}(x,\xi) \in \mathrm{Hom}(T_x^*M \otimes \End(\E_x), \otimes^m_S T_x^*M \otimes \E_x),
\]
given by, for $B \in T_x^*M \otimes \End(\E_x)$:
\[
\sigma_{Q_{\VV}}(x,\xi) B  = C_{n-1+2m}^{-1} \dfrac{2 \pi}{|\xi|} \pi_{\ker \imath_{\xi^\sharp}} {\pi_m}_* E_{\xi}^m \left( \pi_1^* (\pi_{\ker \imath_{\xi^\sharp}}B). \VV(x,\bullet)\right).
\]
\end{lemma}

\begin{proof}
This follows directly from Proposition \ref{proposition:sandwich}.
\end{proof}

\begin{lemma}
\label{lemma:symbol-l}
	For any $\VV \in C^\infty(SM, \pi^*\E)$, the operator 
	\begin{equation}\label{eq:L_v-def}
		L_{\VV} := \pi_{1*} M_{\VV}^* \mc{I}_{\nabla^{\E}} M_{\VV} \pi_1^* \in \Psi^{-1}(M, T^*M \otimes \End(\E) \rightarrow T^*M \otimes \End(\E))
	\end{equation}
	is pseudodifferential of order $-1$ with principal symbol for $(x,\xi) \in T^*M \setminus \left\{ 0 \right\}$
	\[
	\sigma_{L_{\VV}}(x,\xi) \in \End(T^*_xM \otimes \End(\E_x))
	\]
	given by: 
	\[\sigma_{L_{\VV}}(x,\xi)  = C_{n-1+2m}^{-1} \frac{2\pi}{|\xi|}\pi_{\ker \imath_{\xi^\sharp}}{\pi_1}_* M_{\VV}^* (E^m_{\xi})^* E^m_{\xi} M_{\VV} \pi_1^*\pi_{\ker \imath_{\xi^\sharp}}.\]
\end{lemma}

\begin{proof}
Once more, this follows from Proposition \ref{proposition:sandwich} and the formula, for $B \in T^*_xM \otimes \End(\E_x)$:
\[
\begin{split}
\langle \sigma_{L_{\VV}}(x,\xi) B, B \rangle_x & = \frac{2\pi}{|\xi|}\int_{\Ss^{n-2}_{\xi}} \big\langle \pi_{1}^* B(u). \VV(x,u) , \pi_{1}^* B (u). \VV (x,u) \big\rangle_{x} \dd S_{\xi}(u) \\
& = C^{-1}_{n-1+2m} \frac{2\pi}{|\xi|}  \big\langle E^m_{\xi}(\pi_{1}^* B. \VV), E^{m}_{\xi}(\pi_{1}^* B. \VV) \big\rangle_{L^2(\Ss^{n-1}, \pi^* \E_{x})}\\
& = C^{-1}_{n-1+ 2m} \frac{2\pi}{|\xi|}  \big\langle  {\pi_1}_* M_{\VV}^* (E^m_{\xi})^* E^m_{\xi} M_{\VV}  \pi_1^* B, B \big\rangle_{L^2}.
\end{split}
\]
In the second line we used the Jacobian formula \eqref{eq:Jacobian}.
\end{proof}

\subsection{Assuming the second variation is zero} If $\ddot{\lambda} = 0$ for all linear variations $\nabla^{\E} + \tau \Gamma$ as in \S \ref{ssection:variations_ground_state}, where $\Gamma$ is skew-Hermitian, by Lemma \ref{lemma:second-variation} and using the notation of \eqref{eq:Q_v-def}, \eqref{eq:L_v-def}, this implies:
\begin{equation}
\label{equation:equality}
\forall \Gamma \in C^\infty(M,T^*M \otimes \End_{\mathrm{sk}}(\E)), \quad \sum_{i=1}^d \langle L_{\VV_i} \Gamma, \Gamma \rangle_{L^2} = \sum_{i=1}^d \langle \Pi_m^{-1} Q_{\VV_i}\Gamma, Q_{\VV_i} \Gamma \rangle_{L^2}.
\end{equation}
The idea is to apply the equality \eqref{equation:equality} (which is of \emph{analytic} nature) to Gaussian states in order to derive an \emph{algebraic} equality. Let $\mathrm{e}_h(x_0,\xi_0)$ be a Gaussian state centered at $(x_0, \xi_0) \in T^*M \setminus \left\{ 0 \right\}$, that is, a function which has the form in some local coordinates around $x_0$\footnote{Alternatively, a Gaussian state is an $h$-dependent function whose semiclassical defect measure is a point $(x_0,\xi_0) \in T^*M$, see \cite{Zworski-12}.}:
\begin{equation}\label{eq:gaussian}
\mathrm{e}_h(x_0,\xi_0) (x) = \dfrac{1}{(\pi h)^{\frac{n}{4}}} e^{\frac{i}{h}\xi_0 \cdot (x-x_0) - \frac{|x-x_0|^2}{2h}}.
\end{equation}
We will use the following standard technical lemma:

\begin{lemma}
\label{lemma:technical}
Let $P \in \Psi^m(M, E \rightarrow F)$ be a pseudodifferential operator of order $m \in \R$ acting on two Hermitian vector bundles $E, F \rightarrow M$. Let $e \in C^\infty(M,E)$ and $f \in C^\infty(M,F)$. Then:
\begin{multline*}
	\lim_{h \to 0} \Big\langle h^m P\Big(\Re(\mathrm{e}_h(x_0,\xi_0)). e\Big), \Re(\mathrm{e}_h(x_0,\xi_0)). f \Big\rangle_{L^2(M,F)} \\ = \frac{1}{2}\big\langle \sigma_P(x_0,\xi_0) e(x_0), f(x_0)\big\rangle_{F_{x_0}}
+\frac{1}{2}\big\langle \sigma_P(x_0,-\xi_0) e(x_0), f(x_0)\big\rangle_{F_{x_0}}.
\end{multline*}
\end{lemma}

We cannot directly apply \eqref{equation:equality} to $\mathrm{e}_h(x_0,\xi_0). \Gamma$ where $\Gamma \in C^\infty(M,T^*M \otimes \End_{\mathrm{sk}}(\E))$ because $\mathrm{e}_h(x_0,\xi_0). \Gamma$ is not skew-Hermitian. However, applying \eqref{equation:equality} to $\Re(\mathrm{e}_h(x_0,\xi_0)). \Gamma$ and using Lemma \ref{lemma:technical}, as well as the principal symbol formulas in \eqref{eq:symbols-Pi_m-Pi_m^-1}, Lemmas \ref{lemma:symbol-q} and \ref{lemma:symbol-l} (and that in particular these symbols are invariant under the antipodal map $(x,\xi) \mapsto (x,-\xi)$), taking $h \to 0$ we obtain: 
\begin{equation}
\label{eq:cornerstone}
\begin{split}
&\sum_{i=1}^d   \big\langle E^m_{\xi_0}(\pi_{1}^* \Gamma. \VV_i ), E^{m}_{\xi_0}(\pi_{1}^* \Gamma. \VV_i) \big\rangle_{L^2(\Ss^{n-1}, \pi^* \E_{x_0})} \\
& = \sum_{i=1}^d \big\langle \pi_m^* \pi_{\ker \imath_{\xi_0^\sharp}} \big[\pi_{\ker \imath_{\xi_0^\sharp}} {\pi_{m}}_* \pi_m^*  \pi_{\ker \imath_{\xi_0^\sharp}}\big]^{-1} \pi_{\ker \imath_{\xi_0^\sharp}}{\pi_{m}}_* E^m_{\xi_0}(\pi_{1}^* \Gamma. \VV_i),  \\
& \hspace{5cm}  E^m_{\xi_0}(\pi_{1}^* \Gamma. \VV_i) \big\rangle_{L^2(\Ss^{n-1}, \pi^*\E_{x_0})}.
\end{split}
\end{equation}
Since $(\pi_m^* \pi_{\ker \imath_{\xi_0^\sharp}})^* = \pi_{\ker \imath_{\xi_0^\sharp}} \pi_{m*}$ in the $L^2$ sense, we have the orthogonal decomposition:
\begin{equation}
\label{equation:decomposition0}
L^2(\Ss^{n-1}, \pi^*\E_{x_0}) = \pi_m^* \big(\otimes^m_S \ker \imath_{\xi_0^\sharp} \otimes \E_{x_0}\big) \oplus^{\bot} \ker \big(\pi_{\ker \imath_{\xi_0^\sharp}} {\pi_m}_*|_{L^2(\mathbb{S}^{n-1})}\big).
\end{equation}
In particular, if we define $\WW_i := E^m_{\xi_0}\big((\pi_{1}^* \Gamma. \VV_i)|_{\mathbb{S}^{n-2}_{\xi_0}}\big)$, then we can write
\begin{equation}
\label{equation:decomposition}
\WW_i = \pi_m^* T_i + h_i,
\end{equation}
where $T_i \in \otimes^m_S \ker \imath_{\xi_0^\sharp} \otimes \E_{x_0}$ and $h_i \in \ker(\pi_{\ker \imath_{\xi^\sharp}} {\pi_m}_*|_{L^2(\mathbb{S}^{n-1})})$. We also define
\[
\begin{split}
P_m & := \pi_m^* \pi_{\ker \imath_{\xi_0^\sharp}} [\pi_{\ker \imath_{\xi_0^\sharp}} {\pi_{m}}_* \pi_m^*  \pi_{\ker \imath_{\xi_0^\sharp}}]^{-1} \pi_{\ker \imath_{\xi_0^\sharp}}{\pi_{m}}_* : \\
& \hspace{3cm} L^2(\Ss^{n-1}, \pi^*\E_{x_0}) \to L^2(\Ss^{n-1}, \pi^*\E_{x_0}) ,
\end{split}
\]
and observe that $P_m^2 = P_m$, $P_m^* = P_m$, so $P_m$ is the orthogonal projection onto the first factor of \eqref{equation:decomposition0}. In particular $\pi_m^* T_i = P_m \WW_i$ and \eqref{eq:cornerstone} reads:
\begin{equation}
\label{equation:partial}
\sum_{i=1}^d \|\WW_i\|^2_{L^2(\Ss^{n-1},\pi^*\E_{x_0})} = \sum_{i=1}^d \|P_m \WW_i\|^2_{L^2(\Ss^{n-1},\pi^*\E_{x_0})}.
\end{equation}
As a consequence, in order to obtain a contradiction in \eqref{equation:partial}, it is sufficient to exhibit a $\Gamma$ such that $h_1 \neq 0$ (where $h_1$ is given in \eqref{equation:decomposition}). Since $h_1 \in \ker(\pi_{\ker \imath_{\xi^\sharp}} {\pi_m}_*)$ and $\ker({\pi_m}_*) \subset \ker(\pi_{\ker \imath_{\xi^\sharp}} {\pi_m}_*)$, it is sufficient to show that the orthogonal projection of $\WW_1$ onto $\ker ({\pi_m}_*|_{L^2})$ is not zero, that is, it is sufficient to show that $\WW_1 = E^m_{\xi_0}(\pi_{1}^* \Gamma. \VV_1)$ has degree $\geq m+1$:

\begin{lemma}
\label{lemma:step2}
There exists $x_0 \in M$, $\xi_0 \in T_{x_0}^*M \setminus \{0\}$ and $\Gamma \in T^*_{x_0} M \otimes \End_{\mathrm{sk}}(\E_{x_0})$ such that $\Deg(E^m_{\xi_0}(\pi_{1}^* \Gamma. \VV_1)) \geq m+1$.
\end{lemma}

\begin{proof}
For that, we will need the following claim:
	\begin{equation*}
		\forall 1\leq i \leq d,\quad \mathrm{we\,\, have} \quad \Deg(\VV_i) \geq m+1.
	\end{equation*}
Indeed, assuming the contrary, since $\VV_i$ has opposite parity as $m$ this would force $\Deg(\VV_i) \leq m-1$, that is, $\VV_i = \pi_{m-1}^* \widetilde{\VV}_i$ for some $\widetilde{\VV}_i \in C^\infty(M, \otimes_S^{m - 1} T^*M)$. Recalling $\pi_m^* \Delta_0  u_i = \X \VV_i$ by \eqref{eq:u_iVV_i} and using the relation \eqref{equation:x}, this implies $\Delta_0 u_i = D_{\E} \widetilde{\VV}_i$. Hence $\pi_{\ker D_0^*} \Delta_0 u_i = 0$, implying $u_i = 0$ which is a contradiction. 

Now we select $x_0$ according to $\VV_1$, that is, we take $x_0$ such that at this point, the degree of $\VV_1$ in the fibre over $x_0$ is $\geq m + 1$. This also implies that $\pi_1^* \Gamma. \VV_1$ has degree $\geq m+1$ (actually, the degree is at least $m+2$ but we do not need this) for some choice of $\Gamma \in T_{x_0}^*M \otimes \End_{\mathrm{sk}}(\E_{x_0})$ by Lemma \ref{lemma:multiplication-surjective-bundle} applied with $k=1$. Then, by Lemma \ref{lemma:restriction-1}, we know that there exists a $\xi_0 \in T^*_{x_0}M \setminus \{0\}$ such that $\mathrm{deg}(\pi_1^*\Gamma. \VV_1|_{\Ss^{n-2}_{\xi_0}}) \geq m+1$ and then it suffices to apply the extension Lemma \ref{lemma:extension} to get that 
\[
\mathrm{deg}\left(E^m_{\xi_0}(\pi_1^*\Gamma. \VV_1|_{\Ss^{n-2}_{\xi_0}}) \right) \geq m+1,
\]
concluding the proof when $n \geq 3$.
\end{proof}
This allows to complete the proof of Theorem \ref{theorem:genericity-connection}:

\begin{proof}[Proof of Theorem \ref{theorem:genericity-connection}]
Define $\mc{S}_m$ to be the set of smooth unitary connections with $s$-injective twisted generalized X-ray transform $\Pi^{\nabla^{\E}}_m$. It follows from Lemma \ref{lemma:differentiabilty} that this set is open with respect to the $C^{k_0}$-topology (for some $k_0 \gg 1$ large enough) namely, for all $\nabla^{\E} \in \mc{S}_m$, there exists $\eps > 0$ such that for all smooth ${\nabla'}^{\E}$ with $\|\nabla^{\E} - {\nabla'}^{\E}\|_{C^{k_0}} < \eps$, ${\nabla'}^{\E} \in \mc{S}_m$.

In order to show density, let $\eps > 0$ and $\nabla^{\E}$ be a smooth unitary connection not in $\mc{S}_m$. By \cite{Cekic-Lefeuvre-20} we know that the set of connections for which $\ker (\pi^*\nabla^{\E})_X|_{C^\infty} = \{0\}$ is dense, so we may assume that $\nabla^{\E}$ satisfies this property. By the perturbative argument above (Lemmas \ref{lemma:second-variation}, \ref{lemma:step2} and \eqref{equation:partial}), there exists $\Gamma \in C^\infty(M,T^*M \otimes \End_{\mathrm{sk}}(\E))$ such that for all $\tau > 0$ small enough, $\nabla^{\E} + \tau \Gamma$ satisfies $\dim (\ker \Pi_m^{\nabla^{\E} + \tau \Gamma}) < \dim (\ker \Pi_m^{\nabla^{\E}})$; as usual, all the kernels here and in what follows are assumed to be restricted to the space of divergence free tensors. Take $\tau > 0$ small enough such that $\Gamma_1 := \tau \Gamma$ has $C^{k_0}$ norm strictly smaller than $\beta := \tfrac{\varepsilon}{\dim (\ker \Pi_m^{\nabla^{\E}})}$. Iterating finitely many times this construction ($N$ times where $N \leq  \dim (\ker \Pi_m^{\nabla^{\E}})$), we can find $\Gamma_1,...,\Gamma_N \in C^\infty(M,T^*M \otimes \End_{\mathrm{sk}}(\E))$, each one with $C^{k_0}$ norm less than $\beta$, such that
\[
	\ker \Pi_m^{\nabla^{\E} + \Gamma_1 + ... + \Gamma_N}=\left\{0\right\},
\]
that is $\nabla^{\E} + \Gamma_1+...+\Gamma_N \in \mc{S}_m$ and also
\[
\|(\nabla^{\E} + \Gamma_1 + ... + \Gamma_N) - \nabla^{\E}\|_{C^{k_0}} < \eps,
\]
which proves density and concludes the proof of Theorem \ref{theorem:genericity-connection}.
\end{proof}

\begin{remark}\rm
\label{remark:elliptic}
As mentioned in the introduction, Lemma \ref{lemma:second-variation} and \eqref{equation:equality} show that the second order derivative $\dd^2\lambda(\Gamma,\Gamma)$ is of the form $\langle B \Gamma, \Gamma \rangle_{L^2}$, where $B$ is some pseudodifferential operator (and the same will occur in the metric case). However, this operator is \emph{a priori} not elliptic. More precisely, the proof only shows that there exists a point $x_0 \in M$ (where $\VV_i$ has degree $m+1$) where the principal symbol of $B$ is non-zero. Had we been able to show the ellipticity of $B$, we would have obtained that locally the space of connections (up to gauge) with non-injective X-ray transform is finite-dimensional (and its tangent space would have been equal to the kernel of $B$).
\end{remark}

\begin{remark}\rm
\label{remark:n=2-twisted}
	Our proof does not give generic injectivity when $n = 2$. More precisely, Lemma \ref{lemma:step2} does not work in that case, since $E^m_{\xi_0}(\pi_{1}^* \Gamma. \VV_1)$ \emph{always} has degree equal to $m$. Therefore, the equality \eqref{equation:partial} always holds and our proof shows that the pseudodifferential operator $L_{\VV_i} - Q_{\VV_i}^*\Pi_m^{-1} Q_{\VV_i}$ appearing in \eqref{equation:equality} is in fact of order $-2$, as opposed to the case $n \geq 3$ where we show that this operator is strictly of order $-1$. However, we believe that in the former case the second variation should also be non-zero, and that this should be provable using the method in \S \ref{section:pdo} directly.
\end{remark}

To conclude this paragraph, we point out that Theorem \ref{theorem:genericity-connection} also allows to answer positively to the \emph{tensor tomography problem} for generic connections:

\begin{corollary}[of Theorem \ref{theorem:genericity-connection}]
\label{corollary:above}
There exists $k_0 \gg 1$ such that the following holds. Let $(M,g)$ be a smooth Anosov manifold of dimension $\geq 3$ and let $\pi_{\E} : \E \rightarrow M$ be a smooth Hermitian vector bundle. There exists a residual set $\mc{S}' \subset \mc{A}_{\E}$ (for the $C^{k_0}$-topology) such that if $\nabla^{\E} \in \mc{S}'$, the following holds: let $f, u \in C^\infty(SM,\pi^*\E)$ such that $\X u = f$ and $\mathrm{deg}(f) < \infty$; then $\mathrm{deg}(u) \leq \max(\mathrm{deg}(f)-1,0)$.
\end{corollary}

We let $\mc{U}$ be the open dense set of unitary connections on $\nabla^{\E}$ such that $\X$ has no resonances at $z=0$ (density follows from \cite{Cekic-Lefeuvre-20}). The set $\mc{S}'$ in Corollary \ref{corollary:above} is the intersection of the set $\mc{S}$ in Theorem \ref{theorem:genericity-connection} with $\mc{U}$.

\begin{proof}
Let $\nabla^{\E} \in \mc{S}' := \mc{U} \cap \mc{S}$. Consider the transport equation $\X u = f$, where $\mathrm{deg}(f) < \infty$ and both $u$ and $f$ are smooth. We aim to show that $\mathrm{deg}(u) \leq \max(\mathrm{deg}(f)-1,0)$ and $u \equiv 0$ if $\deg(f) = 0$.

Up to decomposing $u$ and $f$ into odd and even parts, we can already assume that $f$ is even and $u$ is odd for instance. Let $m := \mathrm{deg}(f)$. Then $f = \pi_m^* \widetilde{f}$ and by applying ${\pi_m}_*\mc{I}$ to the transport equation, we obtain
\[
\Pi_m^{\nabla^{\E}} \widetilde{f} = {\pi_m}_* \mc{I} \pi_m^* \widetilde{f} = {\pi_m}_* \mc{I} \X u = 0.
\]
By $s$-injectivity of $\Pi_m^{\nabla^{\E}}$, we get that $\widetilde{f} = 0$ if $m=0$, or $\widetilde{f}=D_{\E}p$ for some twisted tensor $p \in C^\infty(M,\otimes^{m-1}_S T^*M \otimes \E)$ if $m \geq 1$. In the former case, we get $\X u = 0$ and thus $u \equiv 0$ as $\X$ has no resonance at $z=0$. In the latter case, we get $\X(u-\pi_{m-1}^*p) = 0$ (using \eqref{equation:xe}) and thus $u = \pi_{m-1}^*p$ has degree $\leq m-1$. This concludes the proof.
\end{proof}

\subsection{Endomorphism case}

We conclude this section with a discussion of the endomorphism case. More precisely, if $\E \rightarrow M$ is a Hermitian vector bundle, we let $\End(\E) \rightarrow M$ be the vector bundle of endomorphisms. If $\nabla^{\E}$ is a unitary connection on $\E$, it induces a canonical connection $\nabla^{\End(\E)}$ on $\End(\E)$ defined so that it satisfies the Leibniz rule:
\[
\left[\nabla^{\End(\E)}u\right].f := \nabla^{\E}(u(f)) - u(\nabla^{\E}f),
\]
for all $f \in C^\infty(M,\E)$, $u \in C^\infty(M,\End(\E))$. Similarly to \S\ref{sssection:t-g-x-ray}, one can define a twisted X-ray transform $\Pi_m^{\End(\E)}$ with values in the endomorphism bundle $\End(\E)$. More precisely, the operator $\pi^* \nabla^{\End(\E)}_X$ always contains $\C \cdot \mathbbm{1}_{\E}$ in its kernel and its kernel is generically reduced to $\C. \mathbbm{1}_{\E}$ (see \cite{Cekic-Lefeuvre-20}, such a connection is also said to be \emph{opaque}). We then set:
\[
\Pi_m^{\End(\E)} := {\pi_m}_*( \mathbf{R}_+ + \mathbf{R}_- + \Pi_{\C. \mathbbm{1}_{\E}} ) \pi_m^*,
\]
where $\Pi_{\C. \mathbbm{1}_{\E}}$ denotes the $L^2$-orthogonal projection onto $\C. \mathbbm{1}_{\E}$ and $\mathbf{R}_\pm$ is (the opposite of) the holomorphic part of the resolvents of $\pi^*\nabla^{\End(\E)}_X$ at $z=0$.

For $m=1$, the solenoidal injectivity of the operator $\Pi_1^{\End(\E)}$ appears to be crucial when studying the \emph{holonomy inverse problem} on Anosov manifolds, namely: to what extent does the trace of the holonomy of a connection along closed geodesics determine the connection? We proved in a companion paper \cite{Cekic-Lefeuvre-21-1} that this problem is locally injective near a connection $\nabla^{\End(\E)}$ such that its induced operator $\Pi_1^{\End(\E)}$ is s-injective. Similarly to Theorem \ref{theorem:genericity-connection}, this turns out to be a generic property:

\begin{theorem}
\label{theorem:genericity-connection-endomorphism}
There exists $k_0 \gg 1$ such that the following holds. Let $(M,g)$ be a smooth Anosov manifold of dimension $n \geq 3$, $\pi_{\E} : \E \rightarrow M$ be a smooth Hermitian vector bundle and let $m \in \Z_{\geq 0}$. Moreover, assume that the X-ray transform $I_m$ with respect to $(M, g)$ is {\rm s}-injective. Then, there exists an open and dense set $\mc{S}_m \subset \mc{A}_{\E}$ (for the $C^{k_0}$-topology) of unitary connections with {\rm s}-injective twisted generalized X-ray transform $\Pi_m^{\nabla^{\End(\E)}}$ on the endomorphism bundle. 
\end{theorem}

Note that the main difference with Theorem \ref{theorem:genericity-connection} is that we need to assume that $I_m$ is s-injective; this is known for $m=0,1$ on all Anosov manifolds \cite{Dairbekov-Sharafutdinov-03} and this is a generic condition with respect to the metric by our Theorem \ref{theorem:genericity-metric}.

\begin{proof}
We just point out the main differences with the proof of Theorem \ref{theorem:genericity-connection}. If $\nabla^{\E}$ is a fixed unitary connection and $\nabla^{\E} + \Gamma$ (for $\Gamma \in C^\infty(M,T^*M \otimes \End_{\mathrm{sk}}(\E))$) is a perturbation, the induced connection on the endomorphism bundle is $\nabla^{\End(\E)} + [\Gamma,\bullet]$. Then, in the computations of \S \ref{ssection:variations_ground_state}, each time that a term $\pi_1^* \Gamma. \VV_i$ appears, it has to be replaced by $[\pi_1^*\Gamma, \VV_i]$ and the $\VV_i$'s are now elements of $C^\infty(SM, \pi^* \End(\E))$, where $\VV_i$ satisfy a version of \eqref{eq:u_iVV_i} for $\X := \pi^*\nabla^{\End}_X$.

Now, each $\VV_i$ can be (uniquely) decomposed as $\VV_i = f_i. \mathbbm{1}_{\E} + \VV_i^\bot$, where $f_i \in C^\infty(SM)$ and $\VV_i^\bot \in C^\infty(SM,\pi^*\End(\E))$ is a (pointwise) trace-free endomorphism-valued section. One still has that $\X \VV_i$ is of degree $m$ and $\VV_i$ is of degree $\geq m+1$. In fact, we claim that $\VV_i^\bot$ is of degree $\geq m+1$. Indeed, assume that this is not the case, that is $\deg(\VV_i^\bot) \leq m$; then $f_i$ has to be of degree $\geq m+1$ and
\[
\X \VV_i = (X f_i). \mathbbm{1}_{\E} + \X \VV_i^\bot
\]
is of degree $m$. As $(X f_i). \mathbbm{1}_{\E}$ and $\X \VV_i^\bot$ are pointwise orthogonal as elements of $\pi^*\End(\E)$ (since $\X \VV_i^\bot$ is trace-free), this forces each of them to be of degree $\leq m$ and thus $Xf_i$ is of degree $\leq m$, and $\deg(f_i) \geq m+1$. But then the assumption that $I_m$ is s-injective rules out this possibility.

Lemma \ref{lemma:step2} is then modified in the obvious way: one chooses a point $x_0$ such that $\Deg(\VV^\bot_i) \geq m+1$ and it suffices to find a $\Gamma \in T_{x_0}^*M \otimes \End_{\mathrm{sk}}(\E_{x_0})$ such that
\[
[\pi_1^*\Gamma, \VV_i] = [\pi_1^*\Gamma,\VV_i^\bot]
\]
has degree $\geq m+1$. For that, we choose an orthonormal basis $(\mathbf{e}_1,...,\mathbf{e}_r)$ of $\E_{x_0}$ and write in that basis $\VV_i^\bot = (m_{j\ell})_{1 \leq j,\ell \leq r}$. By assumption, there is an element $m_{j \ell}$ with degree $\geq m+1$. Without loss of generality, we can assume it is in the first column $\ell=1$. If it is obtained for some $j_0 \neq 1$, then taking a \emph{real-valued} $\alpha \in \Omega_1$ such that $\alpha. m_{j_0 1}$ has degree $\geq m+1$ (which is possible by Lemma \ref{lemma:multiplication-surjective}), and setting $\Gamma := i\alpha \times \mathbf{e}_1 \otimes \mathbf{e}_1^*$, we get easily that $[\pi_1^*\Gamma,\VV_i^\bot]$ has degree $\geq m+1$.

If it is obtained for $j_0=1$, then we write
\[
\VV_i^\bot = \begin{pmatrix} m_{11} & b^\top \\ c & d \end{pmatrix},
\]
where $m_{11}\in C^\infty(SM)$ has degree $\geq m+1$, $b,c$ are vectors of length $r-1$ and $d \in C^\infty(S_{x_0}M,\End(\C^{r-1}))$. Note that $\Tr(\VV_i^\bot)=0=m_{11} + \Tr(d)$. Moreover, writing $d = (m_{j\ell})_{2 \leq j,\ell \leq r}$, there is an element $m_{j_0 j_0}$ on the diagonal of $d$ such that, if $f_{\geq m+1}$ denotes the projection of a function $f$ onto Fourier modes of degree $\geq m+1$, one has $(m_{j_0j_0})_{\geq m+1} \neq (m_{11})_{\geq m+1}$ (indeed, if not, this would contradict $\Tr(\VV_i^\bot)=0$). Without loss of generality, we can assume that $j_0=2$. Then, taking
\[
\Gamma = \begin{pmatrix} 0 & \gamma^\top \\ -\gamma & 0 \end{pmatrix},
\]
where $\gamma$ is a vector of length $r-1$ and $\gamma^\top = (\alpha, 0, ..., 0)$ and $\alpha \in \Omega_1$ is real-valued, we obtain:
\[
[\pi_1^*\Gamma,\VV_i^\bot] = \begin{pmatrix} * & \gamma^\top d - m_{11}.\gamma^\top \\ * & * \end{pmatrix}, \quad \gamma^\top d - m_{11}.\gamma^\top = \left((m_{22}-m_{11}).\alpha, *, ..., *\right).
\]
By assumption, $m_{22}-m_{11}$ has degree $\geq m+1$ and it thus suffices to choose a real $\alpha \in \Omega_1$ such that $(m_{22}-m_{11}).\alpha$ has degree $\geq m+1$. The existence of such an $\alpha$ is once again guaranteed by Lemma \ref{lemma:multiplication-surjective}. This completes the proof.
\end{proof}

\section{Generic injectivity with respect to the metric}

\label{section:genericity-metric}

We now prove Theorem \ref{theorem:genericity-metric}. As we shall see, the computations follow from the same strategy as in the connection case, except that they are more involved.

\subsection{Preliminary computations} A first point to address is that the unit tangent bundle now varies if we perturb the metric.

\subsubsection{Scaling the geodesic vector fields on $SM$}

The metric $g := g_0$ is fixed and we consider a smooth variation $(g_\tau)_{\tau \in (-1,1)}$ of the metric. Each $\tau \in (-1,1)$ defines a unit tangent bundle
\[
SM_\tau := \left\{ (x,v) \in TM ~|~ g_\tau(v,v)=1\right\} \subset TM,
\]
and we write $SM := SM_0$.
Each metric $g_\tau$ induces a geodesic vector field $H_\tau$ defined on the whole tangent bundle $TM$ (which is tangent to $SM_\tau$, for every $\tau \in (-1,1)$). We let
\begin{equation}\label{eq:scalingMap}
\Phi_\tau : SM \rightarrow SM_\tau, ~~~ (x,v) \mapsto \left(x,\frac{v}{|v|_{g_\tau}}\right),
\end{equation}
be the natural projection onto $SM_\tau$. We consider the family $X_\tau := \Phi_\tau^* H_\tau  = (\Phi_\tau^{-1})_* H_\tau \in C^\infty(SM,T(SM))$, which depends smoothly on $\tau$ and is defined so that $X_{\tau=0} = X_0$ is the geodesic vector field of the metric $g_0$. Note that $(X_\tau)_{\tau \in (-1,1)}$ is a smooth family of Anosov vector fields on $SM$. In what follows, when clear from context we will drop the subscript when referring to derivatives at $\tau = 0$. We will use $\sharp : T^*(SM) \to T(SM)$ to denote the musical isomorphism with respect to the Sasaki metric.

\begin{lemma}
\label{lemma:projection2}
We have: 
\begin{equation}
\label{equation:dot-x}
\begin{split}
\dot{X}_{\tau = 0}& = - \frac{1}{2} \pi_2^* \dot{g} \cdot X + \frac{1}{2} J \left[\nabla_{\V}(\pi_2^*\dot{g}) + \nabla_{\HH}(\pi_2^*\dot{g}) \right] \\
& \hspace{2cm} - J \left( (\nabla^{\mathrm{Sas}}_X \pi^*_{2,\mathrm{Sas}} \dot{g})(X, \pi_{\HH}(\bullet))\right)^\sharp + \pi^*_{2,\mathrm{Sas}}\dot{g}(X, \pi_{\HH}(\bullet))^\sharp.
\end{split}
\end{equation}
\end{lemma}

\begin{proof}
	Write $\beta_\tau(x, v)(\xi) = \langle{v, d\pi(x, v) \xi}\rangle_{g_\tau(x)}$ for the contact $1$-form on $SM_\tau$. Writing $\alpha_\tau := \Phi_\tau^*\beta_\tau$ and $\alpha := \beta_0$, and using $\pi \circ \Phi_{\tau} = \pi$, we obtain for any $\xi \in T_{(x, v)}SM$:
	\[
	\alpha_\tau(x, v)(\xi) = |v|_{g_\tau}^{-1} g_\tau(v, \dd \pi(x,v)(\xi)).
	\]
	Differentiating in $\tau$ and restricting to $\tau = 0$, we obtain the relation:
	\begin{equation}\label{eq:alphadot0}
	\dot{\alpha} = -\frac{1}{2} \pi_2^*\dot{g} \cdot \alpha + \dot{g}(v, \dd \pi(x,v)(\bullet)).
	\end{equation}
	The pullback vector field $X_\tau$ is uniquely determined by the relations: $\iota_{X_\tau} \alpha_\tau = 1$ and $\iota_{X_\tau} d\alpha_\tau = 0$. Differentiating, we get $\iota_{\dot{X}} \alpha + \iota_X \dot{\alpha} = 0$ and thus using \eqref{eq:alphadot0}, $\alpha(\dot{X}) = -\frac{1}{2} \pi_2^*\dot{g}$. Since we can decompose $\dot{X} = \alpha(\dot{X}) X + Y$ for some $Y$ orthogonal to $X$ this gives the first term in \eqref{equation:dot-x}. It remains to compute $Y$. 
	
	For that, we introduce the $1$-form, defined for $(x, v) \in SM$ as
	\[
	A(x, v)(\bullet) :=  \dot{g}(v, \dd \pi(x,v)(\bullet)) = \pi_{2,\mathrm{Sas}}^* \dot{g}(X,\bullet),
	\]
	using the Sasaki lift introduced in \eqref{equation:lift-sasaki}. The first step is to compute $\iota_X dA$ and we claim:
	\begin{equation}\label{eq:iota-XdA}
	\iota_X dA = - d (\pi_2^*\dot{g}) + (\nabla^{\mathrm{Sas}}_X \pi^*_{2,\mathrm{Sas}} \dot{g})(X, \pi_{\HH_\mathrm{tot}}(\bullet)) - \pi^*_{2,\mathrm{Sas}}\dot{g} (X,J \pi_{\V}(\bullet)).
	\end{equation}
	Recall here that $\HH_{\mathrm{tot}} = \HH \oplus \R \cdot X$ denotes the total horizontal space, as explained in \S\ref{ssection:elementary}, and $\pi_{\HH_{\mathrm{tot}}}$ is the orthogonal projection onto this space (with respect to the Sasaki metric).

	Note that $A$ defines a $1$-form on $TM$ and we will first compute $\iota_X dA$ on $TM$. Then $\iota_X dA$ on $SM$ is just the restriction. For $W, Z \in C^\infty(TM, T(TM))$, we have the formula:
	\begin{equation}
	\label{equation:formulas}
	\begin{split}
	dA(W, Z) = W \cdot \left( \pi^*_{2,\mathrm{Sas}}\dot{g}(X, Z) \right) - Z \cdot \left(\pi^*_{2,\mathrm{Sas}}\dot{g}(X, W)\right) - \pi^*_{2,\mathrm{Sas}}\dot{g}(X,[W, Z]).
	\end{split}
	\end{equation}
	We now fix a point $p \in M$ and take a geodesic orthonormal frame $(E_1,...,E_n)$ around $p$, i.e. such that $\nabla_{E_i} E_j(p) = 0$. Let $X_i$ be the horizontal lift of $E_i$. We have that $\nabla_{X_i}^{\mathrm{Sas}} X$ is vertical (see \cite[Lemma 1.25]{Paternain-99}) and $\nabla^{\mathrm{Sas}}_X X_i$ is also vertical at $p$ (as $\pi$ is a Riemannian submersion). Hence by \eqref{equation:formulas}, at the point $p$:
	\begin{equation}\label{eq:iota-X_i}
	\begin{split}
	\iota_X dA(X_i) &= (\nabla^{\mathrm{Sas}}_X \pi^*_{2,\mathrm{Sas}} \dot{g})(X,X_i) + \pi^*_{2,\mathrm{Sas}} \dot{g}(\nabla^{\mathrm{Sas}}_X X, X_i) +\pi^*_{2,\mathrm{Sas}} \dot{g}(X, \nabla^{\mathrm{Sas}}_X  X_i) \\
	& - X_i \cdot \pi_2^*\dot{g} - \pi^*_{2,\mathrm{Sas}} \dot{g}(X, \underbrace{[X,X_i]}_{= \nabla^{\mathrm{Sas}}_X X_i - \nabla^{\mathrm{Sas}}_{X_i} X})\\
	& =(\nabla^{\mathrm{Sas}}_X \pi^*_{2,\mathrm{Sas}} \dot{g})(X,X_i) - X_i \cdot \pi_2^*\dot{g}.
	\end{split}
	\end{equation}
	Introducing $Y_i := J X_i$, it can also be checked that $[X,Y_i] = JY_i = -X_i$ at the point $p$ (this is an immediate consequence of the fact that $[X_i, Y_j] = 0$ at $p$, see \cite[Exercise 1.26]{Paternain-99}). Thus by \eqref{equation:formulas}, at $p$:
	\begin{equation}\label{eq:iota-Y_i}
	\begin{split}
	\iota_X dA (Y_i) & = X \cdot \pi^*_{2,\mathrm{Sas}} \dot{g}(X,Y_i)- Y_i \cdot \left( \pi^*_{2,\mathrm{Sas}} \dot{g}(X,X)\right) - \pi^*_{2,\mathrm{Sas}} \dot{g}(X,[X,Y_i])\\
	&  = - Y_i \cdot \pi_2^*\dot{g} -  \pi^*_{2,\mathrm{Sas}} \dot{g}(X,JY_i).
	\end{split}
	\end{equation}
	Combining \eqref{eq:iota-X_i} and \eqref{eq:iota-Y_i} immediately yields \eqref{eq:iota-XdA} and proves the claim.

	Hence, combining \eqref{eq:iota-XdA} with \eqref{eq:alphadot0}, we get
	\begin{align*}
	\iota_X d \dot{\alpha} &= -\frac{1}{2} X(\pi_2^*\dot{g}).\alpha + \frac{1}{2} d(\pi_2^*\dot{g}) + \iota_X dA\\
	&= - \frac{1}{2}  X (\pi_2^* \dot{g}). \alpha - \frac{1}{2} d (\pi_2^* \dot{g}) + (\nabla^{\mathrm{Sas}}_X \pi^*_{2,\mathrm{Sas}} \dot{g})(X, \pi_{\HH_{\mathrm{tot}}}(\bullet)) - \pi^*_{2,\mathrm{Sas}} \dot{g} (X,J \pi_{\V}(\bullet)).
	\end{align*}
Using $\iota_X d\dot{\alpha} + \iota_{\dot{X}} d\alpha =\iota_X d\dot{\alpha} + \iota_Y  d\alpha = 0$ together with \eqref{eq:contact-complex-sasaki}, we get:
\[
\begin{split}
	Y  = -J (\iota_X d \dot{\alpha})^\sharp  & =  \frac{1}{2} J \left( \nabla_{\V} \pi_2^* \dot{g} + \nabla_{\HH} \pi_2^* \dot{g}\right) - J \left( (\nabla^{\mathrm{Sas}}_X \pi^*_{2, \mathrm{Sas}} \dot{g})(X, \pi_{\HH}(\bullet))\right)^\sharp \\
	& + J \left( \pi^*_{2,\mathrm{Sas}} \dot{g} (X,J \pi_{\V}(\bullet)) \right)^\sharp .
	\end{split}
	\]
		Note here that we used $(\nabla^{\mathrm{Sas}}_X \pi^*_{2, \mathrm{Sas}} \dot{g})(X, \pi_{\HH}(\bullet)) = (\nabla^{\mathrm{Sas}}_X \pi^*_{2, \mathrm{Sas}} \dot{g})(X, \pi_{\HH_{\mathrm{tot}}}(\bullet)) - X(\pi_2^*\dot{g}). \alpha$, valid since $X(\pi_2^*\dot{g}) = (\nabla^{\mathrm{Sas}}_X\pi_{2, \mathrm{Sas}}^*\dot{g})(X, X)$. The last term in the expression for $Y$ can be simplified as $\pi^*_{2,\mathrm{Sas}} \dot{g}(X,\pi_{\HH}(\bullet))^\sharp$ (which boils down to the identity $J \pi_{\mathbb{V}} J = -\pi_{\mathbb{H}}$ on $\ker \alpha$), which completes the proof.
\end{proof}

The last two terms of \eqref{equation:dot-x} vanish for a conformal perturbation. We introduce the differential operator $\Lambda^{\mathrm{conf}} \in \mathrm{Diff}^1(SM,\C \rightarrow T_{\mathbb{C}}(SM))$ of order one
\begin{equation}\label{eq:LambdaDef}
	\Lambda^{\mathrm{conf}}(f) := \frac{1}{2}\left(-f X + J\left[ \nabla_{\V}f + \nabla_{\HH}f \right]\right).
\end{equation}
We also introduce $\Lambda^{\mathrm{aniso}} \in \mathrm{Diff}^1(SM, \otimes^2_S T^*(SM) \to T_{\mathbb{C}}(SM))$ of order one by:
\begin{equation}
\label{equation:LambdaAniso}
\Lambda^{\mathrm{aniso}}(f) :=  - J \left( (\nabla^{\mathrm{Sas}}_X f)(X, \pi_{\HH}(\bullet))\right)^\sharp + f(X, \pi_{\HH}(\bullet))^\sharp
\end{equation}
By construction $\dot{X} = \Lambda^{\mathrm{conf}}(\pi_2^*f) + \Lambda^{\mathrm{aniso}}(\pi_{2,\mathrm{Sas}}^*f)$. In order to manipulate compact notations, we will write $\dot{X} := \Lambda(\pi_2^*f)$, although there is some abuse of notations here as there are two distinct lifts of $f$ to $SM$.

We now compute the symbols of $\Lambda^{\mathrm{conf,aniso}}$. They will be useful in the perturbation arguments in the following sections. 

\begin{lemma}
\label{lemma:vector}
For any $(x, v, \xi) \in T(SM)$, we have:
\[
	\sigma(\Lambda^{\mathrm{conf}})(x, v, \xi) = \frac{i}{2} J(x, v) (\pi_{\HH}\xi^\sharp + \pi_{\V} \xi^\sharp) = \frac{i}{2} J(x, v) \big(\xi^\sharp - \xi(X(x, v)). X(x, v)\big).
\]
\end{lemma}
\begin{proof}
	Consider the Lagrangian state $e^{i\frac{S}{h}} f$, where $S(x,v) = 0, \dd S(x,v) = \xi$ and $f(x,v) = 1$. By \eqref{eq:LambdaDef}, we compute
	\begin{equation}\label{eq:LambdaLagrangian}
		\Lambda^{\mathrm{conf}}(e^{i\frac{S}{h}}f) = -\frac{1}{2}e^{i \frac{S}{h}}\Big[f X - \frac{i}{h} fJ(\nabla_{\mathbb{H}}S + \nabla_{\mathbb{V}} S) - J(\nabla_{\HH}f + \nabla_{\V} f)\Big].
	\end{equation}
	We may directly read off the principal symbol from this expression:
	\[
	\begin{split}
	\sigma(\Lambda^{\mathrm{conf}})(x, v, \xi) & = \lim_{h \to 0} h \Lambda^{\mathrm{conf}}(e^{i \frac{S}{h}} f) (x, v) \\
	&= \frac{i}{2} J(\nabla_{\HH}S + \nabla_{\V}S)(x, v) = \frac{i}{2} J(x, v)(\pi_{\HH}\xi^\sharp + \pi_{\V}\xi^\sharp).
	\end{split} \]
\end{proof}

We have:

\begin{lemma}
\label{lemma:LambdaAnisoSymbol}
For any $(x,v,\xi) \in T(SM)$ and $f \in \otimes^2_S T^*_{(x,v)}(SM)$, we have:
\[
\sigma(\Lambda^{\mathrm{aniso}})(x,v,\xi)f = - i \langle \xi, X(x,v) \rangle. J(x, v) \big(f(X(x, v),\pi_{\HH}(\bullet))\big)^\sharp
\]
\end{lemma}

\begin{proof}
We see that it suffices to compute the principal symbol of the first term in \eqref{equation:LambdaAniso} as the second term is of lower order. We take a Lagrangian state $e^{i \frac{S}{h}} f$ with $S(x,v) = 0, \dd S(x,v)= \xi$ and $f \in C^\infty(SM,\otimes^2_S T^*(SM)),f(x) =: f_0$. We have:
\[
\begin{split}
&\sigma(\Lambda^{\mathrm{aniso}})(x,v,\xi)f_0 \\
& = \lim_{h \to 0} h \Lambda^{\mathrm{aniso}}(e^{i\frac{S}{h}}f)(x,v) \\
& =  \lim_{h \to 0} - h J(x, v) \left( \frac{i}{h} XS. \left[f(X,\pi_{\HH}(\bullet))\right]^\sharp + e^{i \frac{S}{h}} \left[(\nabla^{\mathrm{Sas}}_X f)(X, \pi_{\HH}(\bullet))\right]^\sharp \right)(x, v) \\
& = - i \langle \xi, X(x,v) \rangle. J(x, v) \big(f_0(X(x, v),\pi_{\HH}(\bullet))\big)^\sharp.
\end{split} 
\]
\end{proof}

Eventually, we compute the divergence of $\dot{X}$ in a geometric way. We prove:

\begin{lemma}
\label{lemma:divergence}
	The following formula holds:
	\begin{equation}\label{eq:div-X-dot-full}
		\divv(\dot{X}) = X \Big( \pi_0^* \Tr_{g_0} (\dot{g}) - \frac{n}{2}\pi_2^*\dot{g}\Big).
	\end{equation}
\end{lemma}

In local coordinates $(x_i)_{1 \leq i \leq n}$ where $g_0$ and $\dot{g}$ are identified with $n \times n$ symmetric matrices, we have $\Tr_{g_0}(\dot{g}) = \Tr(g_0^{-1} \dot{g})$.

\begin{proof}
	Write $\Omega_\tau := d\vol_{\mathrm{Sas}, g_\tau}$ for the Sasaki volume form of $g_\tau$ in $SM_\tau$ and $\Omega := \Omega_0$. Write $\mc{J}_\tau$ for the Jacobian of $\Phi_\tau$ (where we recall $\Phi_\tau$ was introduced in \eqref{eq:scalingMap}), i.e. $\Phi_\tau^*\Omega_\tau = \mc{J}_\tau \Omega$. Observe that
	\[-\divv(X_\tau) \Omega = \Lie_{X_\tau} \Omega = \Phi_\tau^*(\Lie_{H_\tau}(\mc{J}_\tau^{-1}\circ \Phi_\tau^{-1}.\Omega_\tau)) = \mc{J}_\tau. X_\tau(\mc{J}_\tau^{-1}). \Omega.\]
	It follows that $\divv(X_\tau) = -\mc{J}_\tau. X_\tau (\mc{J}_\tau^{-1}) =  \mc{J}_\tau^{-1}. X_\tau \mc{J}_\tau$ (the second equality follows from the product rule) and differentiating at $\tau = 0$ and using $\mc{J}_0 = 1$:
	\begin{equation}\label{eq:div-X-dot}
		\divv(\dot{X}) = X\dot{\mc{J}}.
	\end{equation}
	In what follows we compute $\mc{J}_\tau$. We will use that
	\[
		d\vol_{\mathrm{Sas}, g_\tau}(x, v) = \pi^*d\vol_{g_\tau}(x, v) \wedge d\vol_{S_xM_\tau} (v).
	\]
We need the following auxiliary lemma:
	
	\begin{lemma}\label{lemma:linearalgebra}
		Let $A_0$ and $A_1$ be two symmetric, positive definite matrices, and denote by $\mathbb{S}^{n - 1}_{A_i} = \{x \in \mathbb{R}^n \mid \langle{A_ix, x}\rangle = 1\} \subset \mathbb{R}^n$ the unit sphere with respect to the metric induced by $A_i$; denote by $d\vol_{A_i}$ the induced volume form on $\mathbb{S}^{n-1}_{A_i}$. If $R(x) = \frac{x}{\sqrt{\langle{A_1x, x}\rangle}}$ is the scaling map between the two spheres, then for $x \in \mathbb{S}^{n-1}_{A_0}$:
		\begin{equation}\label{eq:scaling-linear-algebra}
			R^*d\vol_{A_1}(x) = \Big(\sqrt{\frac{\det A_1}{\det A_0}}. \langle{A_1x, x}\rangle^{-\frac{n}{2}}\Big) d\vol_{A_0}(x).
		\end{equation}
	\end{lemma}
	\begin{proof}
		We first show the claim for $A_0 = \id$. Denote $A_1 = A$ and write $R^*d\vol_{A} = j.d\vol_{\mathbb{S}^{n-1}}$ for some function $j$ on $\mathbb{S}^{n-1}$. Observe that $d\vol_A(x) = \sqrt{\det A}. \imath_{\sum_i x_i e_i} |dx|$ for $x \in \mathbb{S}^{n-1}_A$, where we write $|dx| = dx_1 \wedge \dotso \wedge dx_n$, $e_i$ for the standard basis vectors of $\mathbb{R}^n$ identified with $\partial_i$; $\sum_i x_ie_i = x$ is the outer unit normal to $\mathbb{S}^{n-1}_A$ at $x$. It is straightforward to compute:
		\[\forall i = 1, \dotso, n - 1, \quad d_{e_n}R(e_i) = \partial_iR(e_n) = \frac{1}{\sqrt{A_{nn}}} \left(e_i - e_n. \frac{A_{in}}{A_{nn}}\right).\]
		This shows the claim at $x = e_n$:
		\begin{align*}
			& R^*(d\vol_A)(e_n)(e_1, \dotso, e_{n - 1}) \\
			&= \sqrt{\det A}. \imath_{\frac{e_n}{\sqrt{A_{nn}}}} dx_1 \wedge \dotso \wedge dx_n\left(..., \frac{1}{\sqrt{A_{nn}}} \left(e_i - e_n. \frac{A_{in}}{A_{nn}}\right),...\right)\\
			&= (-1)^{n-1}\sqrt{\det A}. \langle{Ae_n, e_n}\rangle^{-\frac{n}{2}}.
		\end{align*}
		For general $x \in \mathbb{S}^{n-1}$, consider a $B \in SO(n)$ such that $B(e_n) = x$. Using that $B:\mathbb{S}^{n-1}_{B^TAB} \to \mathbb{S}^{n-1}_A$ is an isometry and the previous computation, the formula \eqref{eq:scaling-linear-algebra} for $A_0 = \id$ follows.
		
		For general $A_0$, simply consider a linear coordinate change given by $B$ with $B^TA_0B = \id$ and apply the previous result to $A = B^TA_1B$. This completes the proof.
	\end{proof}
Using Lemma \ref{lemma:linearalgebra}, we see that in local coordinates:
\begin{align*}
	\Phi_\tau^*(\pi^*d\vol_{g_\tau} \wedge d\vol_{S_xM_\tau}) &= \pi^*d\vol_{g_\tau} \wedge \Phi_\tau^*(d\vol_{S_xM_\tau})\\
	&= \sqrt{\det g_\tau} |dx| \wedge \sqrt{\frac{\det g_\tau}{\det g_0}}. (\pi_2^*g_\tau)^{-\frac{n}{2}} d\vol_{S_xM}\\
	&=  \frac{\det g_\tau}{\det g_0}. (\pi_2^*g_\tau)^{-\frac{n}{2}}. \pi^*d\vol_{g_0} \wedge d\vol_{S_xM}.
\end{align*}
This shows that $\mc{J}_\tau = \frac{\det g_\tau}{\det g_0}. (\pi_2^*g_\tau)^{-\frac{n}{2}}$, so taking the derivative at $\tau = 0$ and using \eqref{eq:div-X-dot} completes the proof and shows \eqref{eq:div-X-dot-full}.
\end{proof}

In what follows, we will frequently use the operators, defined for a distribution $\VV \in \mc{D}'(SM)$:
\begin{equation}\label{eq:sandwich-auxiliary-operator}
	\forall u \in C^\infty(SM), \quad M_{\VV}u :=\VV. u, \quad \forall Z \in C^\infty(SM, T(SM)), \quad N_{\VV}Z := Z\VV. 
\end{equation}
Sometimes $u$ or $Z$ will also be singular, in which case we will have to justify the extension of the corresponding operator to such functions.

\subsubsection{Metric-dependent generalized X-ray transform}
\label{sssection:metric-dependent-X-ray}

For each $\tau \in (-\varepsilon, \varepsilon)$ small enough, we can consider the positive resolvent $\C \ni z \mapsto (-X_\tau -z)^{-1}$ and the negative resolvent $\C \ni z \mapsto (X_\tau -z)^{-1}$. Since we have $X_\tau = (\Phi_\tau)^*H_\tau$, the resolvents satisfy on $C^\infty(SM)$ that $(\pm X_\tau - z)^{-1} = (\Phi_{\tau})^*(\pm H_\tau - z)^{-1} (\Phi_{\tau})_*$ and therefore also 
\begin{equation}\label{eq:resolventRescaled}
	R_{\pm, 0}^{X_\tau} = (\Phi_{\tau})^*R_{\pm, 0}^{H_\tau}(\Phi_{\tau})_*, \quad \Pi^{X_\tau}_\pm = (\Phi_{\tau})^*\Pi_\pm^{H_\tau} (\Phi_{\tau})_*,
\end{equation}
where the superscript denotes the vector field with respect to which the resolvent is taken, $R_{\pm, 0}^\bullet$ is its holomorphic part at zero and $\Pi_\pm^\bullet$ is the orthogonal projection to the resonant space at zero. From now on we drop the $\bullet = X_\tau$ superscript and simply write $\tau$ instead. By \eqref{eq:Pi_pm} we have
\begin{equation}\label{eq:Pi_pm-tau}
\Pi_\pm^\tau = \langle \bullet, \mu_\tau \rangle \mathbf{1},
\end{equation}
where $\mathbf{1}$ is the constant function and $\mu_\tau$ is the pullback by $\Phi_\tau$ of the normalised Liouville measure on $SM_\tau$ such that $\langle \mathbf{1}, \mu_\tau\rangle = 1$.

Let $(\pi_m^{g_\tau})^*: C^\infty(M, \otimes_S^m T^*M) \to C^\infty(SM_\tau)$ be the canonical pullback operator; we write $\pi_m^* := (\pi_m^{g_0})^*$. We may then compute, for a symmetric $m$-tensor $f$:
\[
\begin{split}
(\Phi_\tau)^* (\pi_m^{g_\tau})^* f (x, v) & = (\pi_m^{g_\tau})^*f\left(x, \frac{v}{|v|_{g_\tau}}\right) \\
&= f_x\left(\frac{v}{|v|_{g_\tau}}, \dotsc, \frac{v}{|v|_{g_\tau}}\right) = |v|_{g_\tau}^{-m} \pi_m^*f(x, v).
\end{split}
\]
We denote $\chi_\tau(x, v) = |v|^{-m}_{g_\tau}$, so that by the previous equality:
\begin{equation}\label{eq:pushforwardRescaled}
	\chi_\tau \pi_m^* = (\Phi_\tau)^* (\pi_m^{g_\tau})^*, \quad (\pi_m)_* \chi_\tau = (\pi_m^{g_\tau})_* (\Phi_\tau)_*,
\end{equation}
where the lower star denotes the pushforward, that is the $L^2$ adjoint of the pullback operator. We are in position to introduce the generalised $X$-ray transform with respect to $H_\tau$ and study its properties under re-scaling by \eqref{eq:resolventRescaled} and \eqref{eq:pushforwardRescaled}:
\begin{equation}
\label{equation:pi}
\Pi_m^{\tau} := (\pi_{m}^{g_\tau})_* (R_{+,0}^{H_\tau} + R_{-,0}^{H_\tau} + \Pi_+^{H_\tau}) (\pi_{m}^{g_\tau})^* = {\pi_m}_* \chi_\tau (\underbrace{R_{+,0}^\tau + R_{-,0}^\tau + \Pi_+^\tau}_{\mc{I}^\tau :=}) \chi_\tau \pi_m^*.
\end{equation}
Moreover, the family
\[
(-\varepsilon, \varepsilon) \ni \tau \mapsto \Pi_m^\tau \in \Psi^{-1}(M, \otimes^m_S T^*M \rightarrow \otimes^m_S T^*M)
\]
depends smoothly on $\tau$ as stated in Lemma \ref{lemma:differentiabilty}.

We keep the same strategy as in \S\ref{section:genericity-connection} and define
\begin{equation}\label{eq:P_tau-metric-def}
P_\tau = \pi_{\ker D_0^*} \Delta_0 \Pi_m^\tau \Delta_0 \pi_{\ker D^*_0},
\end{equation}
where $\Delta_0$ is an elliptic, formally self-adjoint, positive, pseudodifferential operator of order $k>1/2$ with diagonal principal symbol that induces an isomorphism on Sobolev spaces. As in Lemma \ref{lemma:isomorphisms} (more precisely, apply Lemma \ref{lemma:isomorphisms} to the trivial vector bundle $\E = \M \times \mathbb{C}$ equipped with the trivial unitary connection $d$, and note that as explained in Footnote \eqref{footnote:1} the kernel of $X$ is stably non-empty, so the lemma applies in our setting), the maps:
\begin{align*}
	\pi_{\ker D^*_\tau} \Delta_0 \pi_{\ker D^*_0}&: \ker D_0^* \cap H^s(M,\otimes^m_S T^*M) \to \ker D_\tau^* \cap H^{s-k}(M,\otimes^m_S T^*M),\\
	\pi_{\ker D^*_0} \Delta_0 \pi_{\ker D^*_\tau}&: \ker D_\tau^* \cap H^s(M,\otimes^m_S T^*M) \to \ker D_0^*  \cap H^{s-k}(M,\otimes^m_S T^*M),
\end{align*}
are isomorphisms for $\tau$ small enough depending on $s \in \mathbb{R}$. In particular, as before, $\Pi_m^\tau$ is solenoidal injective (i.e. injective on symmetric tensors in $\ker D^*_\tau$) if and only if $P_\tau$ is solenoidal injective. In what follows we assume that $(u_i)_{i = 1}^d$ is an $L^2$-orthonormal basis of eigenstates of $P_0$ at $0$. As in \eqref{equation:lambda-tau}, we let $\lambda_\tau$ be the sum of the eigenvalues of $P_\tau$ inside a small contour near $0$.

\subsection{Variations of the ground state} We now compute the variations with respect to $\tau$.

\subsubsection{First order variations}

\label{sssection:first-metric}

As in the connection case, the first order variation of $\lambda_{\tau}$ at $\tau=0$ is $\dot{\lambda}_{\tau = 0} = 0$ and the second order variation is given by Lemma \ref{lemma:abstract}. We thus need to compute each term involved in the second derivative $\ddot{\lambda}_{\tau = 0}$, namely $\dot{P}_0$ and $\ddot{P}_0$. More precisely, the goal of this section is to compute $\sum_{i=1}^d \langle P_0^{-1} \dot{P}_0 u_i, \dot{P}_0 u_i \rangle_{L^2}$.

We assume that $g_\tau = g + \tau f$, where $f \in C^\infty(M,\otimes^2_S T^*M)$ is a perturbation such that $\dot{g} = f$. We write
\begin{align}\label{eq:resolvent-expansion}
\begin{split}
-(-X_{\tau}-z)^{-1} &= (z+X_{\tau})^{-1} = \dfrac{\Pi^\tau_+}{z} + R^\tau_{+,0} + z R^{\tau}_{+,1} + z^2 R^\tau_{+,2} + \mc{O}(z^3),\\
-(X_\tau-z)^{-1} &= (z-X_{\tau})^{-1} = \dfrac{\Pi^\tau_-}{z} + R^\tau_{-,0} + z R^{\tau}_{-,1} + z^2 R^\tau_{-,2} + \mc{O}(z^3),
\end{split}
\end{align}
for the meromorphic expansion of the resolvents at zero. First of all, we compute the derivatives of $\chi_\tau$ at $\tau = 0$: 
\begin{align}
\label{equation:chi-dot}
\begin{split}
\dot{\chi}_{\tau=0}(x,v) &= -\frac{m}{2} \dot{g}(v,v) = -\frac{m}{2} \pi_2^* f (x,v),\\
\ddot{\chi}_{\tau = 0}(x,v) &= \frac{m}{2}\left(\frac{m}{2} + 1\right) [\pi_2^* f]^2 (x,v).
\end{split}
\end{align}
In the following we recall that $\Pi_+^\tau = \Pi_-^\tau$ (see \eqref{eq:Pi_pm-tau}). We now turn to the first order variation of the resolvent:

\begin{lemma}
\label{lemma:first-variation-resolvent}
We have $\dot{\Pi}_\pm = - \Pi_\pm \dot{X} R_{+,0}$ and:
\begin{align*}
\dot{R}_{+,0}&= -\left(R_{+,0} \dot{X} R_{+,0} + \Pi_+ \dot{X} R_{+,1}  \right), \quad \dot{R}_{-,0} = + \left( R_{-,0} \dot{X} R_{-,0} + \Pi_- \dot{X} R_{-,1} \right).
\end{align*}
\end{lemma}

\begin{proof}
Let us deal with the second equality (the third one is similar). By \eqref{eq:resolvent-expansion}, we have (recall that the contour $\gamma$ around zero was defined in \S \ref{ssection:resolvent-perturbations-theory}):
\begin{equation}\label{eq:resolvent-+-0-tau}
	R^{\tau}_{+,0} = \dfrac{1}{2 \pi i} \oint_{\gamma} (z+X_\tau)^{-1} \dfrac{\dd z}{z},
\end{equation}
and differentiating with respect to $\tau$, we get:
\[
\begin{split}
\dot{R}_{+,0} & = - \dfrac{1}{2 \pi i} \oint_{\gamma} (z+X)^{-1}  \dot{X} (z+X)^{-1}  \dfrac{\dd z}{z} \\
& =- \left(R_{+,0} \dot{X} R_{+,0} + \Pi_+ \dot{X} R_{+,1} + R_{+,1} \dot{X} \Pi_+ \right).
\end{split}
\]
To conclude, it suffices to observe that $\dot{X} \Pi_+ = 0$ as $\dot{X}$ is a vector field and the range of $\Pi_+$ is always the constant functions $\C \cdot \mathbf{1}$.

As far as the derivative of the spectral projection is concerned, one starts with the equality
\begin{equation}\label{eq:Pi-+-0-tau}
\Pi^\tau_+ = \dfrac{1}{2 \pi i} \oint_{\gamma} (z+X_\tau)^{-1} \dd z,
\end{equation}
and then differentiates with respect to $\tau$ similarly as above; we omit the details.
\end{proof}

Therefore, recalling \eqref{equation:pi}, by \eqref{equation:chi-dot} and Lemma \ref{lemma:first-variation-resolvent} we obtain at $\tau = 0$:
\begin{equation}\label{eq:Pi_m^tau-first-derivative}
\begin{split}
\dot{\Pi}_m & = {\pi_m}_*\left( -\frac{m}{2} \left(\pi_2^* f. \mc{I} + \mc{I} \pi_2^* f\right) + \dot{\mc{I}} \right) \pi_m^*\\
&= {\pi_m}_*\Big( -\frac{m}{2} \left(\pi_2^* f. \mc{I} + \mc{I} \pi_2^* f \right)  -\left(R_{+,0} \dot{X} R_{+,0} + \Pi_+ \dot{X} R_{+,1}  \right)\\  
&+ \left( R_{-,0} \dot{X} R_{-,0} + \Pi_- \dot{X} R_{-,1} \right) 
 \left.  - \Pi_+ \dot{X} R_{+,0} \right) \pi_m^*.
\end{split}
\end{equation}

Next, we observe that for $m$ \emph{odd}, ${\pi_m}_* \mathbf{1} = 0$ and also ${\pi_m}_* (\pi_2^*f) = 0$ and as a consequence, the terms in \eqref{eq:Pi_m^tau-first-derivative} involving $\Pi_\pm$ disappear; when \emph{$m$ is even}, this simplification no longer occurs. 
Recalling the definition \eqref{eq:P_tau-metric-def} of $P_\tau$, we obtain from \eqref{eq:Pi_m^tau-first-derivative} the general expression (valid for $m$ odd or even):
\[
\begin{split}
&\dot{P}_0   = {\pi_{\ker D^*_0}} \Delta_0 {\pi_m}_*\Big( -\frac{m}{2}\big(\pi_2^* f .(R_{+,0} + R_{-,0})+ (R_{+,0} + R_{-,0})\pi_2^* f\big) \\
& \hspace{5cm} \left. -R_{+,0} \dot{X} R_{+,0}    +  R_{-,0} \dot{X} R_{-,0} \right) \pi_m^* \Delta_0 {\pi_{\ker D^*_0}}\\
& + \eps(m) {\pi_{\ker D^*_0}} \Delta_0 {\pi_m}_* \\
& \left(-\frac{m}{2}\pi_2^*f.\Pi_+ \underbrace{- \frac{m}{2}\Pi_+ \pi_2^*f - \Pi_+ \dot{X} R_{+,1} + \Pi_{-} \dot{X} R_{-,1} - \Pi_{+} \dot{X} R_{+,0}  }_{S(\pi_2^*f):=}\right) \\
& \hspace{10cm}\pi_m^* \Delta_0 {\pi_{\ker D^*_0}},
\end{split}
\]
where $\eps(m)=0$ for $m$ odd and $\eps(m)=1$ for $m$ even. This last term is isolated on purpose because as we shall see, it only contributes to a smoothing remainder in the following argument and will therefore disappear in the principal symbol computations.

We let $u_i \in C^\infty(M,\otimes^m_S T^*M) \cap \ker D^*_0$ be one of the elements in the kernel of $P_{0}$. Note that the operator $\Pi_m$ being \emph{real}, we can always assume that the $u_i$'s are real-valued. This implies using Lemma \ref{lemma:relation} that (as in \S \ref{section:genericity-connection})
\begin{equation}\label{eq:Xv_i}
	\pi_m^* \Delta_0 \pi_{\ker D^*_0} u_i = X \VV_i, \quad \mathrm{for\,\, some}\quad\VV_i \in C^\infty(SM)\quad \mathrm{with}\quad \Pi_+\VV_i = 0,
\end{equation}
which can also be chosen real-valued (since $X$ is real). Hence, using \eqref{eq:resolvent-identities} and the fact that $(R_{+,0} + R_{-,0}) X = 0$, $X \Pi_+ = \Pi_+X = \dot{X}\Pi_+ = 0$, and recalling that $\dot{X} = \Lambda(\pi_2^*f)$, we obtain:
\begin{equation}\label{eq:P_0-dot-u_i}
\begin{split}
\dot{P}_0 u_i &= -\frac{m}{2} {\pi_{\ker D^*_0}} \Delta_0 {\pi_m}_*  \big(\pi_2^*f. (R_{+,0} + R_{-,0})+ (R_{+,0} + R_{-,0})\pi_2^* f\big)  X \VV_i \\
& + {\pi_{\ker D^*_0}} \Delta_0 {\pi_m}_*\big(-R_{+,0} \dot{X} \underbrace{R_{+,0}X}_{= \mathbbm{1} - \Pi_+} \VV_i    +  R_{-,0} \dot{X} \underbrace{R_{-,0} X}_{=-\mathbbm{1} + \Pi_-} \VV_i\big) \\
&  + \eps(m) {\pi_{\ker D^*_0}} \Delta_0 {\pi_m}_* \left(-\frac{m}{2} \pi_2^*f. \Pi_+  + S(\pi_2^*f) \right) X\VV_i \\
&= - {\pi_{\ker D^*_0}} \Delta_0 {\pi_m}_*(R_{+,0}+R_{-,0}) \left[ \frac{m}{2} \pi_2^* f. X \VV_i +  \Lambda(\pi_2^* f) \VV_i \right] \\
&   \hspace{5cm} + \eps(m) {\pi_{\ker D^*_0}} \Delta_0 {\pi_m}_* S(\pi_2^*f) X \VV_i.
\end{split}
\end{equation}
Let us introduce shorthand notation for the operators arising in \eqref{eq:P_0-dot-u_i}, for any $\VV \in C^\infty(SM)$ and any symmetric $2$-tensor $h \in C^\infty(M, \otimes_S^2T^*M)$:
\begin{equation}
\label{equation:bvqv}
	B_{\VV} h := {\pi_m}_* S(\pi_2^*h) X \VV_i, \quad Q_{\VV} h := {\pi_m}_* (R_{+,0}+R_{-,0}) \left[ \frac{m}{2} \pi_2^* h. X \VV+  \Lambda(\pi_2^* h) \VV \right].
\end{equation}
By using the facts that $P_0^{-1} = (\pi_{\ker D_0^*} \Delta_0 \pi_{\ker D_0^*})^{-1} \Pi_m^{-1} (\pi_{\ker D_0^*} \Delta_0 \pi_{\ker D_0^*})^{-1}$ on $\ker D_0^*$ and the equalities $D_0^* {\pi_m}_* (R_{+, 0} + R_{-, 0}) = D_0^* {\pi_m}_* S(\pi_2^*f) = 0$, we obtain from \eqref{eq:P_0-dot-u_i} (similarly to \eqref{equation:p-moins-un}):
\begin{equation}
\label{equation:p-moins-un-metrique}
\begin{split}
 \langle P_0^{-1} \dot{P}_0 u_i, \dot{P}_0 u_i \rangle_{L^2} & =  \left\langle \Pi_m^{-1} Q_{\VV_i}f, Q_{\VV_i}f \right\rangle_{L^2} + \varepsilon(m)\left\langle \Pi_m^{-1} B_{\VV_i}f, B_{\VV_i}f \right\rangle_{L^2} \\
 & - 2\varepsilon(m) \Re \left(\left\langle Q_{\VV_i}f, B_{\VV_i}f \right\rangle_{L^2}\right).
\end{split}
\end{equation}

We now prove that the operator $B_{\VV}$ introduced in \eqref{equation:bvqv} is in fact smoothing:

\begin{lemma}
\label{lemma:smoothing-1}
$B_{\VV_i} \in \Psi^{-\infty}(M, \otimes^2_S T^*M \rightarrow \otimes^m_S T^*M)$ is a smoothing operator.
\end{lemma}

\begin{proof}
Similarly to the proof of Lemma \ref{lemma:simplification}, it suffices to show that $B_{\VV_i}$ maps
\[
\mc{D}'(M,\otimes^2_S T^*M) \rightarrow C^\infty(M,\otimes^m_S T^*M)
\]
boundedly. For that, let $h \in \mc{D}'(M,\otimes^2_S T^*M)$. By \cite[Lemma 2.1]{Lefeuvre-19-1} we have $\WF(\pi_2^*h) \subset \V^*$  and $\WF(\pi_{2,\mathrm{Sas}}^*h) \subset \V^*$ (as in the proof of Lemma \ref{lemma:simplification}). By resolvent identities \eqref{eq:resolvent-identities} and using $\dot{X} \Pi_+ = 0$, we have:
\begin{equation}
\label{equation:s-bidule}
\begin{split}
S(\pi_2^*h) X \VV_i = - \frac{m}{2} \Pi_+(\pi_2^*h. X\VV_i) + \Pi_+ \Lambda(\pi_2^*h) (R_{+, 0} + R_{-, 0})\VV_i  - \Pi_{+} \Lambda(\pi_2^*h) \VV_i.
\end{split}
\end{equation}
By \eqref{eq:Pi_pm}, the first term in \eqref{equation:s-bidule} i.e. $- \frac{m}{2} \Pi_+
(\pi_2^*h. X\VV_i) = -\frac{m}{2} \langle \pi_2^*h. X\VV_i, \mu \rangle \mathbf{1}$ is clearly smooth. Similarly, the last term $- \Pi_{+} \Lambda(\pi_2^*h) \VV_i$ is also obviously smooth.

We now deal with the remaining term. Define $\VV_i^u := (R_{+, 0} + R_{-, 0}) \VV_i$. Since $\VV_i$ is smooth, this has wavefront set in $E^*_u \cup E_s^*$ by the characterization of the wavefront set of the resolvent, see \eqref{equation:dz}. Hence, the one-form $\dd \VV_i^u$ has also wavefront set in $E^*_u \cup E_s^*$ and thus by the multiplication rule of distributions (see \cite[Theorem 8.2.10]{Hormander-90} for instance), the inner product $\Lambda(\pi_2^*h) \VV_i^u = \imath_{\Lambda(\pi_2^*h)} \dd \VV_i^u$ is allowed (since $\V^* \cap \left(E^*_u \cup E_s^*\right) = \left\{ 0 \right\}$ by the absence of conjugate points, see \eqref{equation:conjugate}). As a consequence 
\[
\Pi_+ \Lambda(\pi_2^*h) (R_{+, 0} + R_{-, 0}) \VV_i  = \langle \imath_{\Lambda(\pi_2^*h)} \dd \VV_i^u, \mu \rangle \mathbf{1}
\]
is smooth, which proves the lemma.
\end{proof}

We have an analogous statement to Lemma \ref{lemma:symbol-q}:

\begin{lemma}
\label{lemma:symbol-q-bis}
For all $\VV \in C^\infty(SM)$, 
\[
Q_{\VV} \in \Psi^{0}(M, \otimes^2_S T^*M \rightarrow \otimes^m_S T^*M),
\]
is a pseudodifferential operator of order $0$ with principal symbol for $(x,\xi) \in T^*M \setminus \left\{ 0 \right\}$, $\sigma_{Q_{\VV}}(x,\xi) \in \mathrm{Hom}(\otimes^2_S T_x^*M, \otimes^m_S T_x^*M)$, given by:
\[
\sigma_{Q_{\VV}}(x,\xi)h = C^{-1}_{n-1+2m}  \dfrac{i \pi}{|\xi|} \pi_{\ker \imath_{\xi^\sharp}} {\pi_m}_* E_\xi^m \left(\langle \xi_{\mathbb{V}}, \nabla_{\V} \VV \rangle. \pi_2^* h  \right), \quad h \in \otimes^2_S T_x^*M,
\]
where $\xi_{\mathbb{V}}(x, v) \in \mathbb{H}^*(x, v)$ is defined by $\xi_{\mathbb{V}}(x, v) (\bullet) := \xi(\mc{K}_{x, v}(\bullet))$.
\end{lemma}

Observe that the main difference with Lemma \ref{lemma:symbol-q} is that the operator is now of order $0$ instead of $-1$: this is due to the fact that the operator $\Lambda$ is of degree $1$ (it costs derivatives of order $1$ in $f$). Also note that the principal symbol of $Q_{\VV}$ now satisfies $\sigma_{Q_{\VV}}(x,-\xi) = -\sigma_{Q_{\VV}}(x,\xi)$; this will not be a problem as $Q_{\VV}$ will always appear twice in the brackets (hence the two minus signs will eventually give a plus).

\begin{proof}
We may re-write the operator $Q_{\VV}$ in a sandwich form as follows:
\begin{equation}\label{eq:qv-sandwich}
\begin{split}
	Q_{\VV} & =\frac{m}{2}{\pi_m}_*(R_{+, 0} + R_{-, 0})M_{ X \VV} \pi_2^* + {\pi_m}_*(R_{+, 0} \\
	& + R_{-, 0})N_{\VV}(\Lambda^{\mathrm{conf}} \pi_2^* + \Lambda^{\mathrm{aniso}}\pi_{2,\mathrm{Sas}}^*),
	\end{split}
\end{equation}
where we use the notation of \eqref{eq:sandwich-auxiliary-operator}. By the sandwich Proposition \ref{proposition:sandwich} (and Remark \ref{remark:lift} below), it follows that $Q_{\VV}$ is pseudodifferential of order $0$ and the first term of \eqref{eq:qv-sandwich} is of order $-1$. By Lemma \ref{lemma:vector}, the principal symbol of $N_{\VV}\Lambda^{\mathrm{conf}}$ is given by, on co-vectors $\xi_{\HH} = \xi \circ d_{x, v}\pi$:
\begin{align}\label{eq:N_vLambda-symbol}
\begin{split}
	& \sigma(N_{\VV}\Lambda^{\mathrm{conf}})(x, v, \xi_{\HH}(x, v)) \\
	&= \frac{i}{2}d_{x, v}\VV\left(J(x, v)\left((\xi_{\HH}(x, v))^\sharp - X(x, v). \langle{\xi, v}\rangle\right)\right)\\ 
	&= \frac{i}{2} \left\langle{\nabla_{\V}\VV(x, v), (0, \xi^\sharp - v. \langle{\xi, v}\rangle)}\right\rangle = \frac{i}{2} \left\langle{\xi_{\V}(x, v), \nabla_{\V}\VV(x, v)}\right\rangle.
\end{split}
\end{align}
Similarly, by Lemma \ref{lemma:LambdaAnisoSymbol}, we have:
\begin{equation}
\label{equation:N_v-autre-lambda}
\begin{split}
\sigma(N_{\VV}\Lambda^{\mathrm{aniso}})(x,v,\xi_{\HH}(x,v))h & = - i \langle \xi_{\HH}, X(x,v) \rangle. d_{x,v} \VV \left(J \big(h(X,\pi_{\HH}(\bullet))\big)^\sharp\right) \\
& = -i \langle \xi, v \rangle.  d_{x,v} \VV \left(J \big(h(X,\pi_{\HH}(\bullet))\big)^\sharp\right).
\end{split}
\end{equation}
It then suffices to apply Proposition \ref{proposition:sandwich} and Remark \ref{remark:lift} to conclude the proof. Note that the principal symbol of $N_{\VV} \Lambda^{\mathrm{aniso}}$ does not appear as the expression in Remark \ref{remark:lift} involves integration over the sphere $\left\{\langle \xi, v \rangle = 0 \right\} = \mathbb{S}^{n-2}_\xi$.
\end{proof}
To summarise the content of this section, we obtain from \eqref{equation:p-moins-un-metrique} using Lemmas \ref{lemma:smoothing-1} and \ref{lemma:symbol-q-bis}:
\begin{equation}
\begin{split}
\label{equation:rhs-bis}
& \sum_{i=1}^d \langle P_0^{-1} \dot{P}_0 u_i, \dot{P} u_i \rangle_{L^2}\\
& = \sum_{i=1}^d  \langle Q_{\VV_i}^* \Pi_m^{-1} Q_{\VV_i} f, f \rangle_{L^2(M,\otimes^2_S T^*M)} + \langle \mc{O}_{\Psi^{-\infty}}(1)f, f \rangle_{L^2(M,\otimes^2_S T^*M)}.
\end{split}
\end{equation}

\subsubsection{Second order variations}

\label{sssection:second-metric}

We now turn to the second variation of the resolvents which will allow us to compute $\sum_{i=1}^d \langle \ddot{P}_0 u_i, u_i \rangle_{L^2}$. Similarly to Lemma \ref{lemma:first-variation-resolvent}, we have:

\begin{lemma}\label{lemma:second-variation-resolvent}
	We have $\ddot{\Pi}_{\pm} = 2\Pi_\pm \dot{X} R_{\pm, 0} \dot{X} R_{\pm, 0} \mp \Pi_\pm \ddot{X} R_{\pm, 0}$ and:
	\begin{multline}\label{eq:ddot-resolvent}
		\ddot{R}_{\pm, 0} = 2 \left(\Pi_\pm \dot{X} R_{\pm, 0} \dot{X} R_{\pm, 1} + \Pi_\pm \dot{X} R_{\pm, 1} \dot{X} R_{\pm, 0} + R_{\pm, 0} \dot{X} R_{\pm, 0} \dot{X} R_{\pm, 0}\right)\\
		\mp \left(\Pi_\pm\ddot{X} R_{\pm, 1} + R_{\pm, 0} \ddot{X} R_{\pm, 0}\right).
	\end{multline}
\end{lemma}
\begin{proof}
	We prove the claims for the $(X_\tau + z)^{-1}$ resolvent, the other claim follows analogously. Differentiating the expression \eqref{eq:resolvent-+-0-tau} twice at $\tau = 0$, we obtain the formula:
	\[
	\begin{split}
	\ddot{R}_{+, 0} &= \frac{1}{\pi i} \oint_\gamma (X + z)^{-1} \dot{X} (X + z)^{-1} \dot{X} (X + z)^{-1} \frac{dz}{z} \\
	& - \frac{1}{2\pi i} \oint_\gamma (X + z)^{-1} \ddot{X} (X + z)^{-1} \frac{dz}{z}.
	\end{split}
	\]
	The equation \eqref{eq:ddot-resolvent} now follows by the residue theorem and the expansion \eqref{eq:resolvent-expansion}, using the relations $\dot{X} \Pi_+ = \ddot{X} \Pi_+ = 0$ to cancel the extra terms.
	
	For the derivative of the spectral projector, by differentiating the formula \eqref{eq:Pi-+-0-tau} twice:
	\[
	\begin{split}
	\ddot{\Pi}_+ & = \frac{1}{\pi i} \oint_\gamma (X + z)^{-1} \dot{X} (X + z)^{-1} \dot{X} (X + z)^{-1} dz\\
	& - \frac{1}{2\pi i} \oint_\gamma (X + z)^{-1} \ddot{X} (X + z)^{-1} dz,
	\end{split}
	\]
	and the final result again follows from the expansion \eqref{eq:resolvent-expansion}.
\end{proof}
Next, differentiating \eqref{equation:pi} and inserting \eqref{equation:chi-dot}, we obtain at $\tau = 0$:
\begin{equation}\label{eq:Pi-ddot}
\begin{split}
\ddot{\Pi}_m & = {\pi_m}_*\left( \frac{m}{2}\left(\frac{m}{2}+1\right) \left[ (\pi_2^* f)^2. \mc{I} +  \mc{I} (\pi_2^* f)^2\right] \right. \\
& \left.- m \left[ \pi_2^* f. \dot{\mc{I}}  + \dot{\mc{I}} \pi_2^* f \right]  + \ddot{\mc{I}}+\frac{m^2}{2} \pi_2^*f. \mc{I} \pi_2^* f \right)\pi_m^*.
\end{split}
\end{equation}
We can already make some simplifications in the term $\langle \ddot{P} u_i, u_i \rangle_{L^2}$, where $u_i$ is one of the eigenstates at $0$ of $P_0$. Recalling \eqref{eq:Xv_i} and using $\mc{I} X = X \mc{I}= 0$, from \eqref{eq:Pi-ddot} we get at $\tau = 0$:
\begin{multline}\label{eq:P_0-ddot-step}
\langle \ddot{P}_0 u_i, u_i \rangle_{L^2} = \left\langle \pi_{\ker D_0^*} \Delta_0 \ddot{\Pi}_m \Delta_0 \pi_{\ker D_0^*} u_i, u_i \right\rangle_{L^2(M,\otimes^m_S T^*M)} \\
 = \left\langle \left(-m\left[ \pi_2^* f. \dot{\mc{I}}  + \dot{\mc{I}} \pi_2^* f \right]  + \ddot{\mc{I}}+\frac{m^2}{2} \pi_2^*f .\mc{I} \pi_2^* f\right) X \VV_i, X \VV_i \right\rangle_{L^2(SM)}.
\end{multline}
Similarly to \S\ref{sssection:first-metric}, some terms will disappear when $m$ is odd while in the case where $m$ is even, they will only contribute up to a smoothing remainder. First of all, let us assume that \emph{$m$ is odd}. We make the simple observation that $\langle{h_{\mathrm{even}}, h_{\mathrm{odd}}}\rangle_{L^2} = 0$ for any $h_{\mathrm{even}}, h_{\mathrm{odd}} \in C^\infty(SM)$ with even and odd Fourier content, respectively. Therefore, the terms involving $\Pi_\pm$ in \eqref{eq:P_0-ddot-step} obtained after using Lemmas \ref{lemma:first-variation-resolvent} and \ref{lemma:second-variation-resolvent} to expand $\dot{\mc{I}}, \ddot{\mc{I}}$ vanish, and we get:
\[
\begin{split}
& \langle \ddot{P}_0 u_i, u_i \rangle_{L^2} \\
& = -m \left\langle \left[ \pi_2^*f. \left(R_{-,0}\dot{X} R_{-,0} - R_{+,0}\dot{X}R_{+,0}\right) \right. \right. \\
& \hspace{4cm}\left. \left. + \left(R_{-,0}\dot{X} R_{-,0} - R_{+,0}\dot{X}R_{+,0}\right) \pi_2^*f \right] X \VV_i, X \VV_i \right\rangle_{L^2} \\
& + \left\langle \left[  2\left(R_{+,0} \dot{X} R_{+,0} \dot{X} R_{+,0} + R_{-,0} \dot{X} R_{-,0} \dot{X} R_{-,0}\right) \right. \right. \\
& \hspace{4cm} \left. \left. -R_{+,0} \ddot{X} R_{+,0} + R_{-,0} \ddot{X} R_{-,0} \right] X \VV_i, X \VV_i \right\rangle_{L^2} \\
& + \frac{m^2}{2} \left\langle (R_{+,0} + R_{-,0}) (\pi_2^*f. X \VV_i), \pi_2^* f. X \VV_i \right\rangle_{L^2}.
\end{split}
\]
Using the resolvent relations \eqref{eq:resolvent-identities}, the terms involving $\ddot{X}$ cancel each other:
\[
\left\langle \left( -R_{+,0} \ddot{X} R_{+,0} + R_{-,0} \ddot{X} R_{-,0} \right) X \VV_i, X \VV_i \right\rangle_{L^2(SM)} = 0.
\]
Therefore, we get using again \eqref{eq:resolvent-identities}, the fact that involved quantities are \emph{real}, recalling that $\int_{SM} \VV_i\, d\vol_{g_{\mathrm{Sas}}} = 0$, and using \eqref{eq:vector_field_formal_adjoint}:
\begin{equation}\label{eq:I-IV}
\begin{split}
& \langle \ddot{P}_0 u_i, u_i \rangle_{L^2} \\
&  = -m \left( -\left\langle \pi_2^*f. (R_{+,0} + R_{-,0}) \dot{X} \VV_i,  X\VV_i \right\rangle_{L^2(SM)} \right. \\
& \hspace{5cm} \left. + \left\langle \dot{X} (R_{+,0} + R_{-,0}) (\pi_2^* f. X \VV_i), \VV_i \right\rangle_{L^2(SM)}  \right) \\
& - 2 \left\langle \dot{X} (R_{+,0} + R_{-,0}) \dot{X} \VV_i, \VV_i \right\rangle_{L^2(SM)} \\ 
& + \frac{m^2}{2} \big\langle \left(R_{+,0} + R_{-,0}\right) (\pi_2^*f. X \VV_i), \pi_2^* f. X \VV_i \big\rangle_{L^2(SM)} \\
& = \underbrace{2m \left\langle {\pi_2}_*\left(X\VV_i.  (R_{+,0} + R_{-,0}) \dot{X} \VV_i\right), f \right\rangle_{L^2}}_{=:\mathrm{(I)}}\\
& \hspace{2cm} \underbrace{- m \left\langle {\pi_2}_*\left(X\VV_i. (R_{+,0} + R_{-,0}) \left(\mathrm{div}(\dot{X}).\VV_i\right)\right), f \right\rangle_{L^2}}_{=:\mathrm{(I)}} \\
& + \underbrace{2 \left\langle (R_{+,0} + R_{-,0}) \dot{X}\VV_i, \dot{X}\VV_i \right\rangle_{L^2}}_{=:\mathrm{(II)}} - \underbrace{2 \left\langle (R_{+,0} + R_{-,0}) \dot{X}\VV_i, \mathrm{div}(\dot{X}). \VV_i \right\rangle_{L^2}}_{=:\mathrm{(III)}} \\
& + \underbrace{\frac{m^2}{2}  \big\langle {\pi_2}_* \big(X\VV_i. (R_{+,0} + R_{-,0}) (\pi_2^*f. X \VV_i)\big), f  \big\rangle_{L^2}}_{=:\mathrm{(IV)}}.
\end{split}
\end{equation}
(The expression for $\mathrm{(I)}$ runs on two lines.) In the general case (where $m$ is either even or odd), the previous equality still holds by adding on the right-hand side the extra term $\mathrm{(V)}$ which is equal to:
\begin{equation}
\label{equation:b}
\begin{split}
\mathrm{(V)} := ~& \eps(m) \left( -m\left\langle \pi_2^*f. \Pi_+\dot{X} \left(R_{+, 0} + R_{-, 0} - \mathbbm{1}\right) \VV_i,  X\VV_i \right\rangle_{L^2} \right. \\
&\left. + \frac{m^2}{2} \left\langle \pi_2^*f. \Pi_+ (\pi_2^*f. X\VV_i), X\VV_i \right\rangle_{L^2}\right).
\end{split}
\end{equation}
Here again we used \eqref{eq:P_0-ddot-step} along with Lemmas \ref{lemma:first-variation-resolvent} and \ref{lemma:second-variation-resolvent} to expand the terms containing $\Pi_\pm$; we also used \eqref{eq:resolvent-identities} and $X\Pi_+ = \dot{X}\Pi_+ = 0$ to simplify the expression.

As before, we will see that each term in the previous equality can be written in the form $\langle A f, f \rangle_{L^2(M,\otimes^2_S T^*M)}$, for some pseudodifferential operator $A$. In order to shorten the computations, it is important to understand the order of these operators: when taking Gaussian states, only the terms of highest order will remain. Also observe that \eqref{equation:rhs-bis} involves an operator of order $1$: we therefore expect to find operators of order at most $1$. First of all, we deal with the term $\mathrm{(V)}$ (appearing only when $m$ is even):

\begin{lemma}
\label{lemma:smoothing-2}
There exists $B'_{\VV_i} \in \Psi^{-\infty}(M,\otimes^2_S T^*M \rightarrow \otimes^2_S T^*M)$ such that:
\begin{equation}\label{eq:bv'}
\mathrm{(V)} = \varepsilon(m).\langle B'_{\VV_i} f, f \rangle_{L^2}.
\end{equation}
\end{lemma}

\begin{proof}
It follows from \eqref{equation:b} that \eqref{eq:bv'} holds with:
\begin{equation}
\begin{split} 
\label{eq:bv'-def}
	& B_{\VV_i}' = {\pi_2}_*\\
	& \left(-m. M_{X\VV_i} \Pi_+ N_{(R_{+, 0} + R_{-, 0})\VV_i} \Lambda + m.M_{X\VV_i} \Pi_+ N_{\VV_i} \Lambda + \frac{m^2}{2} M_{X\VV_i} \Pi_+ M_{X\VV_i}\right)\pi_2^*
	\end{split}
\end{equation}
The last two terms in \eqref{eq:bv'-def} are obviously smoothing, while exactly the same wavefront set arguments as in Lemma \ref{lemma:smoothing-1} apply to show that the first term is smoothing. 
\end{proof}

We then prove (recall that $\xi_{\mathbb{V}}(x, v) (\bullet) = \xi(\mc{K}_{x, v}(\bullet))$):

\begin{lemma}
\label{lemma:work}
There exist pseudodifferential operators $A^{\mathrm{(I)}}_{\VV_i}, A^{\mathrm{(II)}}_{\VV_i}, A^{\mathrm{(III)}}_{\VV_i}$ and $A^{\mathrm{(IV)}}_{\VV_i}$ in the pseudodifferential algebra $\Psi^*(M,\otimes^2_S T^*M \rightarrow \otimes^2_S T^*M)$ of respective order $0$, $1$, $0$ and $-1$ such that (recall $\mathrm{(*)}$ was defined in \eqref{eq:I-IV}):
\begin{equation}\label{eq:A^I-IV}
\mathrm{(*)} = \langle A^{\mathrm{(*)}}_{\VV_i} f, f \rangle_{L^2},
\end{equation}
for $\mathrm{*} \in \left\{\mathrm{I}, \mathrm{II}, \mathrm{III}, \mathrm{IV}\right\}$. Moreover, we have for $(x, \xi) \in T^*M \setminus \{0\}$ and $h \in \otimes^2_S T^*_{x}M$:
\begin{equation}
\begin{split}
\label{eq:A^II-symbol}
&\langle \sigma_{A^{\mathrm{(II)}}}(x,\xi) h, h \rangle_{\otimes^2_S T^*_{x}M}\\
&  = \dfrac{\pi}{|\xi|} C^{-1}_{n-1+2m} \left\| E^m_{\xi}\left( \pi_2^*h(\bullet). \langle \xi_{\mathbb{V}}(x, \bullet), \nabla_{\V} \VV (x,\bullet) \rangle|_{\Ss^{n-2}_\xi} \right) \right\|_{L^2(S_{x}M)}^2.
\end{split}
\end{equation}
\end{lemma}

\begin{proof}
\emph{Term {\rm (IV)}.} Using the notation of \eqref{eq:sandwich-auxiliary-operator}, we observe that \eqref{eq:A^I-IV} holds for $\mathrm{*} = \mathrm{IV}$ with:
\[
A^{\mathrm{(IV)}}_{\VV_i} := \frac{m^2}{2} {\pi_2}_* M_{X\VV_i} (R_{+,0} + R_{-,0})  M_{X\VV_i} \pi_2^* .
\]
By the sandwich Proposition \ref{proposition:sandwich}, this is a $\Psi$DO of order $-1$. \\

\emph{Term {\rm(I)}.} Next, using the formula for the divergence in Lemma \ref{lemma:divergence} together with $M_{\VV_i} X u = X(M_{\VV_i}u)- M_{X \VV_i} u$ and $(R_{+,0} + R_{-,0})X = 0 = X(R_{+,0} + R_{-,0})$, we see from \eqref{eq:I-IV} that \eqref{eq:A^I-IV} holds for $\mathrm{*} = \mathrm{I}$ with: 
\begin{equation}
\label{equation:a-1}
\begin{split}
A^{\mathrm{(I)}}_{\VV_i}  & := 2m. {\pi_2}_* M_{X \VV_i} (R_{+,0} + R_{-,0}) N_{\VV_i} \Lambda \pi_2^* \\
& \hspace{2cm} + m. {\pi_2}_* M_{X \VV_i} (R_{+,0} + R_{-,0}) M_{X \VV_i} \left(\pi_0^*\Tr_{g_0}(\bullet) - \frac{n}{2} \pi_2^*(\bullet)\right).
\end{split}
\end{equation}
By Proposition \ref{proposition:sandwich}, the first term in \eqref{equation:a-1} is of order $0$ as $\Lambda$ is of order $1$, and the second term is of order $-1$. \\

\emph{Term {\rm(III)}.} By Lemma \ref{lemma:divergence} and using $M_{\VV_i} X u = X(M_{\VV_i}u)- M_{X \VV_i} u$, we get from \eqref{eq:I-IV}:
\[
\begin{split}
A_{\VV_i}^{\mathrm{(III)}} =(2 j_{g_0} {\pi_0}_* - n. {\pi_2}_*) M_{X \VV_i}  (R_{+,0} + R_{-,0}) N_{\VV_i} \Lambda \pi_2^*,
\end{split}
\]
where $j_{g_0}$ denotes the multiplication by $g_0$, that is $j_{g_0} u := u.g_0$ for $u \in C^\infty(M)$. 
By Proposition \ref{proposition:sandwich}, this is a pseudodifferential operator of order $0$. \\

\emph{Term {\rm(II)}.} Finally, for the term $\mathrm{(II)}$, we write from \eqref{eq:I-IV}:
\[
\begin{split}
A^{\mathrm{(II)}}_{\VV_i} := & 2 \left({\pi_2}_* \left(\Lambda^{\mathrm{conf}}\right)^* + {\pi_{2,\mathrm{Sas}}}_* \left(\Lambda^{\mathrm{aniso}}\right)^*\right) \\
& \hspace{3cm} N_{\VV_i}^* (R_{+,0} + R_{-,0}) N_{\VV_i} (\Lambda^{\mathrm{conf}} \pi_2^* + \Lambda^{\mathrm{aniso}} \pi_{2,\mathrm{Sas}}^*),
\end{split}
\]
where $\left(\Lambda^{\mathrm{conf,aniso}}\right)^*$ denotes the formal adjoint of the first order differential operators $\Lambda^{\mathrm{conf,aniso}}$. Once more, applying Proposition \ref{proposition:sandwich} and Remark \ref{remark:lift} shows that $A^{\mathrm{(II)}}_{\VV_i}$ is a $\Psi$DO of order $1$ with the principal symbol given by \eqref{eq:A^II-symbol}, where we also use \eqref{eq:N_vLambda-symbol} and \eqref{equation:N_v-autre-lambda} to compute the principal symbols of $N_{\VV_i} \Lambda^{\mathrm{conf}}$ and $N_{\VV_i} \Lambda^{\mathrm{aniso}}$, and \eqref{eq:E^k-squared} to extend the formula to $S_xM$. Note that $N_{\VV_i} \Lambda^{\mathrm{aniso}}$ does not contribute to the principal symbol as its symbol (see \eqref{equation:N_v-autre-lambda}) is proportional to $\langle \xi, v \rangle$ and the formula of Remark \ref{remark:lift} involves integration over $\left\{\langle \xi,v \rangle = 0\right\} = \mathbb{S}^{n - 2}_{\xi}$.
\end{proof}
As a consequence, \eqref{eq:I-IV} and the previous two lemmas show that:
\begin{equation}
\label{equation:ri}
\sum_{i=1}^d \langle \ddot{P}_0 u_i, u_i \rangle_{L^2}  = \sum_{i=1}^d \langle R_{\VV_i} f, f \rangle_{L^2} + \langle \mc{O}_{\Psi^{-\infty}}(1)f,f\rangle_{L^2},
\end{equation}
where $R_{\VV_i} = A^{\mathrm{(I)}}_{\VV_i} + A^{\mathrm{(II)}}_{\VV_i} +A^{\mathrm{(III)}}_{\VV_i}+A^{\mathrm{(IV)}}_{\VV_i}$ is a pseudodifferential operator of order $1$, whose principal symbol is given by that of $A^{\mathrm{(II)}}_{\VV_i}$ and determined by \eqref{eq:A^II-symbol}.

\subsection{Assuming the second variation is zero}

We now assume that for any variation of the metric $g_\tau = g + \tau f$ with $f \in C^\infty(M,\otimes^2_ST^*M)$, the second derivative $\ddot{\lambda} = 0$ vanishes. Using Lemma \ref{lemma:abstract}, this implies that:
\begin{equation}
\label{equation:summary}
\sum_{i=1}^d \langle \ddot{P}_0 u_i, u_i \rangle_{L^2(M,\otimes^m_S T^*M)} = 2 \sum_{i=1}^d \langle P_0^{-1} \dot{P}_0 u_i, \dot{P}_0 u_i \rangle_{L^2(M,\otimes^m_S T^*M)}.
\end{equation}
In \S \ref{sssection:first-metric} and \S \ref{sssection:second-metric}, we saw that both the right-hand side and the left-hand side can be written as quadratic forms in $f$, with pseudodifferential operators of order $1$ acting on $f$. More precisely, combining \eqref{equation:rhs-bis} and \eqref{equation:ri}, we have that \eqref{equation:summary} implies:
\begin{equation}
\label{equation:summary-2}
\begin{split}
&\forall f \in C^\infty(M,\otimes^2_ST^*M), \\
&  \quad \sum_{i=1}^d \left\langle \left(R_{\VV_i} - 2Q_{\VV_i}^* \Pi_m^{-1} Q_{\VV_i}\right) f, f \right\rangle_{L^2(M,\otimes^2_S T^*M)} = \langle \mc{O}_{\Psi^{-\infty}}(1)f,f\rangle_{L^2},
\end{split}
\end{equation}
where the remainder has this form by Lemmas \ref{lemma:smoothing-1} and \ref{lemma:smoothing-2} and appears only when $m$ is even. 

We then consider a \emph{real} Gaussian state $f_h = \Re(\EE_h(x,\xi)) \widetilde{f}$  (see \eqref{eq:gaussian} for the notation), where $\widetilde{f} \in C^\infty(M,\otimes^2_S T^*M)$ is a smooth section such that $\widetilde{f}(x) =: f$ and $\xi \in T^*_xM \setminus \{0\}$. Note that similarly to \S\ref{section:genericity-connection} we can only allow \emph{real} perturbations of the metric, hence the need for the real part of the Gaussian state. Nevertheless, this will not be a problem insofar as the principal symbols of $R_{\VV_i}$ and $Q_{\VV_i}^* \Pi_m^{-1} Q_{\VV_i}$ are preserved by the antipodal map in the fibres. We thus obtain by applying \eqref{equation:summary-2} to the Gaussian state $h.f_h$ and taking the limit as $h \rightarrow 0$, using Lemma \ref{lemma:technical}: 
\[
\begin{split}
&C^{-1}_{n-1+2m} \dfrac{\pi}{|\xi|} \sum_{i=1}^d  \left\| E^m_{\xi}\left( \pi_2^*f(\bullet).  \left\langle\xi_{\mathbb{V}}(x, \bullet), \nabla_{\V} \VV_i (x,\bullet) \right\rangle|_{\Ss^{n-2}_{\xi}} \right) \right\|_{L^2(S_{x}M)}^2\\
&  =2 \sum_{i=1}^d \left\langle \sigma_{\Pi_m^{-1}\pi_{\ker D_0^*}}(x,\xi) \sigma_{Q_{\VV_i}}(x,\xi) f, \sigma_{Q_{\VV_i}}(x,\xi) f \right\rangle_{\otimes^m_S T^*_{x}M} \\
& = C^{-1}_{n-1+2m} \dfrac{\pi}{|\xi|} \\
 & \times \sum_{i=1}^d \left\langle \left(\pi_{\ker \imath_{\xi^\sharp}} {\pi_m}_* \pi_m^* \pi_{\ker \imath_{\xi^\sharp}}\right)^{-1}  \pi_{\ker \imath_{\xi^\sharp}} {\pi_m}_* E_\xi^m \left(\pi_2^* f. \langle \xi_{\mathbb{V}}, \nabla_{\V} \VV_i\rangle|_{\Ss^{n-2}_{\xi}}\right),\right.\\
&\hspace{5cm} \left. \pi_{\ker \imath_{\xi^\sharp}} {\pi_m}_* E_\xi^m \left(\pi_2^* f. \langle \xi_{\mathbb{V}}, \nabla_{\V} \VV_i\rangle|_{\Ss^{n-2}_{\xi}}\right) \right\rangle_{\otimes^m_S T^*_{x}M}\\
&= C^{-1}_{n-1+2m} \dfrac{\pi}{|\xi|} \\
& \times \sum_{i=1}^d \left\langle P_m  E_\xi^m \left(\pi_2^* f. \langle \xi_{\mathbb{V}}, \nabla_{\V} \VV_i\rangle|_{\Ss^{n-2}_{\xi}}\right),   E_\xi^m \left(\pi_2^* f. \langle \xi_{\mathbb{V}}, \nabla_{\V} \VV_i\rangle|_{\Ss^{n-2}_{\xi}}\right) \right\rangle_{\otimes^m_S T^*_{x}M},
\end{split}
\]
where we used Lemmas \ref{lemma:symbol-q-bis} and \ref{lemma:work} to compute the symbols of $R_{\VV_i}$ and $Q_{\VV_i}^* \Pi_m^{-1} Q_{\VV_i}$, and we introduced the following map:
\[
P_m :=  {\pi_m}^* \pi_{\ker \imath_{\xi^\sharp}}  \left(\pi_{\ker \imath_{\xi^\sharp}} {\pi_m}_* \pi_m^* \pi_{\ker \imath_{\xi^\sharp}}\right)^{-1}  \pi_{\ker \imath_{\xi^\sharp}} {\pi_m}_*.
\]
Note that $P_m$ is the $L^2$-orthogonal projection onto $\ran(\pi_m^*\pi_{\ker \imath_{\xi^\sharp}})$ in $L^2(S_{x}M)$ and so it holds that (cf. \eqref{equation:decomposition0}):
\[L^2(S_xM) = \ran(\pi_m^*\pi_{\ker \imath_{\xi^\sharp}}) \oplus^\perp \ker(\pi_{\ker \imath_{\xi^\sharp}} {\pi_m}_*).\]
Thus, we get cancelling the constant terms (similarly to \eqref{equation:partial}):
\begin{equation}
\begin{split}
\label{eq:symbolblabla}
&\sum_{i=1}^d \left\| E^m_{\xi}\left(\pi_2^* f. \langle \xi_{\mathbb{V}}, \nabla_{\V} \VV_i\rangle|_{\Ss^{n-2}_{\xi}}\right) \right\|_{L^2(S_{x}M)}^2 \\
& = \sum_{i=1}^d \left\| P_m E^m_{\xi}\left(\pi_2^* f. \langle \xi_{\mathbb{V}}, \nabla_{\V} \VV_i\rangle|_{\Ss^{n-2}_{\xi}}\right) \right\|_{L^2(S_{x}M)}^2.
\end{split}
\end{equation}
As $\ker ({\pi_m}_*) \subset \ker (\pi_{\ker \imath_{\xi^\sharp}}{\pi_m}_*)$, to obtain a contradiction it is sufficient to show that the $L^2$-orthogonal projection onto $\ker({\pi_m}_*)$ of the function $E^m_{\xi}\left( \pi_2^* f. \langle \xi_{\mathbb{V}}, \nabla_{\V} \VV_1\rangle|_{\Ss^{n-2}_{\xi}} \right)$ is non-zero, that is, it suffices to show the following:

\begin{lemma}
There exists $x \in M, \xi \in T^*_{x}M \setminus \left\{ 0 \right\}$ and $f \in \otimes^2_S T^*_{x}M$ such that:
\begin{equation}\label{eq:degree-final}
\mathrm{deg}\left(E^m_{\xi}\left(\pi_2^* f. \langle \xi_{\mathbb{V}}, \nabla_{\V} \VV_1\rangle|_{\Ss^{n-2}_{\xi}} \right)\right) \geq m+1.
\end{equation}
\end{lemma}

\begin{proof}
Similarly as in the proof of Lemma \ref{lemma:step2}, we know that $\VV_1$ has degree $\geq m + 1$. Indeed, by the mapping properties of $X$ and using equation \eqref{eq:Xv_i}, parity of $\VV_1$ is opposite to $m$ and so if $\deg(\VV_1) \leq m - 1$ we could write $\VV_1 = \pi_{m - 1}^*\widetilde{\VV}_1$ for some smooth $\widetilde{\VV}_1$. This would imply (using \eqref{equation:x}) that $D\widetilde{\VV}_1 = \Delta_0 u_1$ and hence also $\pi_{\ker D^*} \Delta_0 u_1 = 0$, so $u_1 = 0$ which is a contradiction.

Observe that by Lemma \ref{lemma:differentiated-restriction}, since $\VV_1$ has degree $\geq m+1$ at some point $x \in M$, there exists $\xi \in T^*_{x}M$ such that $\langle \xi_{\V}, \nabla_{\V} \VV_1 \rangle|_{\Ss^{n-2}_{\xi}}$ has degree $\geq m$. Then, by Lemma \ref{lemma:multiplication-surjective} (observe here that the musical map is an isomorphism from $\ker \langle{\xi, \bullet}\rangle \subset T_xM$ to $\ker \imath_{\xi^\sharp} \subset T_x^*M$), there exists $f \in \otimes^2_S (\ker \imath_{\xi^\sharp})$ such that $\pi_2^*f. \langle \xi_{\V}, \nabla_{\V} \VV_1 \rangle|_{\Ss^{n-2}_{\xi}}$ has degree $\geq m+2$. Note that $f$ can be naturally extended as a symmetric tensor in $\otimes^2_S T^*_{x}M$ by setting $f(\xi^\sharp, \bullet) = f(\bullet,\xi^\sharp) = 0$. Thus we can apply Lemma \ref{lemma:extension} to obtain \eqref{eq:degree-final}.
\end{proof}

This allows to complete the proof of Theorem \ref{theorem:genericity-metric}.

\begin{proof}[Proof of Theorem \ref{theorem:genericity-metric}]
The same proof as for Theorem \ref{theorem:genericity-connection} applies \emph{verbatim}.
\end{proof}

\begin{remark}\rm
	When $n = 2$, similarly as in Remark \ref{remark:n=2-twisted} our proof does not work. More precisely, the equality \eqref{eq:symbolblabla} always holds, which shows that the symbol of $R_{\VV_i} - 2Q_{\VV_i}^*\Pi_m^{-1}Q_{\VV_i}$ appearing in \eqref{equation:summary-2} is zero, hence this operator is of order $0$, as opposed to the case $n \geq 3$ where we show it is strictly of order $1$.
\end{remark}

\section{Manifolds with boundary}\label{sec:boundary}

In this section we outline some applications of the previous results to the injectivity of $X$-ray transform on manifolds with boundary.

\subsection{Generic injectivity on manifolds with boundary} We now turn to the case of a smooth Riemannian manifold $(M,g)$ with boundary. We define the incoming (resp. outgoing) tail $\Gamma_-$ (resp. $\Gamma_+$) as the set of points $(x,v) \in SM$ such that $\varphi_t(x,v)$ is defined for all $t \geq 0$ (resp. $t \leq 0$); the trapped set is defined as $K = \Gamma_- \cap \Gamma_+$. We further assume that $(M,g)$ satisfies the following assumptions:

\begin{itemize}
\item \textbf{Strictly convex boundary:} the boundary is strictly convex in the sense that the second fundamental form is strictly positive;
\item \textbf{No conjugate points:} the metric has no conjugate points;
\item \textbf{Hyperbolic trapped set:} there exists a continuous flow-invariant Anosov decomposition as in \eqref{equation:anosov} on the trapped set $K$.
\end{itemize}
We will use the short notation \emph{SNH} for such manifolds. Typical and well-studied examples are provided by \emph{simple manifolds} which are diffeomorphic to balls, without conjugate points and no trapped set; SNH manifolds are a generalization of these, see \cite{Paternain-lecture-notes,Guillarmou-17-2} for further references.

Embedding of SNH manifolds into closed Anosov manifolds was recently established in \cite{Chen-Erchenko-Gogolev-20} under the extra assumption that the manifold has boundary components diffeomorphic to spheres $\mathbb{S}^{n-1}$, or $\mathbb{S}^1 \times \mathbb{S}^{n-2}$. As we will rely on \cite{Chen-Erchenko-Gogolev-20}, we therefore introduce the following terminology: we say that a smooth $n$-dimensional manifold $M$ with boundary is \textbf{admissible} if it has boundary components diffeomorphic to $\mathbb{S}^{n-1}$ or $\mathbb{S}^{1} \times \mathbb{S}^{n-2}$. As pointed out to us by on the referees, in dimension $n \geq 3$, SNH manifolds with spherical boundary components are presumably diffeomorphic to balls.

Given a manifold $M$, we let $\mc{M}_{\mathrm{SNH}}$ be the set of all smooth SNH metrics. As in the closed case, this set is invariant by the action of a gauge group $\mathrm{Diff}_0(M)$ which is the set of diffeomorphisms fixing the boundary $\partial M$.

Given $x \in \partial M$, we let $\nu(x)$ be the outward-pointing normal unit vector to the boundary and
\[
\partial_{\pm}SM := \left\{ (x,v) \in SM, ~ x \in \partial M, \pm g(v,\nu(x)) \geq 0 \right\} 
\]
be the incoming (-) and outgoing (+) boundary. The X-ray transform is defined as the operator
\[
I^g : C^\infty(SM) \rightarrow C^\infty(\partial_-SM \setminus \Gamma_-), ~~~ I^gf(x,v) := \int_0^{\ell_+(x,v)} f(\varphi_t(x,v)) \dd t,
\]
where $\ell_+(x,v)$ denotes the \emph{exit time} of $(x,v) \in SM$, namely the maximal positive time for which the geodesic flow is defined. As in the closed case, the X-ray transform of symmetric tensors is defined as $I^g_m = I^g \circ \pi_m^*$; a similar decomposition $f=Dp+h$ between potential and solenoidal parts holds by requiring the extra condition $p|_{\partial M} = 0$. It is then easy to check that such potential tensors are in the kernel of $I^g_m$ and it is conjectured that this should be the whole kernel of the X-ray transform. The s-injectivity is known to be true in a lot of cases but not in full generality:
\begin{itemize}
\item On simple manifolds: s-injectivity was proved for $m=0,1$ in any dimension \cite{Anikonov-Romanov-97}; further assuming non-positive sectional curvature, it was obtained for all $m \in \Z_{\geq 0}$ in \cite{Pestov-Sharafutdinov-87}; and for all $m \in \Z_{\geq 0}$ on surfaces, without any curvature assumption, it was obtained in \cite{Paternain-Salo-Uhlmann-13}; generic s-injectivity was obtained for $m=2$ in \cite{Stefanov-Uhlmann-05, Stefanov-Uhlmann-08} (by proving s-injectivity for real analytic metrics);
\item On SNH manifolds: in dimension $n \geq 3$, s-injectivity was proved on all SNH manifolds for $m=0,1$ and all SNH manifolds with non-positive sectional curvature for $m \geq 2$ in \cite{Guillarmou-17-2}; it was obtained on all SNH surfaces for all $m \in \Z_{\geq 0}$ in \cite{Lefeuvre-19-1},
\item On manifolds admitting a global foliation by convex hypersurfaces: s-injectivity was obtained in any dimension $\geq 3$ for all $m \in \Z_{\geq 0}$ in \cite{Uhlmann-Vasy-16, Stefanov-Uhlmann-Vasy-18, DeHoop-Uhlmann-Zhai-18}.
\end{itemize}
We will then derive the following:

\begin{corollary}[of Theorem \ref{theorem:genericity-metric} and \cite{Chen-Erchenko-Gogolev-20}]
\label{corollary:genericity-metric-boundary}
There exists an integer $k_0 \gg 1$ such that the following holds. Let $M$ be a smooth admissible manifold of dimension $\geq 3$ carrying SNH metrics. For all $m \in \Z_{\geq 0}$\footnote{For $m=0,1$, the s-injectivity is already established \cite{Guillarmou-17-2}.}, there exists an open dense set $\mc{R}'_m \subset \M_{\mathrm{SNH}}$ (for the $C^{k_0}$-topology) such that for all metrics $g \in \mc{R}'_m$, the X-ray transform $I^g_m$ is s-injective. In particular, the space of metrics $\mc{R}' := \cap_{m \geq 0} \mc{R}'_m$ whose X-ray transforms are all s-injective is \emph{residual} in $\M_{\mathrm{SNH}}$.

\end{corollary}

Once again, the sets $\mc{R}_m$ and $\mc{R}$ are invariant by the action of $\mathrm{Diff}_0(M)$. We believe that the assumption that the boundary components are diffeomorphic to $\mathbb{S}^{n-1}$ or $\mathbb{S}^1 \times \mathbb{S}^{n-2}$ could be removed (for that, one would need to avoid the use of \cite[Theorem 1]{Chen-Erchenko-Gogolev-20} and follow directly the proof of Theorem \ref{theorem:genericity-metric} in the case of a manifold with boundary).

\subsection{Marked boundary distance function}

Let $(M,g)$ be an SNH manifold with boundary. For each pair of points $x,y \in \partial M$ and each homotopy class $[\gamma]$ of curves joining $x$ to $y$, it is well-known that there exists a unique geodesic in that class. We let $d_g(x,y,[\gamma])$ be the length of that unique geodesic (it minimizes the length among all curves in the $[\gamma]$) and call the map $d_g$ the \emph{marked boundary distance function}. When $M$ is simple (it is diffeomorphic to a ball), there is only a single geodesic joining $x$ to $y$; we may then drop the $[\gamma]$ and we call $d_g$ the \emph{boundary distance function}. This function is invariant by the action of the gauge-group $\mathrm{Diff}_0(M)$ (it descends on the moduli space) and it is conjectured that this is the only obstruction to recovering the metric:

\begin{conj}
\label{conjecture:michel}
Simple manifolds are \emph{boundary distance rigid} and, more generally, SNH manifolds are \emph{marked boundary distance rigid} in the sense that the marked boundary distance function:
\begin{equation}
\label{equation:distance}
d:  \mathbb{M}_{\mathrm{SNH}}:= \mc{M}_{\mathrm{SNH}}/\mathrm{Diff}_0(M) \ni \mathfrak{g} \mapsto d_{\mathfrak{g}}
\end{equation}
is injective.
\end{conj}

This conjecture is known in a certain number of cases but it still open in full generality, and was originally stated for simple manifolds by Michel \cite{Michel-81}. We refer to \cite{Mukhometov-77, Mukhometov-81, Mukhometov-Romanov-78,Croke-91,Michel-81,Gromov-83,Otal-90-2,Pestov-Uhlmann-05,Burago-Ivanov-10,Stefanov-Uhlmann-Vasy-18} for further details. Similarly to the closed case, it was shown in \cite{Stefanov-Uhlmann-05} that the local boundary distance rigidity of a metric $g$ can be derived from the s-injectivity of its X-ray transform $I_2^g$. This was extended to SNH manifolds (not necessarily spherical) in \cite{Lefeuvre-19-2}. As a consequence, we obtain:

\begin{corollary}[of Corollary \ref{corollary:genericity-metric-boundary} and \cite{Lefeuvre-19-2}]
There exists $k_0 \gg 1$ such that the following holds. Let $M$ be a smooth admissible $n$-dimensional manifold carrying SNH metrics. There is an open dense set $\mathbbm{R}'_2 \subset \mathbb{M}_{\mathrm{SNH}}$ (with respect to the $C^{k_0}$-topology) such that: for all $\mathfrak{g}_0 \in \mathbbm{R}'_2$, the marked boundary distance function $d$ in \eqref{equation:distance} is locally injective near $\mathfrak{g}_0$.
\end{corollary}

Here $\mathbbm{R}'_2 = \mc{R}'_2/\mathrm{Diff}_0(M)$, where $\mc{R}'_2$ is given by Corollary \ref{corollary:genericity-metric-boundary}.

\subsection{Proofs} This section is devoted to the proof of Corollary \ref{corollary:genericity-metric-boundary}.

\begin{proof}[Proof of Corollary \ref{corollary:genericity-metric-boundary}] We fix an integer $m \in \Z_{\geq 0}$. As in the closed case, the s-injectivity of $I^g_m$ is equivalent to the s-injectivity of the \emph{normal operator} $\left(I^g_m\right)^*I^g_m$ which has the same microlocal properties as the generalized X-ray transform $\Pi^g_m$, see \cite{Guillarmou-17-2} for instance. Hence the fact that $\mc{R}'_m$ is open is an immediate consequence of elliptic theory so it suffices to show that $\mc{R}'_m$ is dense in $\M^{k_0}_{\mathrm{SNH}}$. We let $g \in \M^{k_0}_{\mathrm{SNH}}$. By \cite{Chen-Erchenko-Gogolev-20}, the manifold $(M,g)$ can be embedded into a closed Anosov manifold $(M^{\mathrm{ext}}, g_{\mathrm{ext}})$ such that $g_{\mathrm{ext}}|_{M} = g$ and by Theorem \ref{theorem:genericity-metric}, we can perturb the metric $g_{\mathrm{ext}}$ (in the $C^{k_0}$-topology) to a new metric $g'_{\mathrm{ext}}$ such that this metric has injective X-ray transform $I^{g'_{\mathrm{ext}}}_m$ on $M^{\mathrm{ext}}$. In order to prove Corollary \ref{corollary:genericity-metric-boundary}, it then suffices to show that $g'_{\mathrm{ext}}|_{M}$ restricted to the manifold with boundary $M$ has injective X-ray transform. In other words, it suffices to prove the following:

\begin{lemma}
Let $(M,g)$ be a closed Anosov manifold and let $(N,g|_{N}) \subset (M,g)$ be an SNH manifold with boundary. If the X-ray transform $I^{g}_m$ on the closed manifold is s-injective on $M$, then the X-ray transform $I^{g|_{N}}_m$ on the manifold with boundary is also s-injective on $N$.
\end{lemma}

\begin{proof}
We let $f \in C^\infty(N,\otimes^m_S T^*N)$ such that $I^{g|_{N}}_mf = 0$. First of all, by \cite[Lemma 2.2]{Sharafutdinov-02}, we can write $f=Dp+h$, where $p|_{\partial N}=0, \imath_\nu h = 0$ in a \emph{neighborhood of $\partial N$} (where $\nu$ is the outward-pointing unit vector and is extended in the inner neighborhood of $\partial N$ by flowing along the geodesics), $\imath_\nu$ is the contraction by the vector $\nu$ and both tensors $p,h$ are smooth. Observe that $I^{g|_{N}}f = 0 = I^{g|_{N}}h$. We claim that $\partial_\nu^k h = 0$ for all $k \geq 0$, that is $h$ vanishes to infinite order on the boundary. For $k=0$, this is contained in \cite[Lemma 2.3]{Sharafutdinov-02}. The proof is a simple observation: if $h_x(v,...,v)$ is non-zero for some $x \in \partial M$ and $v \in T_x\partial N$ then it is also true in a small neighboorhood of $(x,v)$ and $\pi_{m}^*h$ has constant sign there; without loss of generality we can take it to be positive. Using short geodesics in a neighborhood of the boundary (with unit speed vector almost equal to $v$) we then get that $I_mh(x,v) > 0$, which is a contradiction. Hence we can write $h = rh'$, where $r(x) := d(x,\partial N)$ is defined locally near the boundary and extended to an arbitrary positive function inside $N$. Then the same argument of positivity applies to $h'$ and by iteration, we get that $h = \mc{O}(r^\infty)$ at $\partial N$. Hence we can extend $h$ by $0$ outside $N$ to get a smooth tensor (still denoted by $h$) in $C^\infty(M,\otimes^{m}_S T^*M)$.

We now claim that $I_m^g h = 0$ on $M$, that is the integral of $\pi_m^*h$ along closed geodesics in $M$ is zero. Indeed, let $\gamma \subset SM$ be a closed orbit of the geodesic flow of length $\ell(\gamma)$, then:
\[
I_m^gh(\gamma) = \dfrac{1}{\ell(\gamma)} \int_0^{\ell(\gamma)} \pi_m^*h(\varphi_t(x,v)) \dd t =  \dfrac{1}{\ell(\gamma)} \left(\int_I \pi_m^*h + \int_J \pi_m^*h\right),
\]
where $I \subset [0,\ell(\gamma)]$ is the union of intervals of times $t$ such that $\pi(\varphi_t(x,v)) \notin N$, $\pi : SM \rightarrow M$ denotes the projection and $J$ is the complement of $I$. Observe that the integral over $I$ is zero since $h$ was extended by $0$ outside $N$. Now, $J$ splits as a union of subintervals, each of them corresponding to a segment of geodesic in $N$. By assumption, the integral of $\pi_m^*h$ over all these segments is $0$. Hence $I_m^g h = 0$.

Since $I_m^g$ is s-injective, we deduce that $h = Du$, for some tensor $u \in C^\infty(M,\otimes^{m-1} T^*M)$, that is $\pi_m^*h = X \pi_{m-1}^*u$ is a coboundary. We now want to show that $u|_{\partial M} = 0$. We let $(x_0,v_0) \in SM$ be a point on the boundary $\partial_+ SN$ such that both forward $(\varphi_t(x_0,v_0))_{t \in \R_{\geq 0}}$ and backward $(\varphi_t(x_0,v_0))_{t \in \R_{\leq 0}}$ orbits are dense in $SM$ and we let $c := \pi_{m-1}^*u(x_0,v_0)$. Observe that the following holds: if $t^-_1 \geq 0$ denotes the first positive time such that $\varphi_{t^-_1}(x_0,v_0)$ intersects $\partial_-SN^{\circ}$ (that is, it is an inward pointing vector that is not tangent to the boundary of $N$), then $ \pi_{m-1}^*u(\varphi_t(x_0,v_0))$ is constant equal to $c$ for all $t \in [0,t^-_1]$. There is then a time $t^+_1 > t^-_1$ such that $\varphi_{t_1^+}(x_0,v_0) \in \partial_+ SN$. For $t \in [t^-_1,t^+_1)$, the value of $ \pi_{m-1}^*u(\varphi_t(x_0,v_0))$ is unknown but one has $\pi_{m-1}^*u(\varphi_{t^+_1}(x_0,v_0)) = c$ since:
\[
\begin{split}
I^{g|_{N}}h(\varphi_{t^-_1}(x_0,v_0)) &= \int_{t^-_1}^{t^+_1} \pi_m^*h(\varphi_t(x_0,v_0)) \dd t \\
&= 0 =  \pi_{m-1}^*u(\varphi_{t^+_1}(x_0,v_0)) -  \pi_{m-1}^*u(\varphi_{t^-_1}(x_0,v_0)).
\end{split}
\]
Since the orbit of $\mc{O}(x_0,v_0)$ of $(x_0,v_0)$ is dense in $SM$, the set $\mc{A} := \mc{O}(x_0,v_0) \cap (\partial_- SN \cup \partial_+ SN)$ is also dense in $\partial_- SN \cup \partial_+ SN$. Moreover, iterating the previous argument shows that $\pi_{m-1}^*u|_{\mc{A}} = c$ and thus $\pi_{m-1}^*u|_{\partial_- SN \cup \partial_+ SN} = c$. If $m$ is even, $m-1$ is odd and this forces $c$ to be $0$ (just use the antipodal map $(x,v) \mapsto (x,-v)$); if $m$ is odd, then changing at the very beginning $u$ by $u+ \lambda g^{\otimes (m-1)/2}$ for some $\lambda \in \R$ allows to take $c=0$. Hence, we conclude that $u|_{\partial N} = 0$. This gives that $f = Dp + h = D(p+u)$, where $p+u$ vanishes on $\partial N$.
\end{proof}

This concludes the proof of Corollary \ref{corollary:genericity-metric-boundary}.
\end{proof}

\bibliographystyle{alpha}
\bibliography{Biblio}

\end{document}